\documentclass{amsart}

\usepackage{amsmath}
\usepackage{amsfonts}
\usepackage{amssymb}
\usepackage{amsthm}

\usepackage{array}

\usepackage{tikz}
\usetikzlibrary{tikzmark,decorations.pathreplacing}

\usepackage[all]{xypic}

\usepackage{pst-node} 
\usepackage{pstricks}
\usepackage{pstricks-add}

\usepackage{xcolor}

\usepackage{genyoungtabtikz}

\usepackage[colorlinks=true,linkcolor=magenta,citecolor=blue]{hyperref}

\newcommand{\sfmod}{{\text{-}\mathsf{mod}}}
\newcommand{\sfMod}{{\text{-}\mathsf{Mod}}}
\newcommand{\sfProj}{{\text{-}\mathsf{Proj}}}
\newcommand{\sfproj}{{\text{-}\mathsf{proj}}}
\newcommand{\sfInj}{{\text{-}\mathsf{Inj}}}
\newcommand{\sfstab}{{\text{-}\mathsf{stab}}}
\newcommand{\sfperf}{{\text{-}\mathsf{perf}}}
\newcommand{\sfperm}{{\text{-}\mathsf{perm}}}

\newcommand{\sfD}{\mathsf{D}}
\newcommand{\sfHo}{\mathsf{Ho}}
\newcommand{\sfC}{\mathsf{C}}

\newcommand{\bfX}{\mathbb{X}}
\newcommand{\bfY}{\mathbb{Y}}
\newcommand{\bfT}{\mathbb{T}}
\newcommand{\bfB}{\mathbb{B}}
\newcommand{\bfG}{\mathbb{G}}
\newcommand{\bfU}{\mathbb{U}}
\newcommand{\bfV}{\mathbb{V}}
\newcommand{\bfL}{\mathbb{L}}
\newcommand{\bfP}{\mathbb{P}}

\newcommand{\calO}{\mathcal{O}}
\newcommand{\calA}{\mathcal{A}}
\newcommand{\calT}{\mathcal{T}}

\newcommand{\simto}{\mathop{\longrightarrow}\limits^\sim}
\newcommand{\testleftlong}{\longleftarrow\!\shortmid}

\def\trait{-\hskip-2mm - \hskip-2mm -}

\theoremstyle{plain}
\newtheorem{theorem}{Theorem}[section]
\newtheorem{lem}[theorem]{Lemma}

\newtheorem{prop}[theorem]{Proposition}
\newtheorem{conj}[theorem]{Conjecture}
\newtheorem{cor}[theorem]{Corollary}
\numberwithin{equation}{section}

\theoremstyle{definition}
\newtheorem{definition}[theorem]{Definition}

\theoremstyle{remark}
\newtheorem{rmk}[theorem]{Remark}
\newtheorem{example}[theorem]{Example}

\newtheorem{exercise}[theorem]{Exercise}

\author{Olivier Dudas}
\date{CIB, July 2016}
\title{Lectures on modular Deligne--Lusztig theory}

\begin{document}

\maketitle

\begin{abstract}
These notes are based on a series of lectures given by the author at the Centre Bernoulli (EPFL) in July 2016. They aim at illustrating the importance of the mod-$\ell$ cohomology of Deligne--Lusztig varieties in the modular representation theory of finite reductive groups. 
\end{abstract}

\section*{Introduction}
In order to construct and study the complex representations of finite reductive groups $G(q)$ (such as $\mathrm{GL}_n(q)$, $\mathrm{Sp}_{2n}(q)$,\ldots) Deligne and Lusztig introduced in 1976 a family of algebraic varieties acted on by $G(q)$ \cite{DelLus}. The subsequent work of Lusztig on the cohomology of these \emph{Deligne--Lusztig varieties} led to a complete classification of the irreducible characters of finite reductive groups \cite{Lus84}. 

\smallskip

The purpose of these lectures is to present a generalization of the theory of Deligne--Lusztig to the modular setting, that is, for representations over fields of positive characteristic. This originated in the work of Brou\'e \cite{Bro89} and Bonnaf\'e--Rouquier \cite{BonRou03}. 

\smallskip

In the ordinary case (in characteristic zero), the representation theory is controlled by the simple objects, which are in turn determined by a numerical datum, their characters. The situation is far more complicated for representations in positive characteristic; several classes of indecomposable objects are of particular interest, and more information is needed to understand the representations, namely:
\begin{itemize}
 \item information of numerical nature: characters of projective modules, multiplicities of simple modules in a given ordinary character, all of which are encoded in the so-called decomposition matrix;
 \item information of homological nature: extensions between simple modules, Loewy series of projective modules, projective resolutions of simple objects.
\end{itemize}
The alternating sum of the cohomology groups of Deligne-Lusztig varieties produces a virtual character \--- an element of the Grothendieck group of the category of representations. In the modular framework, this object does not contain enough information, and one should consider each individual cohomology group, or rather the cohomology complex of the variety. This object now lives in the derived category of representations, and it encodes many aspects of the modular representation theory of the group. One crucial incarnation of this phenomenon is the geometric version of Brou\'e's abelian defect group conjecture, which predicts that the cohomology complex of a suitably chosen Deligne--Lusztig variety induces a derived equivalence between the principal block of a finite reductive group and its Brauer correspondent. Not so many cases of this conjecture are known to hold but a lot of numerical evidence and many partial results have been obtained in that direction.

\smallskip

The first part of these lectures aims at introducing the mod-$\ell$ cohomology of Deligne--Lusztig varieties using the modern language of derived and homotopy categories. Unlike most of the textbooks on \'etale and $\ell$-adic cohomology, we avoid the definition and focus on the properties of the cohomology complexes of varieties acted on by a finite group (such as perfectness), with particular attention on how one can compute such complexes (using decompositions, quotients or fixed points). In the second part we present several recent results obtained using this approach. They include the computation of decomposition numbers in \S\ref{chap:dec} (a joint work with G. Malle) and the determination of Brauer trees in \S\ref{chap:brauer} (a joint work with D. Craven and R. Rouquier). This illustrates how powerful the geometric methods are for solving representation theoretic problems for finite reductive groups. There is a converse to that statement, and we explain in a final chapter how to use representation theory to show that the cohomology of a particular Deligne--Lusztig variety is torsion-free.

\tableofcontents

\section{Introduction to derived categories}

Throughout this chapter, $A$ will denote a ring with unit. The category of left $A$-modules
(resp. finitely generated left $A$-modules) will be denoted by $A\sfMod$ (resp. $A\sfmod$). 

\smallskip

The purpose of this first chapter is to introduce two categories, the homotopy category $\sfHo(A\sfMod)$
and the derived category $\sfD(A\sfMod)$. Here is a non-exhaustive list of reasons why we are
going to work in this framework, instead of working with
$A$-modules or complexes of $A$-modules:
\begin{itemize}
 \item to get rid of (split) exact sequences;
 \item to have uniqueness of projective or injective resolutions;
 \item to have a good notion of duality (e.g. over $\mathbb{Z}$);
 \item to work with non-exact functors.
\end{itemize}

Several steps are needed to understand the construction of the homotopy and derived categories of
$A$-modules:
\begin{equation*} A\sfMod \rightsquigarrow \underbrace{\sfC(A\sfMod)}_{\begin{subarray}{c}\text{complexes of}\\ 
\text{$A$-modules} \end{subarray}} \tikzmark{eq-homotopy}\rightsquigarrow  \sfHo(A\sfMod) 
\tikzmark{eq-derived}\rightsquigarrow \sfD(A\sfMod)\end{equation*}

\begin{tikzpicture}[remember picture,overlay]
\draw[->,>=latex]
  ([shift={(8pt,-5pt)}]pic cs:eq-homotopy) |- 
  ++(10pt,-40pt) 
  node[right,text width=3cm] 
    {\footnotesize getting rid of \emph{split}\\[-4pt] exact sequences
    };
\draw[->,>=latex]
  ([shift={(8pt,-5pt)}]pic cs:eq-derived) |- 
  ++(10pt,-15pt) 
  node[right,text width=3cm] 
    {\footnotesize inverting \\[-4pt] quasi-isomorphisms
    };
\end{tikzpicture}

\bigskip\bigskip

\noindent Note however that this chapter is not intended to provide a detailed account on this construction. For further reading on the subject we recommend for example the excellent textbooks by Gelfand--Manin \cite{GelMan} and Neeman \cite{Nee}.

\subsection{Complexes of $A$-modules}
A \emph{complex of $A$-modules} is 
$$C_\bullet = (\cdots \longrightarrow C_n \overset{d_{n}}{\longrightarrow} C_{n+1} \overset{d_{n+1}}{\longrightarrow}  C_{n+2} \longrightarrow \cdots) $$
where, for each $n$, $C_n$ is a left $A$-module and $d_n : C_n \longrightarrow C_{n+1}$ is a morphism of $A$-modules (the \emph{differential}) satisfying $d_{n+1} \circ d_n = 0$. 

\smallskip

A morphism between two complexes $f = (C_\bullet,d) \longrightarrow (D_\bullet,\partial)$ is given by a family of morphisms of $A$-modules $f_n : C_n \longrightarrow D_n$ making the following diagram commute
 $$\xymatrix@C=2em@R=3em{\cdots \ar[r] & C_n \ar[r]^{d_n}\ar[d]^{f_n} & C_{n+1} \ar[r]^{d_{n+1}} \ar[d]^{f_{n+1}}  & C_{n+2} \ar[r]\ar[d]^{f_{n+2}} &  \cdots \\
\cdots \ar[r] & D_n \ar[r]^{\partial_{n}} & D_{n+1} \ar[r]^{\partial_{n+1}}  & D_{n+2} \ar[r] & \cdots  }$$
Here is an example of a morphism between two complexes of $\mathbb{Z}$-modules (we will see later that it is a quasi-isomorphism)
$$\xymatrix@C=2em@R=2em{\cdots \ar[r] & 0 \ar[r]\ar[d] & \mathbb{Z} \ar[r]^m \ar[d]  & \mathbb{Z} \ar[r]\ar@{->>}[d] & 0 \ar[r]\ar[d] & \cdots \\
\cdots \ar[r] & 0 \ar[r] & 0 \ar[r]  & \mathbb{Z}/m\mathbb{Z} \ar[r] & 0 \ar[r]& \cdots  }$$

The category of complexes of $A$-modules will be denoted $\sfC(A\sfMod)$. It is an abelian category. This can be seen by considering the ring $A[X]/X^2 =: A(d) \simeq A \oplus Ad$ with $d^2 = 0$. Then the functor
$$ \begin{array}{rcl} \sfC(A\sfMod) & \overset{\sim}{\longrightarrow} & A(d)\sfMod \\[5pt] 
												C_\bullet & \longmapsto & \displaystyle\bigoplus_{n \in \mathbb{Z}} C_n \text{ with } d_{|C_n} = d_n \end{array}$$
is an equivalence of categories. The abelian structure on $\sfC(A\sfMod)$ is obtained via the equivalence from the abelian structure of $A(d)\sfMod$.  As a consequence, we can consider kernels and cokernels of morphisms between complexes, as well as exact sequences of complexes. 
 
 \smallskip

We say that a complex $C_\bullet$ is \emph{bounded above} (resp. \emph{bounded below}, resp. \emph{bounded}) if $C_n = 0$ for $n \gg 0$ (resp. $n \ll 0$, resp. $|n|\gg0$). The corresponding full subcategory of $\sfC(A\sfMod)$ will be denoted by $C^-(A\sfMod)$ (resp.  $C^+(A\sfMod)$, resp. $C^b(A\sfMod)$). 

\smallskip

Given $k \in \mathbb{Z}$ and $C_\bullet$ a complex of $A$-modules, we define the \emph{$k$-th shift of $C_\bullet$}, denoted by $C_\bullet[k]$, to be the complex with terms $(C_\bullet[k])_n = C_{n+k}$ and differential $d_{C_\bullet[k]} = (-1)^k d_{C_\bullet}$. If $M$ is an $A$-module, the notation $M[k]$ stands for the complex with zero terms outside the degree $-k$ and $M$ in the degree $-k$. The functor
$$ \begin{array}{rcl} A\sfMod & \longrightarrow & \sfC(A\sfMod) \\[5pt] 
												M & \longmapsto & M[0]
\end{array}$$
is fully faithful. In other words, $A\sfMod$ can be identified with complexes with zero terms outside the degree $0$.

\smallskip

Since $d_{n} \circ d_{n-1} = 0$, $\mathrm{Im}\,{d_{n-1}}$ is a submodule of $\mathrm{Ker}\, d_n$. The quotient 
$H^n(C_\bullet) = \mathrm{Ker}\, d_n / \mathrm{Im}\,{d_{n-1}}$ is an $A$-module, called the \emph{degree $n$ cohomology group} of $C_\bullet$.  We write
$$ H^\bullet(C_\bullet) := \bigoplus_{n \in \mathbb{Z}} H^n(C_\bullet) = \mathrm{Ker}\, d / \mathrm{Im}\,d.$$
We say that a complex $C_\bullet$ is \emph{exact} or \emph{acyclic} if $H^n(C_\bullet) = 0$ for all $n \in \mathbb{Z}$.
The compatibility of maps between complexes and the respective differentials ensures that any morphism of complexes
$f = (C_\bullet,d) \longrightarrow (D_\bullet,\partial)$ induces a family of morphisms of $A$-modules 
$H^n(f) : H^n(C_\bullet) \longrightarrow H^n(D_\bullet)$. 

\smallskip

From now on, we will omit the subscript ${}_\bullet$ in the notation of complexes, as well as the reference to the differentials for morphisms of complexes. 

\begin{prop} Let $0 \longrightarrow C \overset{\iota}{\longrightarrow} C' \overset{\pi}{\longrightarrow} C'' \longrightarrow 0$ be a short exact
sequence of complexes of $A$-modules. Then there are boundary maps $\delta_n : 
H^n(C'') \longrightarrow H^{n+1}(C)$ for all $n \in \mathbb{Z}$ yielding a long exact sequence of $A$-modules
$$ \cdots \longrightarrow  H^n(C) \overset{H^n(\iota)}{\longrightarrow} H^n(C') \overset{H^n(\pi)}{\longrightarrow} H^n(C'') \overset{\delta_n}{\longrightarrow} H^{n+1}(C) \longrightarrow \cdots $$
\end{prop}

\begin{proof}[Sketch of proof]
Let $c'' \in \mathrm{Ker}\,d_n''$, which we write $c'' = \pi_n(c')$ for some $c' \in C_n'$. 
Since $\pi$ is a morphism of complexes, $\pi_{n+1}(d_n'(c)) = d_n''(\pi_n(c')) = 0$, hence $d_n'(c') \in
\mathrm{Ker}\, \pi_{n+1} = \mathrm{Im}\, \iota_{n+1}$. Now write $d_n'(c') = \iota_{n+1}(c)$ and set
$\delta_n(c'') := c$. 
\end{proof}

\begin{exercise}
Check that $\delta_n$ is well-defined, and that it induces the long exact sequence stated in the proposition.
\end{exercise}

In terms of $A(d)$-modules, the proposition shows the existence of a morphism of $A(d)$-modules 
$\delta : H^\bullet(C'') \longrightarrow H^\bullet(C)[1]$ which fits in a triangle
$$\xymatrix@C=0.8em{& H^\bullet(C'') \ar[ld]_{[1]}^\delta & \\ H^\bullet(C) \ar[rr] & & H^\bullet(C') \ar[lu] } $$

\subsection{The homotopy category}
A morphism $f : (C,d) \longrightarrow (D,\partial)$ between complexes of $A$-modules is said to be \emph{null-homotopic} if
it is of the form $f = s\circ d + \partial \circ s$ for some map $s : C \longrightarrow D[-1]$ (not necessarily a morphism
of complexes). We illustrate this with the following diagram:
$$\xymatrix@C=3em@R=3em{\cdots \ar[r] & C_n \ar[r]^{d_n}\ar[d]_{+} \ar[ld] & C_{n+1} \ar[r]^{d_{n+1}} \ar[d]_{+}
\ar[ld]_{s_{n}}  & C_{n+2} \ar[r]\ar[d]_+ \ar[ld]_{s_{n+1}} &  \cdots \ar[ld]_{s_{n+2}} \\
\cdots \ar[r] & D_n \ar[r]^{\partial_{n}} & D_{n+1} \ar[r]^{\partial_{n+1}}  & D_{n+2} \ar[r] & \cdots  }$$   
Each vertical map $f_n : C_n \longrightarrow D_n$ satisfies $f_n = s_n \circ d_n + \partial_{n-1} \circ s_{n-1}.$
 
\smallskip

Given two morphisms of complexes $f,f' : C \longrightarrow D$ we write $f \sim f'$ if $f-f'$ is null-homotopic. 
This is an equivalence relation, compatible with the sum and the composition of morphisms. 
A complex $C$ is \emph{null-homotopic} if the identity map $1_C$ is null-homotopic (\emph{i.e.} $1_C \sim 0$). 
We say that $f$ is a \emph{homotopy equivalence} if there exists a morphism $g : D \longrightarrow C$ such that
$f \circ g \sim 1_D$ and $g \circ f \sim 1_C$.

\begin{definition}The \emph{homotopy category of $A$-modules}, denoted by $\sfHo(A\sfMod)$, is the category with
 \begin{itemize}
  \item objects: complexes of $A$-modules (same as $\sfC(A\sfMod)$),
  \item morphisms: $\mathrm{Hom}_{\sfHo(A\sfMod)}(C,D) := \mathrm{Hom}_{\sfC(A\sfMod)}(C,D)/ \sim$.
 \end{itemize}
\end{definition}
It is an additive category (but non-abelian in general). The isomorphisms in the homotopy category
are exactly the classes of the homotopy equivalences.

\begin{exercise}\label{exo:split}
Let $0 \longrightarrow L \longrightarrow M \longrightarrow N \longrightarrow 0$ be a short
exact sequence of $A$-modules, and $C$ be the complex associated to this sequence, with $L$ in degree $0$.
Show that $C$ is null-homotopic if and only if the exact sequence splits.
\end{exercise} 

A complex of the form $(\cdots \longrightarrow 0 \longrightarrow M \overset{f}{\longrightarrow} M \longrightarrow 0 \longrightarrow \cdots)$ is null-homotopic if and only if $f$ is an isomorphism. More generally, a complex 
$C$ is null-homotopic if and only if it decomposes as a direct sum of complexes of the form
$$(\cdots \longrightarrow 0 \longrightarrow M \overset{\sim}{\longrightarrow} M \longrightarrow 0 \longrightarrow \cdots).$$ 
See also Exercise \ref{exo:split}. When working in $\sfHo(A\sfMod)$ we will often consider reduced complexes where all the null-homotopic direct summands are removed. 

\smallskip

When $f$ is null-homotopic, the corresponding morphism on cohomology groups is zero. As a consequence, if $f \sim g$ then $H^\bullet(f) = H^\bullet(g)$. Also, if $f$ is a homotopy equivalence then $H^\bullet(f)$ is an isomorphism. In particular, if $C \simeq D$ in $\sfHo(A\sfMod)$ then $H^\bullet(C) \simeq H^\bullet(D)$.

\smallskip

We mentioned earlier that projective and injective resolutions are unique in the homotopy and derived categories. Recall that a projective resolution $P$ of an $A$-module $M$ is a complex of projective $A$-modules
$$ \cdots \longrightarrow P_{-n} \longrightarrow P_{-n+1} \longrightarrow \cdots \longrightarrow P_{-1} \longrightarrow P_0 \longrightarrow 0 \longrightarrow \cdots$$
such that $H^\bullet(P) \simeq M[0]$. In other words, there is a surjective map $P_0 \twoheadrightarrow M = H^0(P)$ which fits in the following long exact sequence
$$ \cdots \longrightarrow P_{-n} \longrightarrow P_{-n+1} \longrightarrow \cdots \longrightarrow P_{-1} \longrightarrow P_0 \twoheadrightarrow M \longrightarrow 0 \longrightarrow \cdots.$$
The idea of projective resolutions is to replace $M$ by a complex with the same cohomology ($M$ in degree $0$) but whose terms are `nicer'. 

\begin{prop} If $P$ and $Q$ are two projective resolutions of $M$ then
$P \simeq Q$ in $\sfHo(A\sfMod)$.
\end{prop} 

\begin{proof}[Sketch of proof]
We only show how to construct the morphisms $f : P \longrightarrow  Q$ and $g : Q \longrightarrow  P$ which will be mutually inverse in the homotopy category.

\smallskip

Let us denote by $d$ (resp. $\partial$) the differential of the complex $P$ (resp. $Q$) and
by $d_0 : P_0 \twoheadrightarrow M$ (resp. $\partial_0 : Q_0 \twoheadrightarrow M$) the respective projections. 
Since $\partial_0$ is surjective and $P_0$ is projective, the map $d_0 : P_0 \longrightarrow M$ factors through $Q_0 \twoheadrightarrow M$. In other words, there exists $f_0 : P_0 \longrightarrow Q_0$ such that
$d_0 = \partial_0 \circ f_0$.

\smallskip

Since $\partial_0 f_0 d_{-1} = d_0 d_{-1} = 0$, we have $\mathrm{Im}\, (f_0 d_{-1}) \subset \mathrm{Ker}\, \partial_0 = \mathrm{Im}\, \partial_{-1}$. Therefore $f_0 d_{-1}$ can be seen as a map from $P_{-1}$ to $\mathrm{Im}\, \partial_{-1}$. Since $P_{-1}$ is projective, it should factor through the surjective map $\partial_{-1} : Q_{-1} \twoheadrightarrow \mathrm{Im}\, \partial_{-1}$. In other words, there exists $f_{-1} : P_{-1} \longrightarrow Q_{-1}$ such that
$\partial_{-1} f_{-1} = f_0 d_{-1}$. By iterating this construction, we obtain, for all $ n < 0$, maps $f_n : P_{n} \longrightarrow Q_{n}$ such that $\partial_{n} f_{n} = f_{n+1} d_{n}$. This means that $f = (f_n)_{n \in \mathbb{Z}} : P \longrightarrow Q$ is a morphism of complexes. The construction of $g$ is similar.
\end{proof}
 
\begin{exercise} Show that $g\circ f \sim 1_{P}$ and $f \circ g \sim 1_{Q}$.
\end{exercise} 

More generaly, if $C$ is a complex of $A$-modules, a \emph{projective resolution} $P$ of $C$ is a bounded above complex of projective modules together with a morphism $s : P \longrightarrow C$ such that $H^n(s)$ is an isomorphism for all $n \in\mathbb{Z}$ (a \emph{quasi-isomorphism}, see \S\ref{sec:derived}). We shall see in Proposition \ref{prop:morphismsind} that $P$ is uniquely determined by $C$ up to homotopy equivalence, which generalizes the previous proposition.

\begin{exercise}
Spell out the case where $C = M[0]$ has only one non-zero term, say $M$ in degree $0$.
\end{exercise}

\begin{definition} The \emph{mapping cone} of a morphism of complexes $f : (C,d) \longrightarrow (D,\partial)$ is the complex $\mathrm{Cone}(f) = C[1] \oplus D$ with differential
$$ d_{\mathrm{Cone}(f)} = \left[\begin{array}{cc} d[1] & 0 \\ f[1] & \partial \end{array}\right].$$
\end{definition} 
 
If $f = 0$, the mapping cone is just the direct sum of $C[1]$ and $D$ in the category of complexes. However, if $f$ is non-zero, it encodes more information. For example, if $C=M[0]$ and $D=N[0]$ are complexes concentrated in degree $0$, then $f$ is induced by a morphism of $A$-modules $f_0 : M \longrightarrow N$. In this case the complex $\mathrm{Cone}(f)$ has only two non-zero terms (in degree $-1$ and $0$) and its cohomology is $H^{-1}(\mathrm{Cone}(f)) = \mathrm{Ker}\, f $ and $H^{0}(\mathrm{Cone}(f)) = \mathrm{CoKer}\, f$. We deduce the existence of a long exact sequence
of $A$-modules
$$ 0 \longrightarrow \underbrace{H^{-1}(\mathrm{Cone}(f))}_{\mathrm{Ker}\, f} \longrightarrow M \overset{f}{\longrightarrow} N \longrightarrow \underbrace{H^{0}(\mathrm{Cone}(f))}_{\mathrm{CoKer}\, f} \longrightarrow 0.$$
This generalizes to morphisms of complexes as follows.

\begin{prop}\label{prop:coneseq}
Any morphism of complexes $f : C \longrightarrow D$ induces a long exact sequence
$$ \cdots \longrightarrow  H^n(C) \overset{}{\longrightarrow} H^n(D) \overset{}{\longrightarrow} H^n(\mathrm{Cone}(f)) \overset{}{\longrightarrow} H^{n+1}(C) \longrightarrow \cdots $$
\end{prop}

The maps $H^{n}(D) \longrightarrow H^n(\mathrm{Cone}(f))$ and $H^n(\mathrm{Cone}(f)) \longrightarrow H^{n+1}(C)$ are induced by the natural morphisms of complexes $\iota : D \longrightarrow \mathrm{Cone}(f)$ and $\pi : \mathrm{Cone}(f) \longrightarrow C[1]$. Note that $\pi \circ \iota = 0$. The maps $ f \circ \pi$ and $\iota \circ f$ are non-zero in general, and only null-homotopic. This ensures that we have a triangle 
$$\xymatrix@C=0.8em{& \mathrm{Cone}(f) \ar[ld]_{[1]}^\pi & \\ C \ar[rr]^f & & D\ar[lu]^\iota } $$
in the homotopy category $\sfHo(A\sfMod)$. 

\begin{exercise}
Show that $f \circ \pi$ and $\iota \circ f$ are null-homotopic.
\end{exercise}

More generally, a \emph{triangle} in $\sfHo(A\sfMod)$ is $C \overset{f}{\longrightarrow} D \overset{g}{\longrightarrow} E \overset{h}{\longrightarrow} C[1]$ such that $g\circ f$, $h \circ g$ and $f[1] \circ h$ are null-homotopic. We represent it as $C {\longrightarrow} D {\longrightarrow} E {\rightsquigarrow}$ or as a triangular diagram 
$$\xymatrix@C=0.8em{& E \ar[ld]_{[1]} & \\ C \ar[rr] & & D\ar[lu] } $$
A morphism of triangles between $C {\longrightarrow} D {\longrightarrow} E {\rightsquigarrow}$ and
$C' {\longrightarrow} D' {\longrightarrow} E' {\rightsquigarrow}$
is the data of morphisms of complexes $u : C \longrightarrow C'$, $v : D \longrightarrow D'$  and
$w : E \longrightarrow E'$ making the following diagram commute
$$ \xymatrix{C \ar[r]^f \ar[d]^u & D \ar[r]^g \ar[d]^v & E \ar[r]^h \ar[d]^w & C[1] \ar[d]^{u[1]} \\ 
C' \ar[r]^{f'}  & D' \ar[r]^{g'} & E' \ar[r]^{h'}  & C'[1] }$$
A triangle is \emph{distinguished} if it is isomorphic to a triangle
$C \overset{f}{\longrightarrow} D {\longrightarrow} \mathrm{Cone}(f) {\rightsquigarrow}$ for some morphism 
of complexes $f : C \longrightarrow D$. In particular, by Proposition \ref{prop:coneseq} any distinguished triangle $C {\longrightarrow} D {\longrightarrow} E {\rightsquigarrow}$ yields a long exact sequence in cohomology
$$ \cdots \longrightarrow  H^n(C) \overset{}{\longrightarrow} H^n(D) \overset{}{\longrightarrow} H^n(E) \overset{}{\longrightarrow} H^{n+1}(C) \longrightarrow \cdots $$
The category $\sfHo(A\sfMod)$ together with the suspension $[1]$ and the collection of distinguished triangles is a \emph{triangulated category} (see for example \cite[Chap. 1]{Nee} for the list of axioms of triangulated categories).

\smallskip

Morphisms in the homotopy category can be expressed in terms of the cohomology of the total Hom complex. Given $C$ and $D$
two complexes of $A$-modules, the total Hom complex, denoted by $\mathrm{Hom}_A^\bullet(C,D)$, is defined by
$$ \mathrm{Hom}_A^n(C,D) = \prod_{j-i = n} \mathrm{Hom}_{A} (C_i,D_j)$$
with the differential given by
$$ \delta(f_{i,j}) = \partial_j \circ f_{i,j} - (-1)^{j-i} f_{i,j} \circ d_{i-1}$$
for every $f_{i,j} \in  \mathrm{Hom}_{A} (C_i,D_j)$. One can readily check that $\mathrm{Ker}\, \delta_n$ consists of the morphisms of complexes from $C$ to $D[n]$ whereas  $\mathrm{Im}\, \delta_{n-1}$ is the subgroup of null-homotopic morphisms. Consequently,
\begin{equation}\label{eq:HnHom}
 H^n(\mathrm{Hom}_A^\bullet(C,D)) = \mathrm{Hom}_{\sfHo(A\sfMod)}(C,D[n]).
\end{equation}

\subsection{The derived category of $A$-$\mathsf{Mod}$}\label{sec:derived}
A morphism of complexes $f: C \longrightarrow D$ is a \emph{quasi-isomorphism} if the maps $H^n(f) : H^n(C) \longrightarrow H^n(D)$ induced on the cohomology groups are isomorphisms for all $n \in \mathbb{Z}$. By Proposition \ref{prop:coneseq} this is equivalent to the complex $\mathrm{Cone}(f)$ being acyclic.

\smallskip

The derived category is obtained from the homotopy category by formally inverting the quasi-isomorphisms (equivalently, by taking the quotient by the acyclic complexes). This is analogous to the construction of the fraction field of a domain. More precisely, given two pairs of morphisms in $\sfHo(A\sfMod)$, say 
$(s,f)$ and $(t,g)$ with $s$ and $t$ being quasi-isomorphisms, we write $(s,f) \equiv (t,g)$ if there exists a commutative diagram in $\sfHo(A\sfMod)$
  $$  \xymatrix@R=1.5em@C=2.5em{  & X \ar[ld]_s \ar[rd]^f & \\
  C & Z \ar[l]_r \ar[u]_a \ar[d]^b \ar[r]^h& D \\
  & Y \ar[lu]^t \ar[ru]_g & } $$
  with $r$ a quasi-isomorphism. This represents the relation ``$f s^{-1} = fa (sa)^{-1} = gb (tb)^{-1} = gt^{-1}$''.
  
\begin{definition} 
The \emph{derived category of $A$-modules}, denoted by $\sfD(A\sfMod)$, is the category with
 \begin{itemize}
  \item objects: complexes of $A$-modules (same as $\sfC(A\sfMod)$ and $\sfHo(A\sfMod)$),
  \item morphisms: 
  $$\mathrm{Hom}_{\sfD(A\sfMod)}(C,D) := \left\{  
  \vcenter{\vbox{\xymatrix@R=0.2em@C=1.5em{ & X \ar[ld]_s \ar[rd]^f & \\ C & &  D}}} \left| \begin{array}[c]{l} 
  X \text{ a complex of } A\text{-modules} \\ f \in \mathrm{Hom}_{\sfHo(A\sfMod)}(X,D) \\  s \in \mathrm{Hom}_{\sfHo(A\sfMod)}(X,C) \\ s \text{ a quasi-isomorphism} \end{array} \right.\right\}/ \equiv.$$ \end{itemize}
\end{definition}
The natural functor $\sfHo(A\sfMod) \longrightarrow \sfD(A\sfMod)$ (sending a morphism $f$ to the class of
$(1,f)$) sends quasi-isomophisms to isomorphisms. It is universal for this property. A distinguished triangle
in $\sfD(A\sfMod)$ will be by definition the image of a distinguished triangle in $\sfHo(A\sfMod)$. This endows
the derived category with a structure of triangulated category. We can think of distinguished triangles in
 $\sfD(A\sfMod)$ as analogues of short exact sequences of complexes. 

\pagebreak
 \begin{prop}\leavevmode \begin{itemize} 
 \item[$\mathrm{(i)}$] The functor $H^\bullet$ is well-defined on $\sfD(A\sfMod)$.
 \item[$\mathrm{(ii)}$] Any short exact sequence of complexes $0 \longrightarrow C \longrightarrow D \longrightarrow E \longrightarrow 0$
 in $\sfC(A\sfMod)$ yields a distinguished triangle $C \longrightarrow D \longrightarrow E \rightsquigarrow$ in
 $\sfD(A\sfMod)$.
 \item[$\mathrm{(iii)}$] The functor $A \longrightarrow \sfD(A\sfMod)$ sending a module $M$ to the complex $M[0]$ 
 is fully faithful.
\end{itemize}
\end{prop}
 
Note however that in general $H^\bullet(C) \simeq H^\bullet(D)$ does not imply $C \simeq D$ in $\sfD(A\sfMod)$. For example  $0 \longrightarrow \mathbb{C} \overset{0}{\longrightarrow}  \mathbb{C}\longrightarrow 0$ and $0 \longrightarrow  \mathbb{C}[x]/x^2 \overset{x}{\longrightarrow}  \mathbb{C}[x]/x^2\longrightarrow 0$ are not isomorphic in $\sfD( \mathbb{C}[x]/x^2\sfmod)$.

 \subsection{Morphisms in $D(A\sfMod)$}

By definition, morphisms in the derived category are equivalence classes of pairs of morphisms in the homotopy category $C \mathop{\longleftarrow}\limits^s X \mathop{\longrightarrow}\limits^f D$, representing ``$f\circ s^{-1}$''. We explain here how replacing $C$ and $D$ by projective or injective resolutions helps finding nice representatives for these morphisms.   

\smallskip

We will denote by $A\sfProj$ the full subcategory of $A\sfMod$ whose objects are the projective $A$-modules. The corresponding categories $\sfC(A\sfProj)$ and $\sfHo(A\sfProj)$ correspond to the full subcategories of  
$\sfC(A\sfMod)$ and $\sfHo(A\sfMod)$ respectively, whose objets are complexes of projective $A$-modules.
Similarly, $A\sfInj$ will refer to the additive category of injective $A$-modules, and $\sfC(A\sfInj)$ (resp. $\sfHo(A\sfInj)$) to the corresponding category of complexes (resp. the homotopy category).

\begin{lem}\label{lem:acyclictoproj}
Let $P \in C^-(A\sfProj)$ be a bounded above complex of projective $A$-modules, and $X,Y$ be two complexes of $A$-modules.
\begin{itemize}
\item[$\mathrm{(i)}$] If $X$ is an acyclic complex, any morphism $P \longrightarrow X$ is null-homotopic.
\item[$\mathrm{(ii)}$] Any quasi-isomorphism $Y \longrightarrow P$ splits in $\sfHo(A\sfMod)$.
\end{itemize}
\end{lem}

\begin{proof} Let $X$ be an acyclic complex and $f : P \longrightarrow X$ be a morphism. We can assume without loss of generality that $P_i = 0$ for $i > 0$.
We construct a homotopy $h : P \longrightarrow X[-1]$ as follows:
$$\xymatrix@C=3em@R=3em{\cdots \ar[r] & P_{-2} \ar[d] \ar[r]^{ } & P_{-1} \ar[r]^{d_{-1}} \ar[d]_{f_{-1}}
\ar[ld]_{h_{-1}}  & P_{0} \ar[r]\ar[d]_{f_0} \ar[ld]_{h_{0}} & 0 \ar[d]\ar[r] & \cdots \\
\cdots \ar[r] & X_{-2} \ar[r]^{} & X_{-1} \ar[r]^{\partial_{-1}}  & X_{0} \ar[r]^{\partial_0} & X_1 \ar[r] & \cdots  }$$ 
Since $f$ is a morphism of complexes, then $\partial_0 f_0 = 0$ hence $\mathrm{Im}\, \partial_{-1} = \mathrm{Ker}\, \partial_0 \supset \mathrm{Im}\, f_0$. Therefore since $P_0$ is projective there exists $h_0 : P_0 \longrightarrow X_{-1}$ such that $f_0 = \partial_{-1} h_0$. 

\smallskip
Similarly, using the fact that $\partial_{-1} (f_{-1}-h_0d_{-1}) = \partial_{-1} f_{-1}-\partial_{-1}h_0d_{-1} = \partial_{-1} f_{-1}-f_0 d_{-1} = 0$, there exists $h_{-1} : P_{-1} \longrightarrow X_{-2}$ such that $f_{-1}-h_0d_{-1} = \partial_{-2} h_{-1}$. We iterate this construction to get (i).

\smallskip
For (ii) we consider the distinguished triangle $Y \mathop{\longrightarrow}\limits^s P \longrightarrow \mathrm{Cone}(s) \rightsquigarrow$. By (i) the morphism $P \longrightarrow \mathrm{Cone}(s)$ is null-homotopic, yielding a map $h : P \longrightarrow \mathrm{Cone}(s)[-1] = P[-1]\oplus Y$. If $t : P \longrightarrow Y$ denotes the second projection of this map then one can show that $t$ is a morphism of complexes and $st \sim 1_P$.
\end{proof}

\begin{exercise} Check that $t$ is a morphism of complexes satisfying $st \sim 1_P$.
\end{exercise} 

\begin{prop}\label{prop:morphismsind}
Given $P \in \sfC^-(A\sfProj)$ and $C \in \sfC(A\sfMod)$ there is a natural isomorphism of $\mathbb{Z}$-modules
$$ \begin{array}{rcl} 
\mathrm{Hom}_{\sfHo(A\sfMod)}(P,C) & \mathop{\longrightarrow}\limits^\sim & \mathrm{Hom}_{\sfD(A\sfMod)}(P,C)\\ f & {\longmapsto} & (1,f) \end{array}$$
\end{prop}

\begin{proof}
For the injectivity of the map, let $f,g \in \mathrm{Hom}_{\sfHo(A\sfMod)}(P,C)$ be such that $(1,f) \equiv (1,g)$. This means that there is a commutative diagram
  $$  \xymatrix@R=1.5em@C=2.5em{  & P \ar@{=}[ld] \ar[rd]^f & \\
  P & Z \ar[l]_s \ar[u] \ar[d] \ar[r]^h&  C \\
  & P \ar@{=}[lu] \ar[ru]_g & } $$
with $s$ a quasi-isomorphism. By Lemma \ref{lem:acyclictoproj}, there exists $t : P \longrightarrow Z$ such that 
$st \sim 1_P$. From $fs = h = gs$ we get $f= fst = gst = g$ in $\sfHo(A\sfMod)$.

\smallskip
The surjectivity is another application of Lemma \ref{lem:acyclictoproj}. Indeed, any pair of morphisms $P \mathop{\longleftarrow}\limits^s X \mathop{\longrightarrow}\limits^f C$ can be completed in a commutative diagram
  $$  \xymatrix@R=1.5em@C=2.5em{  & X \ar[ld]_s \ar[rd]^f & \\
  P & P \ar[l]_{st} \ar[u]^t \ar[r]^{ft} &  C}$$
with $t$ satisfying $st \sim 1_P$. 
\end{proof}

For $? \in \{b,+,-\}$ we will denote by $\sfHo^?$ and $\sfD^?$ the essential images of the categories
of bounded/bounded below/bounded above complexes of $A$-modules.
The following theorem follows from Proposition \ref{prop:morphismsind} and the existence of projective (or injective) resolutions. The proof is left as an exercise.

\begin{theorem}\label{thm:hoandd}
The functor $\sfHo(A\sfMod) \longrightarrow \sfD(A\sfMod)$ induces equivalences
$$\begin{aligned}
 \sfHo^-(A\sfProj) & \mathop{\longrightarrow}^\sim\sfD^-(A\sfMod) \\
\text{and} \qquad  \sfHo^+(A\sfInj) & \mathop{\longrightarrow}^\sim\sfD^+(A\sfMod). 
\end{aligned}$$
\end{theorem}

\begin{rmk}
More generally, if one works with another abelian category $\mathcal{A}$ instead of $A\sfMod$, then the first isomorphism in Theorem \ref{thm:hoandd} (resp. the second isomorphism) remains true if $\mathcal{A}$ has enough projective objects (resp. enough injective objects). 
In these notes we shall often use that $A\sfmod$, the category of finitely generated $A$-modules has enough projectives. This guarantees the existence of projective resolutions of bounded above complexes. 
\end{rmk}

 \subsection{Derived functors}
Let $B$ be another ring with unit.
Given an additive functor $F : A\sfMod \longrightarrow B\sfMod$ we can form the triangulated functors 
$$\begin{aligned}
LF & : \sfD^-(A\sfMod) \simeq \sfHo^-(A\sfProj) \mathop{\longrightarrow}^F \sfHo^-(B\sfMod) \longrightarrow \sfD^-(B\sfMod) \\
RF & : \sfD^+(A\sfMod) \simeq \sfHo^+(A\sfInj) \mathop{\longrightarrow}^F \sfHo^+(B\sfMod) \longrightarrow \sfD^+(B\sfMod) \end{aligned}$$
where the first equivalences are quasi-inverses of the ones given in Theorem \ref{thm:hoandd}. 

\smallskip

When $F$ is right exact and $M \in A\sfMod$,  $LF(M)$ is a complex whose terms in positive degrees are zero and which satisfies $H^0(LF(M)) \simeq F(M)$. In that case we refer to $LF$ as the \emph{left derived functor} of $F$. Similarly, when $F$ is left exact, $RF$ is the \emph{right derived functor} and it satisfies $H^0(RF(M)) \simeq F(M)$ for every $A$-module $M$. Historically, only the lower (resp. higher) cohomology groups of $LF(M)$ (resp. $RF(M)$) were considered, not the complex itself. They yield additive functors between the module categories which we will denote by $L^n F := H^{-n} \circ LF$ and
$R^n F = H^n \circ LF$.  Note that if $F$ is exact then $LF \simeq RF \simeq F$.

\smallskip

\begin{example}\label{ex:extandtor}
(a) Given an $A$-module $M$, the functor $\mathrm{Hom}_A(M,-)$ is an additive, covariant, left exact functor from $A\sfMod$ to $\mathbb{Z}\sfMod$. It yields a right derived functor
$$ R\mathrm{Hom}_A(M,-) : \sfD^+(A\sfMod) \longrightarrow \sfD^+(\mathbb{Z}\sfMod).$$
Given another $A$-module $N$, the group of degree $n$ extensions between $M$ and $N$ is by definition 
$\mathrm{Ext}_A^n(M,N) := R^n\mathrm{Hom}_A(M,N)$. Therefore we have
$$\begin{aligned}
\mathrm{Ext}_A^n(M,N) & = H^n\big(R\mathrm{Hom}_A(M,N)\big) \\
& = H^n\big(\mathrm{Hom}_A^\bullet(M,I)\big) \qquad \text{\quad \ \ \ for $I$ an injective resolution of $N$} \\
& = \mathrm{Hom}_{\sfHo(A\sfMod)}(M,I[n]) \qquad \text{by \eqref{eq:HnHom}} \\
& = \mathrm{Hom}_{\sfD(A\sfMod)}(M,N[n]) \qquad \text{by Proposition \ref{prop:morphismsind}}. 
\end{aligned}$$
Note that more generally, the definition of $R\mathrm{Hom}_A(M,-)$ makes sense whenever $M$ is a bounded above complex of $A$-modules.

\smallskip
\noindent
(b) Given a right $A$-module $M$, the functor $M \otimes_A - : A \sfMod \longrightarrow \mathbb{Z}\sfMod$ is a right exact functor (not exact if $M$ is not flat). The corresponding left derived functor is denoted
by 
$$M {\mathop{\otimes}\limits^L}_A - : \sfD^-(A\sfMod) \longrightarrow \sfD^-(\mathbb{Z}\sfMod).$$
It is defined by $M {\mathop{\otimes}\limits^L}_A N := M \otimes_A P$ where $P$ is any projective resolution of $N$. This left derived functor yields  the torsion groups
$\mathrm{Tor}^A_n(M,N) = H^{-n}(M\otimes_A P)$ for any non-negative integer $n$.
\end{example}

\subsection{Truncation and applications}\label{sec:truncation}
Given a complex of $A$-modules $C$, one can consider the following truncations of $C$:
$$ \xymatrix@R=1mm@C=5mm{
\tau_{\geq n} (C) \, = \,  \cdots \ar[r] & 0 \ar[r]  & 0 \ar[r]  & \mathrm{CoKer}\, d_{n-1} \ar[r] & C_{n+1} \ar[r] & C_{n+2} \ar[r] &  \cdots \\
\widetilde{\tau}_{\geq n} (C) \, = \,  \cdots \ar[r] &  0 \ar[r] &  \mathrm{Im} \, d_{n-1} \ar[r] &    C_n \ar[r] & C_{n+1} \ar[r]& C_{n+2} \ar[r] &  \cdots  \\
\tau_{\leq n} (C)  \, = \, \cdots \ar[r] & C_{n-2} \ar[r] &  C_{n-1}  \ar[r] &  \mathrm{Ker}\, d_n \ar[r] &  0 \ar[r] &  0 \ar[r] &  \cdots \\\widetilde{\tau}_{\leq n} (C) \, = \,  \cdots \ar[r] & C_{n-2} \ar[r] &  C_{n-1}  \ar[r] & C_n \ar[r] &  \mathrm{Im}\, d_{n} 
\ar[r] & 0 \ar[r] &  \cdots }$$
The truncated complexes are constructed so that they have the same cohomology up to (or starting from) a given degree. For example, 
$$H^k(\tau_{\geq n}(C)) = H^k(\widetilde{\tau}_{\geq n}(C)) = \left\{\hskip-1.3mm\begin{array}{ll} H^k(C) & \text{if $k \geq n$,}\\ 0 & \text{otherwise.} \end{array}\right.$$
The following proposition summarizes the relations and properties of truncation operations.

\begin{prop}\label{prop:truncation}
Let $C$ be a complex of $A$-modules.
\begin{itemize}  
\item[$\mathrm{(i)}$] The natural map $\widetilde{\tau}_{\leq n}(C) \longrightarrow \tau_{\leq n}(C)$ is a quasi-isomorphism. 
\item[$\mathrm{(ii)}$] There are short exact sequences of complexes
$$\begin{aligned}
& 0 \longrightarrow \widetilde{\tau}_{< n}(C) \longrightarrow \tau_{\leq n}(C) \longrightarrow H^n(C)[-n] \longrightarrow 0 \\
& 0 \longrightarrow {\tau}_{\leq n}(C) \longrightarrow C \longrightarrow \widetilde{\tau}_{> n}(C) \longrightarrow 0 \\
\end{aligned}$$
\item[$\mathrm{(iii)}$] The truncation operations are functorial. They preserve the class of acyclic complexes, null-homotopic complexes, null-homotopic morphisms and quasi-isomorphisms.
\end{itemize}
\end{prop}

\noindent 
Note that similar statements are obtained by reversing the arrows and swapping $\tau_{\leq n}$ and $\widetilde{\tau}_{\leq n}$ with $\tau_{\geq n}$ and  $\widetilde{\tau}_{\geq n}$ respectively.

\smallskip
We deduce that the truncation functors induce functors at the level of the homotopy and derived categories. The short exact sequences in (ii) together with the quasi-isomorphism in (i) yield distinguished triangles in $\sfD(A\sfMod)$ 
\begin{equation}\label{eq:dttruncation}
\begin{aligned}
&  {\tau}_{< n}(C) \longrightarrow \tau_{\leq n}(C) \longrightarrow H^n(C)[-n] \rightsquigarrow \\
&  {\tau}_{\leq n}(C) \longrightarrow C \longrightarrow {\tau}_{> n}(C) \rightsquigarrow 0 
\end{aligned}
\end{equation}
Consequently, a complex $C$ whose homology vanishes outside the degrees $n,n+1,\ldots,m$ is quasi-isomorphic to its truncation $\tau_{\geq m} \tau_{\leq n} (C)$, hence to a bounded complex whose terms are zero outside the degrees $n,n+1,\ldots,m$. In particular a complex with a unique non-zero homology group is quasi-isomorphic to a module shifted in that degree. Another consequence is that the category $\sfD^b(A\sfMod)$ is the full subcategory of $\sfD(A\sfMod)$ with objects satisfying $H^i(C) = 0$ for $|i| \gg 0$. 

\smallskip

Another example that will often appear in these notes is the case of a complex $C$ with only two non-zero homology groups, say $H^0(C)$ and $H^n(C)$. Then $C$ fits into a distinguished triangle
$$ H^0(C)[0] \longrightarrow C \longrightarrow H^n(C)[-n] \rightsquigarrow$$
which means that $C$ is quasi-isomorphic to the cone of the map $H^n(C)[-n] \longrightarrow H^0(C)[1]$.  This implies that $C$ is determined by $H^0(C)$, $H^n(C)$ and by an element of $\mathrm{Ext}_A^{n+1}(H^n(C),H^0(C)) = \mathrm{Hom}_{\sfD(A\sfMod)}(H^n(C),H^0(C)[n+1])$ (see Example \ref{ex:extandtor}). 

\subsection{Examples of derived categories}\label{sec:derex}
When $A$ is a semisimple algebra, every injective or surjective map between modules splits. Consequently one can easily show that every complex of $A$-modules is quasi-isomorphic to the complex formed by its cohomology groups (with zero differential). In other words, the functor  $C \longmapsto H^\bullet(C)$ induces an equivalence between $\sfD(A\sfMod)$ and the category of $\mathbb{Z}$-graded modules. This is in particular the case for the derived category of  $k$-vector spaces $D(k\sfMod)$ when $k$ is a field, or more generally for the derived category of $kG$-modules $D(kG\sfMod)$ when $G$ is a finite group whose order is invertible in $k$.

\smallskip

A ring $A$ is said to be \emph{hereditary} if $\mathrm{Ext}_A^n(-,-) = 0$ for all $n \geq 2$. In that case, every bounded complex is again quasi-isomorphic to its cohomology, but not in a canonical way, and the functor $C \longmapsto H^\bullet(C)$ is not faithful in general. This is for example the case for the bounded derived category of abelian groups  $D^b(\mathbb{Z}\sfMod)$. Another example is  the bounded derived category of $\mathbb{Z}_\ell G$-modules $D^b(\mathbb{Z}_\ell G\sfMod)$ when $G$ is a finite group whose order is prime to $\ell$.

\begin{exercise} When $A$ is hereditary, show using \eqref{eq:dttruncation} and Example \ref{ex:extandtor}.a that every bounded complex is quasi-isomorphic to its cohomology.
\end{exercise} 

\subsection{The stable category $A\sfstab$}

Let $k$ be a field. 
Throughout this section we will assume that $A$ is a finite dimensional $k$-algebra.
In particular every finite dimensional $A$-module $M$ has a projective cover, which we will denote by $P_M$ ($A$ is said to be \emph{semiperfect}).
We shall also assume that  $A$ is \emph{symmetric} (\emph{i.e.} $A$ is isomorphic to its dual $A^* = \mathrm{Hom}_k(A,k)$ as an $(A,A)$-bimodule).  In that case $A$-modules are projective if and only if they are injective (see for example \cite[\S1.6]{Ben}). Consequently any finite dimensional $A$-module admits an injective resolution. A typical example of such an algebra in these lectures is the group algebra $kG$ of a finite group $G$. 

\begin{definition}The \emph{stable category of finitely generated $A$-modules}, denoted by $A\sfstab$, is the category with
 \begin{itemize}
  \item objects: finitely generated $A$-modules (same as $A\sfmod$),
  \item morphisms: $\underline{\mathrm{Hom}}_{A}(M,N) := \mathrm{Hom}_{A}(M,N)/\hskip-1.3mm\approx$ where $f \approx g$ if and only if $f-g$ factors through a projective module.
 \end{itemize}
\end{definition}

\noindent
In particular, in the stable category any projective module is isomorphic to zero. 
There is a canonical additive (in fact $k$-linear) functor $A\sfmod \longrightarrow A\sfstab$, making $A\sfstab$ into an additive ($k$-linear) category. This category has an additional triangulated structure, as we will see below. 

\smallskip

Given a finite dimensional $A$-module $M$, we define the \emph{Heller operator} $\Omega$ by
$$\Omega M = \mathrm{Ker}\, (P_M \twoheadrightarrow M).$$
We then define inductively $\Omega^n(M) = \Omega(\Omega^{n-1}(M))$ for $n \geq 1$ 
with the convention that $\Omega^0(M)$ is the minimal submodule of $M$ such that $M/\Omega^0(M)$ is projective. 

\begin{exercise}
Check that $\Omega^n(M)$ is well defined up to isomorphism. Show that $\Omega$ is functorial in $A\sfstab$ (but not in $A\sfmod$).
\end{exercise}

Similarly, we set $\Omega^{-1} M = \mathrm{Coker}\, (M \hookrightarrow I_M)$ where $I_M$ is an injective hull of $M$. One can readily check that $(\Omega^{-1} M)^* \simeq \Omega M^*$ as right $A$-modules and more generally that $(\Omega^{-n} M)^* \simeq \Omega^n M^*$ for all $n \in \mathbb{Z}$. 

\begin{prop}\label{prop:stab}
Let $M$ and $N$ be finitely generated $A$-modules.
\begin{itemize}
  \item[$\mathrm{(i)}$] $M \simeq N$ in $A\sfstab$ if and only if there exist finitely generated projective modules $P$ and $Q$ such that $M \oplus P \simeq N \oplus Q$ in $A\sfmod$.
  \item[$\mathrm{(ii)}$] If $M$ and $N$ are indecomposable non-projective modules with $M$ or $N$ being simple, then $\mathrm{Hom}_A(M,N) \mathop{\longrightarrow}\limits^\sim \underline{\mathrm{Hom}}_A(M,N)$.
  \item[$\mathrm{(iii)}$] If $n > 0$ then 
  $$ \underline{\mathrm{Hom}}_A(\Omega^n M,N) \simeq \underline{\mathrm{Hom}}_A(M,\Omega^{-n} N) \simeq \mathrm{Ext}_A^n(M,N).$$
\end{itemize}
\end{prop}

\begin{proof} (i) If $M \simeq N$ in $A\sfstab$ then there exists morphisms $M 
\mathop{\longrightarrow}\limits^f N \mathop{\longrightarrow}\limits^g M$ and projective modules $R$ and 
$T$ such that $gf-1_M$ and $fg-1_N$ factor  through $R$ and $T$ respectively. Write $gf-1_M = \Psi \circ \Phi$ with $\Phi : M \longrightarrow R$ and $\Psi : R \longrightarrow M$ and consider the following morphisms
$$ \xymatrix{ M \ar@/^1pc/[r]^{(f,\Phi)} & \ar@/^1pc/[l]^{g-\Psi} N \oplus R }$$
which satisfy $(g-\Psi)\circ(f,\Phi) = 1_M$. This shows that $M$ is a direct summand of $N \oplus R$. Similarly, $N$ is a direct summand of $M \oplus T$ and we can invoke Krull-Schmidt Theorem to conclude.

\smallskip

For (ii) it is enough to see that no injective (resp. surjective) morphism can factor through a projective (hence injective) module if both $M$ and $N$ have no non-trivial projective summands. 

\smallskip

For (iii) we start with a minimal projective resolution of $M$, given by
$$ \cdots \longrightarrow P_{\Omega^n M} \longrightarrow \cdots \longrightarrow P_{\Omega M} \longrightarrow P_M \twoheadrightarrow M.$$
By definition (see Example \ref{ex:extandtor}), the degree $n$ extension group $\mathrm{Ext}_A^n(M,N)$ is the quotient of  
the subgroup of maps $f$ in $\mathrm{Hom}_A(P_{\Omega^n M}, N)$ such that $\Omega^{n+1} M \subset \mathrm{Ker}\, f$ by the maps which factor through $P_{\Omega^{n} M} \longrightarrow P_{\Omega^{n-1} M}$.
In particular we have a well defined surjective map
$$ \begin{array}{rcl} \mathrm{Ext}_A^n(M,N) & {\longrightarrow} & \underline{\mathrm{Hom}}_A(\Omega^n M,N) \\[5pt] f & {\longmapsto} & \overline{f} : \underbrace{P_{\Omega^n M}/\Omega^{n+1}M}_{\Omega^n M}  \longrightarrow N \end{array}$$
Now if $\overline{f} = 0$ there exists a projective (hence injective) module $P$ such that $\overline{f}$ factors through $\Omega^n M \longrightarrow P \longrightarrow N$. Since $P$ is injective and $\Omega^n M$ is a submodule of $P_{\Omega^{n-1} M}$, the map
$\Omega^n M \longrightarrow P$ can be extended to a map $P_{\Omega^{n-1} M} \longrightarrow P$ in the following commutative diagram
$$\xymatrix@R=2mm@C=5mm{ \Omega^n M \ar[rr]^{\overline{f}} \ar@{^(->}[dd] \ar@{-->}[rd]& & N \\ & P \ar@{-->}[ru] & \\ P_{\Omega^{n-1} M} \ar@{-->}[ru]  }$$
Therefore $f$ factors through  $P_{\Omega^{n-1} M}$ and hence it is zero in $\mathrm{Ext}_A^n(M,N)$. 
The case of $\Omega^{-n} N$ is similar.
\end{proof}

Consequently, any short exact sequence $0 \longrightarrow U \longrightarrow V \longrightarrow W \longrightarrow 0$ in $A\sfmod$ yields an exact sequence $ U \longrightarrow V \longrightarrow W \longrightarrow \Omega^{-1} U$ in $A\sfstab$. This, in turn, endows $A\sfstab$ with a structure of triangulated category, with suspension functor (or shift) $\Omega^{-1}$, such that the images of short exact sequences of $A\sfmod$ are distinguished triangles in $A\sfstab$. 

\smallskip

The triangulated structure appears in a more natural way from the bounded derived category of $A\sfmod$. We say that a complex of $A$-modules $C$ is \emph{perfect} if it is quasi-isomorphic to a bounded complex of finitely generated projective modules. We denote by $A\sfperf$ the full subcategory of $\sfD^b(A\sfmod)$ of perfect complexes (it is a thick subcategory, \emph{i.e.} stable under direct summands and cones). 

\begin{theorem}[Rickard \cite{Ric89}]\label{thm:derstab}
The natural functor $A\sfmod \longrightarrow \sfD^b(A\sfmod)$ induces an equivalence of triangulated categories
$$ A\sfstab \mathop{\longrightarrow}\limits^\sim \sfD^b(A\sfmod) \tikzmark{eq-perf}/ A\sfperf.$$
\end{theorem}

\begin{tikzpicture}[remember picture,overlay]
\draw[->,>=latex]
  ([shift={(2pt,-5pt)}]pic cs:eq-perf) |- 
  ++(10pt,-15pt) 
  node[right,text width=4cm] 
    {\footnotesize inverting maps whose cone \\[-4pt] is a perfect complex
    };
\end{tikzpicture}

\vskip 15pt

\begin{rmk}
Let us consider a projective resolution of $M[-n]$, truncated in degrees above zero 
$$C = (\cdots 0  \longrightarrow P_{\Omega^{n-1} M} \longrightarrow \cdots \longrightarrow P_{\Omega M} \longrightarrow P_M \twoheadrightarrow M \longrightarrow 0 \cdots)$$
Then $C \simeq \Omega^n M[0]$ in $\sfD^b(A\sfmod)$. On the other hand, since all the terms of $C$ are projective modules except the term in degree $n$ we have $C \simeq M[-n]$ in $\sfD^b(A\sfmod)/A\sfperf$. This shows that 
$$ M[-n] \simeq \Omega^n M \quad \text{ in } \sfD^b(A\sfmod)/A\sfperf$$
and proves the compatibility of the suspension functors ($\Omega^{-1}$ and $[1]$) under the equivalence given by Rickard's theorem.
\end{rmk}

\section{Varieties and cohomology}

The aim of this chapter is to introduce the geometric tools that we will need to construct the representations of finite reductive groups. They will be obtained from linear invariants (cohomology groups or cohomology complexes) of algebraic varieties acted on by finite groups:
$$\begin{minipage}[t]{.25\textwidth}
Algebraic variety $\bfX$ \\
+ action of\\
a finite group $G$\end{minipage}
\hfill
\noindent
\begin{minipage}[t]{.20\textwidth}
$ \mathrm{ }$ \\
$\underrightarrow{\scriptsize\text{ \quad cohomology \quad }}$
\end{minipage}\hfill
\noindent
\begin{minipage}[t]{.45\textwidth}
Family of vector spaces $H^i(\bfX)$\\
or complex of vector spaces $R\Gamma(\bfX)$ \\
+ linear action of $G$.
\end{minipage} $$

\smallskip

\noindent
For example, if $\bfX$ is a finite set acted on by $G$, then we can form the permutation module $\Lambda\bfX$ over any ring $\Lambda$.

\smallskip

Although this construction makes sense for any abstract finite group $G$, it will be particularly suited for finite reductive groups, since in that case the algebraic variety $\bfX$ will be constructed from the underlying algebraic group (see \S\ref{chap:dlvarieties} for the definition of Deligne--Lusztig varieties).

\smallskip

Since we will be interested in modular representations (with coefficients in fields of positive characteristic) the language of derived categories and derived functors introduced in the previous chapter will be particularly suited for our purpose:
\vskip 8pt \hskip-0.3cm
\begin{tabular}{>{\centering}m{2.7cm}@{$\ \ \rightsquigarrow\ \ $}>{\centering}m{2.4cm}@{$\ \ \rightsquigarrow\ \ $}>{\centering}m{2.2cm}@{$\ \ \rightsquigarrow\ \ $}>{\centering}m{2.3cm}} 
$A$-modules $M$, $N$ &  $P$ a projective resolution of $N$ & 
 $\left\{ \hskip-1.3mm \begin{array}{l} P \otimes_A N \\ \mathrm{Hom}_A^\bullet(P,N) \end{array} \right.$ &  $\left\{\hskip-1.3mm  \begin{array}{l} \mathrm{Tor}_i^A(M,N) \\ \mathrm{Ext}_A^i(M,N) \end{array} \right. $ 
 \tabularnewline 
 \multicolumn{4}{c}{ }
 \tabularnewline
$\mathcal{F}$ a sheaf on  $\bfX$ & $\mathcal{P}$ a flabby resolution of $\mathcal{F}$ & 
 $\Gamma(\mathcal{P})$\tikzmark{eq-der} &  $H^i(\bfX,\mathcal{F})$\tikzmark{eq-std} 
  \tabularnewline 
 \end{tabular}

\begin{tikzpicture}[remember picture,overlay]
\draw[->,>=latex]
  ([shift={(-10pt,-10pt)}]pic cs:eq-der) |- 
  ++(-10pt,-20pt) 
  node[left,text width=1cm] 
    {\footnotesize derived\\[-4pt] setting
    };
\draw[->,>=latex]
  ([shift={(-45pt,-10pt)}]pic cs:eq-der)  rectangle
  ([shift={(25pt,55pt)}]pic cs:eq-der) ;
\draw[->,>=latex]
  ([shift={(-10pt,-10pt)}]pic cs:eq-std) |- 
  ++(-10pt,-20pt) 
  node[left,text width=1.2cm] 
    {\footnotesize standard\\[-4pt] setting
    };
\draw[->,>=latex]
  ([shift={(-55pt,-10pt)}]pic cs:eq-std)  rectangle
  ([shift={(15pt,55pt)}]pic cs:eq-std) ;
\end{tikzpicture}

\vskip 1.1cm

The definition of \'etale or $\ell$-adic cohomology would go far beyond the scope of these notes. For the reader interested in the topic we recommend reading Deligne's notes \cite{SGA45}, or the excellent textbook by Milne \cite{Mil}. For a more representation-theoretic perspective, most of the properties listed in this chapter are also addressed in \cite[Appendix A3]{CabEng} and in \cite[Appendix A]{Bon}.

\subsection{Definition and first properties}\label{sec:cohodef}
Let $\bfX$ be a quasi-projective variety over $\overline{\mathbb{F}}_p$ and $G$ be a finite group acting on $\bfX$.
We fix a prime number $\ell \neq p$ and an $\ell$-modular system  $(K,\calO,k)$ such that $K$ is a finite extension of $\mathbb{Q}_\ell$. Finally, we denote by $\Lambda$ any ring among $K$, $\calO$ and $k$. We will be interested in representations of $G$ over $\Lambda$. 

\smallskip

The theory of \'etale cohomology of sheaves on $\bfX$ produces two complexes of $\calO G$-modules $R \Gamma(\bfX,\calO)$ and $R \Gamma_c(\bfX,\calO)$, unique up to quasi-isomorphism, called the \emph{cohomology complex} of $\bfX$ and the \emph{cohomology complex with compact support} of $\bfX$. By extension of scalars, we also have complexes 
$$R\Gamma(\bfX,\Lambda) := R\Gamma(\bfX,\calO) {\mathop{\otimes}\limits^L}_\calO \Lambda \quad \text{and} 
 \quad R\Gamma_c(\bfX,\Lambda) := R\Gamma_c(\bfX,\calO) {\mathop{\otimes}\limits^L}_\calO \Lambda.$$
When $\Lambda = K$ (resp. $\Lambda=k$) we will refer to these complexes as the $\ell$-adic cohomology complexes (resp. the mod-$\ell$ cohomology complexes). The groups
$$H^i(\bfX,\Lambda) := H^i\big(R\Gamma(\bfX,\Lambda)\big) \quad \text{and} 
 \quad H_c^i(\bfX,\Lambda) := H^i\big(R\Gamma_c(\bfX,\Lambda)\big) $$
are the cohomology groups (or cohomology groups with compact support) with coefficients in $\Lambda$.

\smallskip
The cohomology is functorial: if $f : \bfY \longrightarrow \bfX$ is a $G$-equivariant morphism of algebraic varieties
then it induces a morphism in $\sfD(\Lambda G\sfMod)$ 
$$ f^* : R\Gamma(\bfX,\Lambda) \longrightarrow R\Gamma(\bfY,\Lambda)$$
between the cohomology complexes of $\bfX$ and $\bfY$. If in addition $f$ is proper (e.g. $f$ is a finite morphism), the same holds for the cohomology complexes with compact support. 

\smallskip

The cohomology complexes are ``small'': the $\Lambda G$-modules  $H^i(\bfX,\Lambda)$ and 
$H_c^i(\bfX,\Lambda)$ are finitely generated over $\Lambda$. Moreover, they vanish for $i <0$ and $i>2\dim\bfX$.
Consequently, $R\Gamma(\bfX,\Lambda)$ and $R\Gamma_c(\bfX,\Lambda)$ are quasi-isomorphic to complexes of ($\Lambda$-free) finitely generated $\Lambda G$-modules with terms in degrees $0,1,\ldots,2\dim \bfX$.

\begin{rmk} The $\Lambda$-modules $H^i(\bfX,\calO)$ are not free in general, but $H^0(\bfX,\calO)$ is. If $C$ is a $\Lambda$-free resolution of 
a given representative of $R\Gamma(\bfX,\Lambda)$, then $C$ is quasi-isomorphic to $\tau_{\geq 0} (\tau_{\leq 2\dim\bfX}(C))$ and the terms of the latter complex are $\Lambda$-free since $H^0(\bfX,\Lambda)$ is (see also Proposition \ref{prop:middle-torsionfree}). 
\end{rmk}

The following theorem, due to Rickard \cite{Ric94} (see also \cite{Rou02}), gives the most satisfactory representative for the cohomology complex of $\bfX$ from a representation-theoretic perspective.

\begin{theorem}[Rickard \cite{Ric94}]\label{thm:rickardlperm}
$R\Gamma(\bfX,\Lambda)$ and $R\Gamma_c(\bfX, \Lambda)$ are quasi-isomorphic
to bounded complexes whose terms are direct summands of finite sums of permutation modules $\Lambda G/\mathrm{Stab}_G(x)$ for $x \in \bfX$.
\end{theorem}

\begin{proof}[Idea of proof]
Given a sheaf $\mathcal{F}$ on $\bfX$, we can construct the complex $R\Gamma(\bfX,\mathcal{F})$ from the global sections of the Godement resolution. This resolution involves $\Lambda$-modules of the form $\prod_{x \in \bfX} \mathcal{F}_x$ on which $G$ acts naturally. This suggests that $R\Gamma(\bfX,\Lambda)$ has a representative in the category generated by the permutation modules $\Lambda G/\mathrm{Stab}_G(x)$ and we can invoke the finiteness property of the cohomology groups to conclude.
\end{proof}

For $\Lambda =\calO$ or $k$, we denote by $\Lambda G\sfperm$ the category of finitely generated $\ell$-permutation $\Lambda G$-modules. This is the smallest full subcategory of $\Lambda G\sfmod$ closed under direct summands and containing the permutation modules. As a consequence of Theorem \ref{thm:rickardlperm}, there exist (unique up to isomorphism) bounded complexes $\widetilde{R}\Gamma(\bfX,\calO)$ and $\widetilde{R}\Gamma_c(\bfX,\calO)$ in $\sfHo^b(\calO G\sfperm)$ of finitely generated $\ell$-permutation modules which are quasi-isomorphic to  ${R}\Gamma(\bfX,\calO)$ and ${R}\Gamma_c(\bfX,\calO)$ respectively. The following particular case will be intensively used in these notes.

\begin{cor}\label{cor:perfect}
Assume that for all $x \in \bfX$ the order of the group $\mathrm{Stab}_G(x)$ is invertible in $\Lambda$. Then 
${R}\Gamma(\bfX,\Lambda)$ and ${R}\Gamma_c(\bfX,\Lambda)$ are perfect complexes.
\end{cor}

For affine varieties, the vanishing property of cohomology groups can be refined. If $\bfX$ is an affine variety of pure dimension (\emph{i.e.} all the irreducible components have the same dimension) then 
\begin{itemize}
 \item[$\bullet$] $H^i(\bfX,\Lambda) = 0$ if $i > \dim\bfX$;
 \item[$\bullet$] $H_c^i(\bfX,\Lambda) = 0$ if $i < \dim\bfX$.
\end{itemize}
Consequently, $R\Gamma(\bfX,\Lambda)$ (resp. $R\Gamma_c(\bfX,\Lambda)$) has a representative with terms in degrees $0,\ldots,\dim\bfX$ (resp. $\dim\bfX,\ldots,2\dim\bfX$). 
 
\smallskip

We conclude this section by the relation between the compact and non-compact versions of the cohomology complexes. There is a natural map $R\Gamma_c(\bfX,\Lambda) \longrightarrow R\Gamma(\bfX,\Lambda)$ which is an isomorphism when $\bfX$ is a projective variety (compare with the case of affine varieties above). In addition, the cohomology complexes of smooth varieties are mutually dual.

\begin{theorem}[Poincar\'e-Verdier \cite{SGA45}]\label{thm:poincareverdier}
Assume that $\bfX$ is smooth of pure dimension~$d$. Then
$$ R\Gamma(\bfX,\Lambda)[2d]  \simeq  \tikzmark{eqverdier}R\mathrm{Hom}_\Lambda(R\Gamma_c(\bfX,\Lambda),\Lambda)$$
\begin{tikzpicture}[remember picture,overlay]
\draw[->,>=latex]
  ([shift={(5pt,-5pt)}]pic cs:eqverdier) |- 
  ++(10pt,-18pt) 
  node[right,text width=6.5cm] 
    {\footnotesize not necessary to right derive if one works\\[-4pt] with a representative of $R\Gamma_c(\bfX,\Lambda)$\\[-4pt] with $\Lambda$-free terms 
    };
\end{tikzpicture}
in $\sfD^b (\Lambda G\sfmod)$. 
\end{theorem}

\smallskip

\subsection{Tools for computing $R\Gamma_c(\mathbb{X},\Lambda)$}\label{sec:cohomologytools}
Unless otherwise stated, all the isomorphisms considered in this section are in the category $\sfD^b(\Lambda G\sfmod)$ for $\Lambda$ a ring among $K$, $\calO$ and $k$.

\begin{theorem}[K\"unneth formula]\label{thm:Kunneth}
The cohomology of a product of varieties is given by
$$ R\Gamma_c(\bfX \times \bfY,\Lambda) \simeq R\Gamma_c(\bfX,\Lambda) \mathop{\otimes}\limits^L R\Gamma_c(\bfY,\Lambda).$$
\end{theorem}

\begin{theorem}[Open-closed situation]\label{thm:openclose}
Let $\mathbb{U} \subset \bfX$ be an open $G$-stable subvariety of $\bfX$, and $\mathbb{Z} = \bfX\smallsetminus \mathbb{U}$ be the closed complement. There is a distinguished triangle
$$R\Gamma_c(\mathbb{U},\Lambda) \longrightarrow R\Gamma_c(\bfX,\Lambda) \longrightarrow R\Gamma_c(\mathbb{Z},\Lambda) \rightsquigarrow$$
in $\sfD^b(\Lambda G\sfmod)$ which splits if $\mathbb{U}$ is also closed.
\end{theorem}

Taking the cohomology of this distinguished triangle yields a long exact sequence of cohomology groups
$$ \cdots \longrightarrow H_c^i(\mathbb{U}) \longrightarrow H_c^i(\bfX) \longrightarrow 
H_c^i(\mathbb{Z}) \longrightarrow H_c^{i+1}(\mathbb{U}) \longrightarrow \cdots$$

\begin{theorem}\label{thm:cohomologyA}
The cohomology of the affine space of dimension $n$ is given by
$$ R\Gamma_c(\mathbb{A}_n,\Lambda) \simeq \Lambda[-2n].$$
\end{theorem}

These three results are enough to compute the cohomology of a large class of varieties. We give below the examples of projective spaces and tori.

\begin{example}
(a) Write $\mathbb{P}_1 = \mathbb{A}_1 \sqcup \{pt\}$. Theorems \ref{thm:openclose} and \ref{thm:cohomologyA} yield a distinguished triangle
$$ \begin{aligned}
R\Gamma_c(\mathbb{A}_1,\Lambda) \longrightarrow & \, R\Gamma_c(\mathbb{P}_1,\Lambda) \longrightarrow R\Gamma_c(\{pt\},\Lambda) \rightsquigarrow \\
\Lambda[-2] \longrightarrow & \, R\Gamma_c(\mathbb{P}_1,\Lambda) \longrightarrow \Lambda[0] \tikzmark{eq-p1}\longrightarrow \Lambda[-1]
\end{aligned}$$ 
\begin{tikzpicture}[remember picture,overlay]
\draw[->,>=latex]
  ([shift={(8pt,-5pt)}]pic cs:eq-p1) |- 
  ++(10pt,-15pt) 
  node[right,text width=3cm] 
    {\footnotesize zero map since\\[-3pt]$\mathrm{Ext}_{\Lambda G}^{-1}(\Lambda,\Lambda) = 0$
    };
\end{tikzpicture}

\vskip 0.5cm
\noindent This shows that $R\Gamma_c(\mathbb{P}_1,\Lambda) \simeq \Lambda[0] \oplus\Lambda[-2]$.

\smallskip
More generally, $R\Gamma_c(\mathbb{P}_n,\Lambda) \simeq \Lambda[0] \oplus\Lambda[-2] \oplus \cdots \oplus \Lambda[-2n]$. This method works for any variety paved by affine spaces, e.g. the flag varieties. 

\smallskip

\noindent (b) Let $\mathbb{G}_m$ be the one-dimensional torus $\mathbb{G}_m = \mathbb{A}_1 \smallsetminus \{0\}$ acted on by multiplication by the group $\mu_n$ of $n$-th roots of unity in $\overline{\mathbb{F}}_p$. Again, we have a distinguished triangle 
$$ \begin{aligned}
R\Gamma_c(\mathbb{G}_m,\Lambda) \longrightarrow &\,  R\Gamma_c(\mathbb{A}_1,\Lambda) \longrightarrow R\Gamma_c(\{pt\},\Lambda) \rightsquigarrow \\
R\Gamma_c(\mathbb{G}_m,\Lambda) \longrightarrow &\,  \Lambda[-2]  \tikzmark{eq-gm} \longrightarrow \Lambda[0] \rightsquigarrow
\end{aligned}$$ 
\begin{tikzpicture}[remember picture,overlay]
\draw[->,>=latex]
  ([shift={(8pt,-5pt)}]pic cs:eq-gm) |- 
  ++(10pt,-12pt) 
  node[right,text width=6cm] 
    {\footnotesize element of  $\mathrm{Ext}_{\Lambda \mu_n}^{2}(\Lambda,\Lambda) \simeq H^2(\mu_n,\Lambda)$
    };
\end{tikzpicture}
\vskip 0.5cm
\noindent The long exact sequence in cohomology gives $H_c^\bullet(\mathbb{G}_m,\Lambda) \simeq \Lambda[-1] \oplus \Lambda[-2]$ but we need more information to compute the cohomology complex.

\smallskip

Since $\mu_n$ acts freely, the complex $R\Gamma_c(\mathbb{G}_m,\calO)$ is perfect by Corollary \ref{cor:perfect}. Therefore it is quasi-isomorphic to $0 \longrightarrow P \mathop{\longrightarrow}\limits^d Q \longrightarrow 0$ with $P$ and $Q$ being two finitely-generated projective modules in degrees $1$ and $2$ respectively. From the previous computation we deduce that $\mathrm{Ker}\, d \simeq \mathrm{CoKer}\, d \simeq \calO$. Consequently, the trivial module $k$ is in the head of $Q$ and the projective cover $P_k$ of $k$ is a direct summand of $Q$. In other words, $R\Gamma_c(\mathbb{G}_m,\calO)$ has a representative of the form
$$0 \longrightarrow P \mathop{\longrightarrow}\limits^d Q'\oplus P_k \longrightarrow 0$$
with $Q' \subset \mathrm{Im}\, d$. This implies that the composition $P \mathop{\longrightarrow}\limits^d Q'\oplus P_k \longrightarrow Q'$ is surjective. Since $Q'$ is projective, it must split and we can write $P \simeq P' \oplus Q'$ such that the restriction of $d$ to $Q'$ is the identity. This shows that the previous complex is homotopy equivalent to 
$$0 \longrightarrow P' \mathop{\longrightarrow}\limits^{d_{|P'}}  P_k \longrightarrow 0.$$
To determine $P'$ we can either use the kernel of $d$ or the fact that in the Grothendieck group 
$[R\Gamma_c(\mathbb{G}_m,\calO)] = [P_k]-[P'] = \sum (-1)^i [H_c^i(\mathbb{G}_m,\calO)] = 0$ which forces $P' \simeq P_k$. Finally,
$$R\Gamma_c(\mathbb{G}_m,\calO) \simeq (\cdots 0 \longrightarrow P_k \longrightarrow  P_k \longrightarrow 0 \cdots) .$$
Using Theorem \ref{thm:Kunneth} we can also compute the cohomology of a higher dimensional torus by $R\Gamma_c((\mathbb{G}_m)^r,\calO) \simeq R\Gamma_c(\mathbb{G}_m,\calO)^{{\mathop{\otimes}\limits^L}r}$.
\end{example}

\begin{exercise}
Let $\zeta$ be a primitive $n$-th root of $1$. Show that 
$$ R\Gamma_c(\mathbb{G}_m,\Lambda) \simeq (\cdots 0 \longrightarrow \Lambda \mu_n \mathop{\longrightarrow}\limits^{\zeta -1} \Lambda \mu_n \longrightarrow 0 \cdots).$$
\end{exercise}

\subsection{Group action}
Recall that $G$ is a finite group acting on the quasi-projective variety $\bfX$. In this section we discuss the relation between the cohomology complexes of $\bfX$, $G \backslash \bfX$ and $\bfX^G$.
Further details can be found in \cite{Ric94} or \cite{Rou02}.

\begin{theorem}\label{thm:cohomologyinv}
Assume that for all $x \in \bfX$, the order of the group $\mathrm{Stab}_G(x)$ is invertible in $\Lambda$ (in particular $R\Gamma_c(\bfX, \Lambda)$ is perfect). Then 
 $$ R\Gamma_c(G \backslash \bfX, \Lambda) \simeq \Lambda \, {\mathop{\otimes}\limits^L}_{\Lambda G}\, R\Gamma_c(\bfX,\Lambda)$$
in $\sfD^b(\Lambda\sfmod)$.
\end{theorem}

\begin{proof}[Sketch of proof]
Let $\pi_*\Lambda$ be the push-forward of the constant sheaf $\Lambda$ along the quotient map $\pi : \bfX \longrightarrow G\backslash \bfX$. Since $\pi$ is finite we have $\pi_* = \pi_{!}$ and therefore
$$ R\Gamma_c(\bfX, \Lambda) \simeq R\Gamma_c(G\backslash \bfX,\pi_*\Lambda).$$
Taking the coinvariants we get, using the projection formula
$$ \Lambda \, {\mathop{\otimes}\limits^L}_{\Lambda G}\, R\Gamma_c(\bfX,\Lambda) 
\simeq \Lambda \, {\mathop{\otimes}\limits^L}_{\Lambda G}\, R\Gamma_c(G \backslash \bfX,\pi_*\Lambda)
\simeq  R\Gamma_c(G \backslash \bfX,\Lambda \, {\mathop{\otimes}\limits^L}_{\Lambda G}\, \pi_*\Lambda).$$
It remains to check that the natural map $\Lambda {\mathop{\otimes}\limits^L}_{\Lambda G}\, \pi_* \Lambda  \longrightarrow \Lambda$ is an isomorphim of sheaves. The stalk of $\pi_* \Lambda$ at a point $x$ is the permutation module $\Lambda G/\mathrm{Stab}_G(x)$, which is projective by assumption on the order of $\mathrm{Stab}_G(x)$. Therefore the fact that the previous map is an isomorphism can be checked on the stalks with the usual tensor product.
\end{proof}

When $\Lambda$ is a field and $\ell \nmid |G|$ then $\Lambda G$ is a semisimple algebra and complexes of $\Lambda G$-modules are quasi-isomorphic to their cohomology (see \S\ref{sec:derex}). Furthermore, invariants and coinvariants are isomorphic as $\Lambda$-modules in that case and the previous theorem shows that 
$$ H_c^i(G\backslash \bfX,\Lambda) \simeq H_c^i(\bfX,\Lambda)^G \simeq  \Lambda \otimes_{\Lambda G} H_c^i(\bfX,\Lambda).$$

\smallskip

Assume until the end of this section that $\Lambda$ is either $\calO$ or $k$. Given $P \subset G$ an $\ell$-subgroup of $G$ and $V$ an $\ell$-permutation module, we denote by $\mathrm{Br}_P(V)$ the image of the invariants $V^P$ in the coinvariants $k \otimes_{\Lambda P} V$. It induces an additive functor on the homotopy category of $\ell$-permutation modules, which we will still denote by $\mathrm{Br}_P$. We refer to \cite[\S 27]{The} for basic results on $\ell$-permutation modules and the Brauer functor. 

\begin{theorem}[Rickard]\label{thm:brauerfunctor}
The inclusion $\bfX^P \hookrightarrow \bfX$ induces an isomorphism
$$ \mathrm{Br}_P\big(\underbrace{\widetilde{R}\Gamma_c(\bfX,\Lambda)}\big) \tikzmark{eq-brauer} \, \mathop{\longrightarrow}\limits^\sim
\, R\Gamma_c(\bfX^P,k)$$
in $\sfD^b(kN_G(P)\sfmod)$
\end{theorem}
\begin{tikzpicture}[remember picture,overlay]
\draw[->,>=latex]
  ([shift={(-29pt,-10pt)}]pic cs:eq-brauer) |- 
  ++(10pt,-10pt) 
  node[right,text width=6cm] 
    {\footnotesize representative in $\sfHo^b(\Lambda G\sfperm)$ (see \S\ref{sec:cohodef})
    };
\end{tikzpicture}

\begin{proof}[Sketch of proof]
Assume for simplicity that $P$ is a Sylow subgroup of $G$ and that $P \simeq \mathbb{Z}/\ell \mathbb{Z}$. We consider the closed subvariety of $\bfX$ defined by
$$ \bfX_\ell = \{x \in \bfX \text{ such that } \ell \mid |\mathrm{Stab}_G(x)|\}.$$
Then $\bfX_\ell \simeq G \times_{N_G(P)} \bfX^P$, and hence $R\Gamma_c(\bfX_\ell) \simeq \mathrm{Ind}_{N_G(P)}^G R\Gamma_c( \bfX^P,\Lambda)$. Now by Corollary \ref{cor:perfect}, the cohomology complex of $\bfX \smallsetminus \bfX_\ell$ is perfect therefore its image by the Brauer functor is zero. Using the distinguished triangle in Theorem \ref{thm:openclose} we deduce that 
$\mathrm{Br}_P R\Gamma_c(\bfX,\Lambda) \simeq \mathrm{Br}_P R\Gamma_c(\bfX_\ell,\Lambda)$ which in turn is isomorphic to $R\Gamma_c(\bfX^P,k)$. 

\smallskip
One can generalize this argument to any $\ell$-subgroup $P$ of $G$ by considering a filtration of $\bfX$ by subvarieties with respect to the size of the $\ell$-part of the stabilizer of points.
\end{proof}

\subsection{Trace formula}
Assume now that the quasi-projective variety $\bfX$ is defined over $\mathbb{F}_q$, and denote by $F : \bfX \longrightarrow \bfX$ the corresponding Frobenius endomorphism so that in particular $\bfX(\mathbb{F}_q) = \bfX^F$.
The Frobenius induces a quasi-isomorphism on the complexes $R\Gamma(\bfX,\Lambda)$ and $R\Gamma_c(\bfX,\Lambda)$. All the quasi-isomorphisms and triangles listed in \S\ref{sec:cohomologytools} are compatible with $F$.
The corresponding action on the $\ell$-adic cohomology groups can be computed partially from the number of $\mathbb{F}_q$-points of $\bfX$ as follows.

\begin{theorem}[Lefschetz trace formula \cite{SGA45}]\label{thm:lefschetz}
$$ \#\bfX(\mathbb{F}_q) = \sum_{i\in\mathbb{Z}} (-1)^i \mathrm{Tr}\big(F,H_c^i(\bfX, K)\big).$$
\end{theorem} 

From this theorem we can for example derive a formula for the Euler characteristic of $\bfX$, given by 
\begin{equation}\label{eq:euler}
 \sum_{i\in\mathbb{Z}} (-1)^i \dim H_c^i(\bfX, K) = - \lim_{t \to \infty} \sum_{n=1}^\infty \# \bfX(\mathbb{F}_{q^n}) t^n.
\end{equation}

\begin{exercise} Show the latter formula (hint: use the eigenvalues of $F$ (and $F^n)$ on $H_c^i(\bfX)$ and Theorem \ref{thm:lefschetz} for $F^n$).
\end{exercise}

\begin{example}
(a) For the affine space of dimension $n$, we have $H_c^\bullet(\mathbb{A}_n,K) \simeq K[-2n]$ and $\#\mathbb{A}_n(\mathbb{F}_q) = q^n$. Therefore $F$ acts on the cohomology of $\mathbb{A}_n$ by multiplication by $q^n$ (the same actually holds for $H_c^\bullet(\mathbb{A}_n,\Lambda)$).

\smallskip 

To take the action of $F$ into account, we will write $R\Gamma_c(\mathbb{A}_n,\Lambda) \simeq \Lambda[-2n](n)$, and $(n)$ will be referred to as a \emph{Tate twist}. With this notation, the $F$-equivariant form of
Poincar\'e-Verdier duality (see Theorem \ref{thm:poincareverdier}) is 
$$R\Gamma(\bfX,\Lambda)[2d](-d) \simeq \tikzmark{eq-verdier} R\mathrm{Hom}_\Lambda(R\Gamma_c(\bfX,\Lambda),\Lambda)$$
for $\bfX$ a smooth variety of pure dimension $d$.

\smallskip

\noindent (b) $H_c^\bullet(\mathbb{P}_1,\Lambda) \simeq \Lambda[0] \oplus \Lambda[-2](1)$ and $\#\mathbb{P}_1(\mathbb{F}_q) = 1+q$.

\smallskip

\noindent (c) $H_c^\bullet(\mathbb{G}_m,\Lambda) \simeq \Lambda[-1] \oplus \Lambda[-2](1)$ and $\#\mathbb{G}_m(\mathbb{F}_q) = -1+q$.
\end{example}

Formula \eqref{eq:euler} can be extended to the case of a group action. Assume that the action of $G$ on $\bfX$ is $F$-equivariant. Then the virtual character of the representation of $G$ afforded by the $\ell$-adic cohomology groups is
\begin{equation}\label{eq:eulerg}
 \sum (-1)^i \mathrm{Tr}\big(g, H_c^i(\bfX, K)\big) = - \lim_{t \to \infty} \sum_{n=1}^\infty \# \bfX^{gF^n} t^n.
\end{equation}
Since this value is both an algebraic integer (left-hand side) and a rational number independent of $\ell$, this shows in particular that it is an integer independent of $\ell$. Note however that the individual cohomology groups could depend on $\ell$, but it was proved recently that this is not the case for Deligne--Lusztig varieties (see \cite{Rou17}).

\section{Deligne--Lusztig varieties and their cohomology}\label{chap:dlvarieties}

This chapter presents the construction by Deligne--Lusztig of algebraic varieties acted on by finite reductive groups \cite{DelLus}. We discuss several properties of the cohomology complexes of these varieties which we will use to deduce representation-theoretic results for finite reductive groups in the following chapters.

\smallskip

Starting from this chapter, $\bfG$ will denote a connected reductive algebraic group defined over $\overline{\mathbb{F}}_p$, and $F : \bfG \longrightarrow \bfG$ a Frobenius endomorphism defining an $\mathbb{F}_q$-structure on $\bfG$. The group of fixed points $\mathbb{G}(\mathbb{F}_q) :=  \bfG^F$ is a \emph{finite reductive group}. 
Given a closed $F$-stable subgroup $\mathbb{H}$ of $\bfG$, the corresponding finite group will be denoted by $H := \mathbb{H}(\mathbb{F}_q) = \mathbb{H}^F$. We refer to Meinolf Geck's lecture notes \cite{Gec17} for more on these finite groups (see also the textbooks \cite{Car,DigMic}).

\smallskip

The first section of this chapter serves as a motivation for the introduction of Deligne--Lusztig varieties as a generalization of Harish-Chandra induction. It is intentionally  very sketchy and will not be used in the rest of these notes. 

\subsection{Generalizing Harish-Chandra induction}\label{sec:hcdl}
The Harish-Chandra (or parabolic) induction and restriction functors provide an inductive approach to the construction of representations of finite reductive groups. Let $\bfP$ be an $F$-stable parabolic subgroup of $\bfG$. It has a Levi decomposition $\bfP = \bfL \ltimes \bfV$ where $\bfL$ is an $F$-stable Levi complement and $\bfV$ is the unipotent radical of $\bfP$. The finite set $(\bfG/\bfV)^F = G/V$ is endowed with a left action of $G$ by left multiplication, and a right action of $L$ by right multiplication (since $L$ normalizes $V$). We can therefore consider the adjoint pair of exact functors
$$R_{L\subset P}^G =  \Lambda G/V \otimes_{\Lambda L} - \quad \text{and} \quad
{}^*R_{L\subset P}^G =  \mathrm{Hom}_{\Lambda G }(\Lambda G/V , -) $$
between the categories  $\Lambda L\sfmod$ and  $\Lambda G\sfmod$. Howlett--Lehrer showed in \cite{HowLeh94} that these functors depend only on $L$ and not on $P$, up to isomorphism. Therefore they will be simply denoted by $R_L^G$ and ${}^*R_L^G$. 

\smallskip

There are two issues when working with these functors. The first one is that not every representation occurs in a representation induced from a proper Levi subgroup (unlike the usual induction). The second problem is that an $F$-stable Levi subgroup $\bfL$ of $\bfG$ in not necessarily a Levi complement of an $F$-stable parabolic subgroup. Even though $L$ exists, the finite set $G/V$ might not. However, the variety $\bfG/\bfV$ does and one can consider the following subvariety 
$$\bfY_\bfV := \{g\bfV \in \bfG/\bfV \, \mid \, g^{-1} F(g) \in \bfV \cdot F(\bfV)\}$$
called the \emph{parabolic Deligne--Lusztig variety} associated with $\bfV$. As in the case of the set $G/V$, it has a left action of $G$ by left multiplication, and a right action of $L$ by right multiplication. Consequently, the cohomology complex $R\Gamma_c(\bfY_\bfV,\Lambda)$ is a bounded complex of $(G,L)$-bimodules and we can consider the triangulated functors 
$$R_{\bfL\subset \bfP}^\bfG =  R\Gamma_c(\bfY_\bfV,\Lambda)\, {\mathop{\otimes}\limits^L}_{\Lambda L}\,  - \quad \text{and} \quad
{}^*R_{\bfL\subset \bfP}^\bfG =  R\mathrm{Hom}_{\Lambda G }(R\Gamma_c(\bfY_\bfV,\Lambda) , -) $$
between the derived categories  $\sfD^b(\Lambda L\sfmod)$ and  $\sfD^b(\Lambda G\sfmod)$. They are called Deligne--Lusztig induction and restriction functors. When $F(\bfV) = \bfV$ (\emph{i.e.} when $\bfP$ is $F$-stable) then $\bfY_\bfV$ is just the finite set $(\bfG/\bfV)^F \simeq G/V$ and these functors coincide with Harish-Chandra induction and restriction functors. 
 
\smallskip

In these notes we will focus on the case where $\bfL$ is a torus. In that case $\bfL$ is $\bfG$-conjugate to a quasi-split torus $\bfT$ and $\bfV$ is determined by an element $\dot w$ in $N_\bfG(\bfT)$. The corresponding variety will be denoted by $\bfY(\dot w)$ and studied in the next sections. This case corresponds to the original definition of Deligne--Lusztig \cite{DelLus} which was later generalized to the parabolic case in \cite{Lus76}.  

\subsection{Bruhat decomposition}

We fix a pair $(\bfT,\bfB)$ where $\bfT$ is a maximal torus of $\bfG$ contained in a Borel subgroup $\bfB$. We assume that both $\bfT$ and $\bfB$ are $F$-stable (such pairs always exist and form a single $G$-conjugacy class). Such a torus is said to be \emph{quasi-split}. In that case $F$ acts on the Weyl group $W = N_\bfG(\bfT)/\bfT$ of $\bfG$. 

\smallskip

Given $w \in W$ we define
\begin{itemize}
 \item[$\bullet$] $\ell(w) = \dim \bfB w\bfB - \dim \bfB = \dim \bfB w \bfB/\bfB$, the \emph{length of $w$} with respect to $\bfB$,
 \item[$\bullet$] $S = \{w \in W \mid \ell(w) = 1\}$ the set of \emph{simple reflections} of $W$ with respect to~$\bfB$.
\end{itemize}
 
\begin{theorem}[Bruhat decomposition {\cite[\S1]{DigMic}}]\leavevmode
\begin{itemize}
 \item[$\mathrm{(i)}$] $S$ generates $W$ and $(W,S)$ is a Coxeter system.
 \item[$\mathrm{(ii)}$] $\bfG$ decomposes as the disjoint union of \emph{Bruhat cells}
  $$\bfG = \displaystyle \bigsqcup_{w \in W} \bfB w \bfB.$$
 \item[$\mathrm{(iii)}$] $\bfB s \bfB w \bfB = \left\{\hskip-1.3mm \begin{array}{l} \bfB sw \bfB \text{ if } \ell(sw) > \ell(w) \\
 									\bfB w \bfB \sqcup \bfB sw\bfB \text{ otherwise.} \end{array}\right.$
 \item[$\mathrm{(iv)}$] The \emph{Schubert cell} $\bfB w \bfB / \bfB$ is isomorphic to $\mathbb{A}_{\ell(w)}$ (the affine space of dimension $\ell(w)$).
\end{itemize}
\end{theorem}

Note that $\ell(w)$ coincides with the length corresponding to the Coxeter system $(W,S)$. Indeed, $\dim (\bfB s \bfB w \bfB) 
\leq \dim (\bfB s \bfB \times_\bfB \bfB w \bfB) = \ell(s) + \ell(w) + \dim \bfB$, therefore the inequality $\ell(sw) > \ell(w)$ forces $\ell(sw) = \ell(w) + 1$ by (iii). More generally $\ell(w)$ is the smallest integer $r$ such that $w = s_1 s_2 \cdots s_r$ with $s_i \in S$.
 
\smallskip

The closure $\overline{\bfB w \bfB}$ in $\bfG$ of a Bruhat cell is a closed subvariety of $\bfG$ stable by left and right multiplication by $\bfB$. Therefore by (ii) it must be a finite union of Bruhat cells. We consider a partial order on $W$, called the \emph{Bruhat order}, defined by  $v \leq w$ if $\overline{\bfB v \bfB} \subset \overline{\bfB w \bfB}$  (or equivalently $\bfB v \bfB \subset\overline{\bfB w \bfB}$). Then by (ii) we have
$$\overline{\bfB w \bfB} = \bigsqcup_{v \leq w} \bfB v \bfB.$$
The singularities of these varieties are of considerable interest for the study of representations of semisimple Lie algebras and reductive groups, not only for finite reductive groups.

\begin{example}
(a) For the trivial element of $W$ we have $\bfB 1 \bfB = \bfB$, which is a closed subvariety of $\bfG$. Therefore $1$ is the unique minimal element for the Bruhat order.
\smallskip

\noindent (b) The variety $\bfG$ is irreducible, therefore there exists a unique $w_0 \in W$ such that $\overline{\bfB w_0 \bfB} = \bfG$. The element $w_0$ is the unique element of maximal length in $W$, and its length equals the dimension of the flag variety $\bfG/\bfB$, which is the number of positive roots of $W$. The element $w_0$ is also the unique maximal element for the Bruhat order.
\end{example} 

\begin{example} 
For $\bfG = \mathrm{GL}_2(\overline{\mathbb{F}}_p)$ we have $\bfG = \bfB \sqcup \bfB s \bfB$, hence 
 $$\bfB s \bfB = \bfG \smallsetminus \bfB = \left\{ \left(\begin{array}{cc} * & * \\ \lambda & * \end{array}\right) \mid \lambda \neq 0 \right\}.$$
\end{example} 

 \subsection{Deligne--Lusztig varieties}
Let $\bfU = R_u(\bfB)$ be the unipotent radical of the Borel subgroup $\bfB$.
We fix a set $\{\dot w\}_{w\in W}$ of representatives of $W$ in $N_\bfG(\bfT)$. The \emph{Deligne--Lusztig varieties} attached to $w$ are 
$$\begin{aligned} 
& \bfY_\bfG(\dot w) = \bfY(\dot w) = \{g\bfU \in \bfG/\bfU \, \mid \, g^{-1}F(g) \in \bfU \dot w \bfU\}, \\
& \bfX_\bfG( w) = \bfX( w) = \{g\bfB \in \bfG/\bfB \, \mid \, g^{-1}F(g) \in \bfB  w \bfB\}. \end{aligned}$$
The finite group $G = \bfG^F$ acts by left multiplication on both $\bfX(w)$ and $\bfY(\dot w)$. Furthermore, $\bfT^{\dot w F}$ acts by right multiplication on $\bfY(\dot w)$. Indeed, if $g^{-1}F(g) \in \bfU \dot w \bfU$ then using that $\bfT$ normalizes $\bfU$ we have, for every $t\in \bfT$ 
$$(gt)^{-1} F(gt) = t^{-1}g^{-1} F(g) F(t) \in \bfU t^{-1} \dot w F(t) \bfU.$$
Now $t \in \bfT^{\dot wF}$ if and only if ${}^{\dot wF}t = t$ which we can rewrite as $t^{-1} \dot w F(t) = \dot w$. 

\smallskip

Using properties of Schubert cells, one can prove that both $\bfX(w)$ and $\bfY(\dot w)$ are smooth quasi-projective varieties of pure dimension $\ell(w)$. Furthermore, the canonical projection $\pi : \bfG/\bfU \longrightarrow \bfG/\bfB$ induces a $G$-equivariant isomorphism
\begin{equation}
\label{eq:xandy} 
\bfY(\dot w) /\bfT^{\dot w F} \mathop{\longrightarrow}\limits^\sim \bfX(w).
\end{equation}
As in the previous chapter we can consider the cohomology complexes attached to these varieties. The complex $R\Gamma_c(\bfY(\dot w),\Lambda)$ is a bounded complex of finitely generated $(\bfG,\bfT^{\dot w F})$-bimodules, and $R\Gamma_c(\bfX( w),\Lambda)$ is a bounded complex of finitely generated $\bfG$-modules. Since $\bfT^{\dot w F}$ acts freely on $\bfY(\dot w)$, we deduce from Theorem \ref{thm:cohomologyinv} and \eqref{eq:xandy} that
\begin{equation}\label{eq:cohytox}
 R\Gamma_c(\bfX(w),\Lambda) \simeq R\Gamma_c(\bfY(\dot w),\Lambda) {\mathop{\otimes}\limits^L}_{\Lambda \bfT^{\dot w F}} \Lambda .
\end{equation}

\begin{example}\label{ex:x1}
If $\dot w = w =1$ then $\bfY(1) = \{g \bfU \, \mid \, g^{-1}F(g) \in \bfU\} = (\bfG/\bfU)^F$. The latter is just a finite set isomorphic to $G/U$. Similarly, $\bfX(1) \simeq G/B$ and therefore the cohomology complexes of $\bfY(1)$ and $\bfX(1)$ are given by a single permutation module in degree $0$, namely
$$ R\Gamma_c(\bfY(1),\Lambda) \simeq  \Lambda G/U[0] \quad \text{and} \quad R\Gamma_c(\bfX(1),\Lambda) \simeq  \Lambda G/B[0].$$
\end{example}

\begin{example}\label{ex:xs}
Let $\bfG = \mathrm{SL}_2(\overline{\mathbb{F}}_p)$ and $F$ be the \emph{standard} Frobenius of $\bfG$, raising the entries of a $2\times2$ matrix to the $q$th power, so that $G = \mathrm{SL}_2(q)$. The usual subgroups of $\bfG$ can be chosen as follows: 
$$ \bfT =  \left\{ \left(\begin{array}{cc} \lambda & 0 \\ 0 & \lambda^{-1} \end{array}\right) \right\} \subset \bfB = 
\left\{ \left(\begin{array}{cc} \lambda & * \\ 0 & \lambda^{-1} \end{array}\right) \right\}, \ 
\bfU = \left\{ \left(\begin{array}{cc} 1 & * \\ 0 & 1 \end{array}\right) \right\}, \ \dot s = \left(\begin{array}{cc} 0 & -1 \\ 1 & 0 \end{array}\right)$$
\noindent Then the varieties $\bfG/\bfU$ and $\bfG/\bfB$ are given explicitly by
$$ \begin{array}{rcl} \mathbb{A}_2\smallsetminus \{(0,0)\} & \simto & \bfG/\bfU \\
				(x,y) & {\longmapsto} &  \left(\begin{array}{cc} x &* \\ y & *  \end{array}\right)\bfU \\ 
  \end{array}
  \qquad \text{and} \qquad
  \begin{array}{rcl} \mathbb{P}_1 & \simto & \bfG/\bfB \\
				{[}x:y{]} & {\longmapsto} &  \left(\begin{array}{cc} x &* \\ y & *  \end{array}\right)\bfB \\ 
  \end{array}$$
and the cosets $\bfU \dot s \bfU$ and $\bfB s \bfB$ by
$$ \bfU \dot s \bfU = \left\{\left(\begin{array}{cc} * &* \\ 1 & *  \end{array}\right)\right\} \cap \mathrm{SL}_2
\quad \text{and} \quad \bfB \dot s \bfB = \left\{\left(\begin{array}{cc} * &* \\ \lambda & *  \end{array}\right) \mid \lambda\neq 0\right\} \cap \mathrm{SL}_2.$$
Finally, the element $g^{-1}F(g)$ is given by
$$ \left(\begin{array}{cc} x &* \\ y & *  \end{array}\right)^{-1} F\left(\begin{array}{cc} x &* \\ y & *  \end{array}\right)
 = \left(\begin{array}{cc} * &* \\ -y & x  \end{array}\right)\left(\begin{array}{cc} x^q &* \\ y^q & *  \end{array}\right)
 = \left(\begin{array}{cc} * &* \\ xy^q-yx^q & *  \end{array}\right).$$
We deduce the following explicit descriptions of the varieties $\bfY(\dot s)$ and $\bfX(s)$.
$$\begin{aligned}
\bfY(\dot s)& \, \simeq \{(x,y) \in \mathbb{A}_2 \, \mid \, xy^q-yx^q = 1\}, \\
\bfX( s)& \, \simeq \{[x,y] \in \mathbb{P}_1 \, \mid \, xy^q-yx^q \neq 0\} = \mathbb{P}_1 \smallsetminus \mathbb{P}_1(\mathbb{F}_q) .
\end{aligned}$$
The variety $\bfY(\dot s)$ is the famous \emph{Drinfeld curve}, discovered and studied by Drinfeld in \cite{Dri}.
\end{example}

Since $\mathrm{char}(K) = 0$, the algebra $KG$ is semisimple and the complex $R\Gamma_c(\bfY(\dot w),K)$ is quasi-isomophic to its cohomology $ \bigoplus H_c^i(\bfY(\dot w),K)[-i]$ as a complex of $(G,\bfT^{\dot w F})$-bimodules. Given an irreducible character $\theta$ of $\bfT^{\dot w F}$, we can consider the $\theta$-isotypic part of each individual cohomology groups $ H_c^i(\bfY(\dot w),K)_\theta := \mathrm{Hom}_{\bfT^{\dot w F}} (\theta, H_c^i(\bfY(\dot w),K))$. Since the actions of $G$ and $\bfT^{\dot w F}$ commute, $H_c^i(\bfY(\dot w),K)_\theta$ is a $G$-module. Given $g \in G$ we set
$$R_w(\theta)(g) = \sum_{i \in \mathbb{Z}} (-1)^i \mathrm{Tr}\big(g,H_c^i(\bfY(\dot w),K)_\theta\big).$$
The function $R_w(\theta)$ is the character of the virtual module $\sum (-1)^i H_c^i(\bfY(\dot w),K)_\theta$ (or equivalently of the complex $R\Gamma_c(\bfY(\dot w),K)_\theta$). It is called a \emph{Deligne--Lusztig character}. A particular interesting case is when $\theta = 1$. We have
$$R_w := R_w(1) = \sum_{i \in \mathbb{Z}} (-1)^i [H_c^i(\bfX( w),K)] = [R\Gamma_c(\bfX( w),K)]$$
since in that case the isotypic component of $ H_c^i(\bfY(\dot w),K)$ corresponding to the trivial representation of $\bfT^{\dot w F}$ is the invariant part under $\bfT^{\dot w F}$ which by \eqref{eq:cohytox} is isomorphic to $H_c^i(\bfX( w),K)$.

 \subsection{Properties of $R\Gamma_c(\mathbb{Y}(\dot w),\Lambda)$ and $R\Gamma_c(\mathbb{X}(w),\Lambda)$}

In the examples \ref{ex:x1} and \ref{ex:xs} the varieties $\bfY(\dot w)$ and $\bfX(w)$ are affine. This was proven in general by Deligne and Lusztig \cite{DelLus} when $q \geq h$ (the Coxeter number of $W$) and it is conjectured to hold unconditionaly (see for example \cite{BonRou08} for further examples). In any case, the consequences on the vanishing of the cohomology groups (see \S\ref{sec:cohodef}) hold.

\begin{theorem}[Lusztig \cite{Lus78}] 
$H_c^i(\bfY(\dot w),\Lambda) = H_c^i(\bfX(w),\Lambda) = 0$ for $i < \ell(w)$.
\end{theorem}

Consequently, the complexes $R\Gamma_c(\bfY(\dot w),\Lambda)$ and $R\Gamma_c(\bfX(w),\Lambda)$ can be represented by complexes with ($\Lambda$-free) terms in degrees $\ell(w),\ell(w)+1,\ldots,2\ell(w)$. In addition, one can compute the stabilizer of any point in $\bfY(\dot w)$ under the action of $G$ and invoke Corollary \ref{cor:perfect} to show the following additional properties.

\begin{prop}\label{prop:ywperfect}
The complex $\mathrm{Res}_G^{G\times \bfT^{\dot w F}} R\Gamma_c(\bfY(\dot w),\Lambda)$ is perfect. Furthermore,  if the order of $\bfT^{\dot w F}$ is invertible in $\Lambda$ then $R\Gamma_c(\bfX( w),\Lambda)$ is perfect as well.
\end{prop}

More generally, one can show that if  $Q$ in an $\ell$-subgroup of $G\times \bfT^{\dot w F}$ such that $\bfY(\dot w)^Q \neq \emptyset$, then $Q$ is necessarily conjugate to a diagonal subgroup of $G\times \bfT^{\dot w F}$. 

\begin{example}
Let $\bfG = \mathrm{SL}_2(\overline{\mathbb{F}}_p)$ and $\dot w = \dot s = \left(\begin{array}{cc} 0 & -1 \\ 1 & 0 \end{array}\right)$ so that
$$ \bfT^{\dot s F} \simeq \left\{ \left.\left(\begin{array}{cc} \lambda & 0 \\ 0 & \lambda^{-1} \end{array}\right) \right| \lambda^{q+1} =1 \right\} \simeq \mu_{q+1}(\overline{\mathbb{F}}_p).$$
If $(x,y) \in \bfY(\dot s)$ (see Example \ref{ex:xs}) then $g \cdot (x,y) \cdot \mathrm{diag}(\lambda,\lambda^{-1}) = (x,y)$ if and only if $(x,y)$ is an eigenvector of $g$ with eigenvalue $\lambda^{-1}$. In particular, either $\pm g$ is unipotent (with eigenvalues $\pm 1$) or $g$ is conjugate to $\mathrm{diag}(\lambda,\lambda^{-1})$.
\end{example}

If $|\bfT^{\dot w F}|$ is invertible in $\Lambda$ then there is no non-trivial $\ell$-subgroup $G\times \bfT^{\dot w F}$ such that $\bfY(\dot w)^Q \neq \emptyset$. In particular $R\Gamma_c(\bfY(\dot w),\Lambda)$ is perfect as a complex of bimodules in that case. Otherwise we have the following result (see \cite[\S9]{Alp} for basic results on vertices and sources). 

\begin{prop}
Let $C$ be a representative of $R\Gamma_c(\bfY(\dot w),k)$ as a complex of \linebreak$\ell$-permutation modules with no null-homotopic direct summand. Then the vertices of the terms of $C$ are contained in $\Delta \bfT^{\dot wF}$. 
\end{prop}

\begin{proof}
Let $P$ be an $\ell$-subgroup of $G$ which is not conjugate to a subgroup of $\Delta \bfT^{\dot w F}$. Then $\mathbf{Y}(\dot w)^P  =\emptyset$ and therefore by Theorem \ref{thm:brauerfunctor} we have $\mathrm{Br}_P(C) \simeq R\Gamma_c(\bfY(\dot w),k)^P \simeq 0$.

\smallskip
 
Now assume that there is an indecomposable direct summand $M$ in $C_i$ such that $\mathrm{Br}_P(M) \neq 0$. Without loss of generality we can assume that $P$ is maximal for this property. Then for any other direct summand $N$ of the terms of $C$, $\mathrm{Br}_P(N)$ is either zero or projective. Consequently $\mathrm{Br}_P(C)$ is an acyclic complex with projective terms.  Take $M$ to be in the largest degree $i$ of $C$ so that $\mathrm{Br}_P(C_j) = 0 $ for $j >i$ and $ \mathrm{Br}_P(C_{i-1}) \twoheadrightarrow \mathrm{Br}_P(M)$. Then $C_{i-1} \longrightarrow M$ is a split surjection (by \cite[Lem. A.1]{BonDatRou16}), which contradicts the minimality of $C$.
\end{proof}

\subsection{Applications}
The following result will be intensively used in the rest of these notes. It was first proved by Lusztig in the case where $\Lambda = K$ \cite{Lus78} and then extended by Bonnaf\'e--Rouquier to the modular setting. 

\begin{theorem}[Bonnaf\'e--Rouquier \cite{BonRou03}]\label{thm:mainBR}
Let $M$ be a simple $\Lambda G$-module and $w \in W$ be minimal for the Bruhat order such that 
$R\mathrm{Hom}_{\Lambda G}\big(R\Gamma_c(\bfY(\dot w),\Lambda), M\big)\neq 0.$ Then there exists a representative $0 \longrightarrow P_0 \longrightarrow \cdots \longrightarrow P_{\ell(w)} \longrightarrow 0$ of $R\Gamma_c(\bfY(\dot w),\Lambda)$ such that 
\begin{itemize}
\item[$\bullet$] each $P_i$ is a finitely generated projective $\Lambda G$-module (in degree $\ell(w)+i$),
\item[$\bullet$] $P_M$ is a direct summand of $P_i$ for $i = 0$ only (middle degree).
\end{itemize}
\end{theorem}
\noindent In other words, the terms of that complex satisfy $\mathrm{Hom}_{\Lambda G}(P_i,M) \neq 0 \iff i=0$. 

\begin{proof}[Sketch of proof]
The key property shown by Bonnaf\'e--Rouquier is that the cone of the natural map $R\Gamma_c(\bfY(\dot w),\Lambda) \longrightarrow  R\Gamma(\bfY(\dot w),\Lambda)$ lies in the thick subcategory of $\sfD^b(\Lambda G\sfmod)$ generated by the complexes $R\Gamma_c(\bfY(\dot v),\Lambda)$ for $v < w$. In particular, the minimality of $w$ show that this map induces an isomorphism
$$ R\mathrm{Hom}_{\Lambda G}(R\Gamma(\bfY(\dot w),\Lambda),M)\tikzmark{eq-rg} \, \simto \, R\mathrm{Hom}_{\Lambda G}(R\Gamma_c(\bfY(\dot w),\Lambda),M)\tikzmark{eq-rgc}.$$
\begin{tikzpicture}[remember picture,overlay]
\draw[->,>=latex]
  ([shift={(-50pt,-7pt)}]pic cs:eq-rg) |- 
  ++(10pt,-18pt) 
  node[right,text width=3cm] 
    {\footnotesize terms in degrees\\[-4pt]$0,1,\ldots,\ell(w)$
    };
\draw[->,>=latex]
  ([shift={(-50pt,-7pt)}]pic cs:eq-rgc) |- 
  ++(10pt,-18pt) 
  node[right,text width=3cm] 
    {\footnotesize terms in degrees\\[-4pt]$\ell(w),\ldots,2\ell(w)$
    };
\end{tikzpicture}
\vskip0.7cm

\noindent
Consequently, the cohomology of $R\mathrm{Hom}_{\Lambda G}(R\Gamma_c(\bfY(\dot w),\Lambda),M)$ vanishes outside the degree $\ell(w)$. In other words, $\mathrm{Hom}_{\sfHo^b(\Lambda G\sfmod)}(R\Gamma_c(\bfY(\dot w),\Lambda),M[-i]) = 0$ for $i \neq \ell(w)$.

\smallskip

Now let $0 \longrightarrow P_0 \longrightarrow \cdots \longrightarrow P_{\ell(w)} \longrightarrow 0$ be a representative of $R\Gamma_c(\bfY(\dot w),\Lambda)$ with projective terms and with no null-homotopic direct summand. Let $i$ be maximal for the property that $P_M$ is a direct summand of $P_i$, and assume that $i > 0$. If the composition $P_{i-1} \longrightarrow P_i \twoheadrightarrow M$ is zero, then the map $P_i \longrightarrow M$ induces a morphism between the complexes $R\Gamma_c(\bfY(\dot w),\Lambda)$ and $M[-i-\ell(w)]$. Therefore it must be null-homotopic by the above argument, which is impossible since $P_M$ is not a direct summand of $C_{i+1}$. Therefore the composition $P_{i-1} \longrightarrow P_i \twoheadrightarrow M$ must be non-zero, hence surjective, which shows that $P_M \simto P_M$ is a (null-homotopic) direct summand of $C$. By assumption of $C$, this is again impossible.  
\end{proof}

\begin{cor}[Bonnaf\'e--Rouquier \cite{BonRou03}]\label{cor:generation}
The triangulated category of perfect complexes $\Lambda G\sfperf$ is generated by the cohomology complexes $R\Gamma_c(\bfY(\dot w), \Lambda)$ for $w \in W$.
\end{cor}

\begin{proof}
We show by induction on the length of $w$  that  for every simple $\Lambda G$-module $M$, if $R\mathrm{Hom}_{\Lambda G}(R\Gamma_c(\bfY(\dot w),\Lambda),M) \neq 0$ then $P_M$ lies in the thick subcategory of $\Lambda G\sfperf$ generated by the complexes $R\Gamma_c(\bfY(\dot v),\Lambda)$ for $v \leq w$. This is true for $w = 1$ since $R\Gamma_c(\bfY(1),\Lambda) \simeq \Lambda G/U[0]$ (by definition a thick subcategory is stable under direct summands).

\smallskip

Let $w \in W$, and consider a minimal representative $C$ of $R\Gamma_c(\bfY(\dot w),\Lambda)$ as a bounded complex of projective modules. It follows from Theorem \ref{thm:mainBR} that the indecomposable direct summands of $C_i$ for $i > \ell(w)$ already appear in the cohomology complexes $R\Gamma_c(\bfY(\dot v),\Lambda)$ for $v < w$. By induction $C_i$ lies in the thick subcategory of $\Lambda G\sfperf$ generated by the complexes $R\Gamma_c(\bfY(\dot v),\Lambda)$ for $v < w$. But the term in middle degree can be written as
$$C_{\ell(w)} = \mathrm{Cone}\big((C_{\ell(w)+1} \longrightarrow \cdots \longrightarrow C_{2\ell(w)}) \longrightarrow R\Gamma_c(\bfY(\dot w),\Lambda)[\ell(w)+1]\big)$$
which proves that $C_{\ell(w)}$ lies in the category generated by $R\Gamma_c(\bfY(\dot v),\Lambda)$ for $v \leq w$.
\smallskip

To conclude, it remains to show that any projective indecomposable module appears as a direct summand of a minimal representative of $R\Gamma_c(\bfY(\dot w),\Lambda)$ for some $w \in W$. It is enough to show it at the level of characters. This follows from the fact that the regular representation of $G$ is \emph{uniform}, which means that it is a linear combination of Deligne--Lusztig characters $R_w(\theta)$.
\end{proof}

One can invoke Corollary \ref{cor:generation} to see that a morphism $f$ of bounded complexes of $\Lambda G$-modules is a quasi-isomorphism if and only if $\mathrm{Cone}(f) \otimes_{\Lambda G} R\Gamma_c(\bfY(\dot w),\Lambda) = 0$ for all $w \in W$. This was a key step in Bonnaf\'e--Rouquier's proof of the Jordan decomposition as a Morita equivalence (see \cite{BonRou03}).

\smallskip

The analogue of Corollary \ref{cor:generation} for general bounded complexes (whose terms can have non-trivial vertices) was proved recently by Bonnaf\'e--Dat--Rouquier in \cite{BonDatRou16}. This again was proven essential to show that the Jordan decomposition is a splendid Rickard equivalence. Recall that $\widetilde{R}\Gamma_c(\bfY(\dot w),k)$ denotes the (unique up to homotopy equivalence) representative of $R\Gamma_c(\bfY(\dot w),k)$ as a complex of $\ell$-permutation modules (see \S\ref{sec:cohodef}). 

\smallskip

We say that the prime number $\ell$ is \emph{very good} for $\bfG$ if $\ell$ is good for every simple component of $\bfG$ and $\ell \nmid m+1$ for every component of $\bfG$ of type $A_m$. A sufficient condition for $\ell$ to be very good is $\ell > h$ where $h$ is the Coxeter number of $\bfG$. 

\begin{theorem}[Bonnaf\'e--Dat--Rouquier \cite{BonDatRou16}] Assume that $\ell$ is very good. Let $\mathcal{X}$ be the set of complexes $\widetilde{R}\Gamma_c(\bfY(\dot w),k)\otimes_{k Q} \theta$ where $Q$ runs over the $\ell$-subgroups of $\bfT^{\dot w F}$, $\theta \in \mathrm{Irr}kQ$ and $w \in W$. Then
\begin{itemize}
 \item[$\mathrm{(i)}$]  The thick subcategory of $\sfHo^b(kG\sfmod)$ generated by $\mathcal{X}$ coincide with  $\sfHo^b(\mathcal{B})$, where $\mathcal{B}$ is the additive category generated by the indecomposable modules with one-dimensional sources and abelian vertices.
 \item[$\mathrm{(ii)}$] The image of $\mathcal{X}$ in $\sfD^b(kG\sfmod)$ generates $\sfD^b(kG\sfmod)$ as a triangulated category closed under direct summands.
\end{itemize}
\end{theorem}

\section{Decomposition numbers from Deligne--Lusztig characters}
\label{chap:dec}

Recall that $(K,\calO,k)$ denotes an $\ell$-modular system with $K$ being a finite extension of $\mathbb{Q}_\ell$. We will furthermore assume from this chapter on that $K$ and $k$ are big enough for all the groups considered (all the group algebras over $K$ and $k$ will split).

\smallskip

The purpose of this chapter is to explain how one can compute $\ell$-decomposition numbers for finite reductive groups using the Deligne--Lusztig characters $R_w(\theta)$ introduced in the previous chapter.
We start by recalling elementary results on decomposition matrices (see for example \cite[\S 14]{Ser}) before explaining the case of finite reductive groups. 

\subsection{Grothendieck groups and dualities}\label{sec:dualities}
Let $\mathcal{A}$ be an abelian (resp. additive) category. The \emph{Grothendieck group} of $\calA$, denoted by $K_0(\calA)$ (or sometimes $[\calA]$) is the abelian group generated by the isomorphism classes of objects of $\calA$ subject to the relations $[M] = [L]+[N]$ for every short exact sequence (resp. split short exact sequence) $0 \longrightarrow L \longrightarrow M\longrightarrow N \longrightarrow 0$.
Note that any abelian category is additive. When there is a risk of confusion, the Grothendieck group of $\calA$ as an additive category will be referred to as the \emph{split Grothendieck group}. 
If $\calA$ is an abelian category in which every object has finite composition length, then $K_0(\calA)$ can be identified with the free abelian group with basis $\mathrm{Irr}\, \calA$. Similarly, the split Grothendieck group of a Krull-Schmidt category is the free abelian group generated by the isomorphism classes of the indecomposable objects.
 
\smallskip

Let $\calT$ be a triangulated category. The \emph{Grothendieck group} $K_0(\calT)$ of $\calT$ is the abelian group generated by the isomorphism classes of objects of $\calT$ subject to the relations $[M] = [L]+[N]$ for every distinguished triangle $ L \longrightarrow M\longrightarrow N\rightsquigarrow $.
\smallskip

Given an abelian category $\calA$, the fully faithful functor $\calA \longrightarrow \sfD^b(\calA)$ induces an isomorphism $K_0(\calA) \simto K_0(\sfD^b(\calA))$. Under this identification, the class of a bounded complex $C$ is 
$$[C] = \sum_{i \in \mathbb{Z}} (-1)^i[C_i] = \sum_{i \in \mathbb{Z}} (-1)^i [H^i(C)].$$
Similarly, for any additive category $\calA$ the functor $\calA \longrightarrow \sfHo^b(\calA)$ induces an isomorphism between the corresponding Grothendieck groups.

\smallskip

The Grothendieck groups we will be interested in in this chapter are:\begin{tikzpicture}[remember picture,overlay]
\draw [decoration={brace,amplitude=0.5em},decorate,black]
  ([shift={(0pt,7pt)}]pic cs:item-groth) |- 
  ++(5pt,-17pt);
\draw ([shift={(7pt,-4pt)}]pic cs:item-groth) node[right,text width=6cm] 
    {\footnotesize$\mathbb{Z}$-basis given by simple modules
    }
;
\end{tikzpicture}
\begin{itemize} 
 \item[$\bullet$] $K_0(KG\sfmod)$ \tikzmark{item-groth}
 \item[$\bullet$] $K_0(kG\sfmod)$
 \item[$\bullet$] $K_0(kG\sfproj)  \, \leftarrow$\, {\footnotesize basis given by projective indecomposable modules (PIMs)}
\end{itemize}
\noindent Since exact sequences split in $KG\sfmod$ and $kG\sfproj$ then for $M$, $N$ in $KG\sfmod$ (resp. 
$kG\sfproj$) $[M] = [N]$ if and only if $M\simeq N$. This is not true in general in $kG\sfmod$ since $kG$-modules can have non-trivial extensions when $\ell$ divides the order of $|G|$ (which is the interesting case from our point of view).

\smallskip

In addition, there are perfect pairings 
 $$ \begin{aligned}
 \langle - ; - \rangle_K & : K_0(KG\sfmod) \times K_0(KG\sfmod) \longrightarrow \mathbb{Z} \\
 \langle - ; - \rangle_k & : K_0(kG\sfproj) \times K_0(kG\sfmod) \longrightarrow \mathbb{Z}
 \end{aligned}$$
defined by $ \langle [P]; [M] \rangle_\Lambda = \dim_\Lambda \mathrm{Hom}_{\Lambda G}(P,M)$ when $P$ and $M$ are actual modules and $\Lambda$ is the field $K$ or $k$. Then $\mathrm{Irr}_K G$ is a self-dual basis for the pairing $\langle - ; - \rangle_K$, whereas a dual basis of $\mathrm{Irr}_kG$ for $\langle - ; - \rangle_k$ is given by the classes of projective indecomposable modules.

\subsection{Lifting projective modules} 
We say that a $kG$-module $M$ \emph{lifts to characteristic zero} if there exists an $\calO G$-lattice $\widetilde M$ (an $\calO G$-module which is free as an $\calO$-module) such that $k\widetilde M \simeq M$ as $kG$-modules. Not every $kG$-module can be lifted to characteristic zero in general, unless $\ell \nmid |G|$ in which case $kG$ is semisimple. This holds nevertheless for projective modules. Indeed, given a finitely generated projective $kG$-module $P$, we can consider the projective cover $P_m$ of $P$ as an $\calO/\ell^m \calO$-module. Then $kP_m \simeq P$ and the $\calO G$-module $\widetilde P :=\mathop{\lim}\limits_{\longleftarrow} P_m$ is an $\calO G$-module lifting $P$. In addition, it is projective and it is $-$ up to isomorphism $-$ the unique projective $\calO G$-module lifting $P$. Note that if $M$ is a simple $kG$-module, then $\widetilde P_M$, together with the composition $\widetilde P_M \twoheadrightarrow P_M \twoheadrightarrow M$ is a projective cover of $M$, viewed as a simple $\calO G$-module.

\smallskip

Given $P\in kG\sfproj$ and its lift $\widetilde P$ to characteristic zero, we can form the $KG$-module $K\widetilde P$. Its character (or rather its image in the Grothendieck group) will be denoted $e([P])$. This defines a group homomorphism 
$$e : K_0(kG\sfproj) \longrightarrow K_0(KG\sfmod).$$

\subsection{Decomposing ordinary characters}
Let $M$ be a $KG$-module. One can choose an $\calO$-free $\calO$-submodule $\widetilde M$ such that $K\widetilde M \simeq M$. By saturating by the action of $G$ one can assume that $\widetilde M$ is stable by $G$, so that $\widetilde M$ is an $\calO G$-lattice such that $K\widetilde M \simeq M$. Then one can form the $kG$-module $k\widetilde M$ and consider its image in $K_0(kG\sfmod)$, which we will denote by $d([M])$. 

\begin{prop}\label{prop:brauerrep}\leavevmode
\begin{itemize}
  \item[$\mathrm{(i)}$] $d$ is well-defined and extends to a group homomorphism
  $$d : K_0(KG\sfmod) \longrightarrow K_0(kG\sfmod)$$
called the \emph{decomposition map}.
  \item[$\mathrm{(ii)}$] \emph{(Brauer reciprocity)} $d$ is the transpose of $e$ for the pairings defined in \S\ref{sec:dualities}. In other words
   $$ \langle - ; d(-)\rangle_k = \langle e(-) ; -\rangle_K.$$
\end{itemize}
\end{prop}

\begin{proof}
It is enough to prove (ii). Let $M$ be a $KG$-module and $P$ be a projective $kG$-module. We constructed $\calO G$-lattices $\widetilde M$ and $\widetilde P$ such that $K\widetilde M \simeq M$ and $k\widetilde P \simeq P$. Then 
$$\begin{aligned}
 \langle [P];d([M])\rangle_k & = \dim_k \mathrm{Hom}_{kG}(P,k\widetilde M) = \mathrm{rk}_\calO \mathrm{Hom}_{\calO G}(\widetilde P,\widetilde M) \\ & = \dim_K \mathrm{Hom}_{KG}(K\widetilde P,K\widetilde M) = 
\langle e([P]);[M]\rangle_K. \end{aligned}$$
\end{proof}

The \emph{decomposition matrix} $D$ (or $D_\ell$) is the matrix with entries
 $$d_{\chi,S} = \langle [P_S];d(\chi)\rangle_k = \langle e([P_S]);\chi \rangle_K $$
for $\chi \in \mathrm{Irr}_K G$ and $S \in \mathrm{Irr}_k G$. With this notation we have
$$ d(\chi) = \sum_{S \in \mathrm{Irr}_k G} d_{\chi,S} [S] \quad \text{and} \quad e([P_S])
= \sum_{\chi \in \mathrm{Irr}_K G} d_{\chi,S} \chi.$$

\begin{example}
(a) If $\ell \nmid |G|$ then every simple $kG$-module is projective, and lifts to an irreducible ordinary character. Consequently $D$ is the identity matrix up to reordering.
\smallskip

\noindent (b) If $G$ is an $\ell$-group, then the only irreducible $kG$-module is the trivial representation.  Since the decomposition map preserves the dimension, we deduce that $d(\chi) = (\dim \chi) [k]$ for every irreducible ordinary character $\chi$ of $K$. The decomposition matrix in that case is a column encoding the dimensions of the irreducible $KG$-modules.

\smallskip

Dually, the projective cover of the trivial representation is the regular representation $P_k = kG$, which lifts to characteristic zero as $\calO G$, and whose character is $[KG] = \sum (\dim \chi) \chi$. This is an example of Brauer reciprocity as stated in Proposition \ref{prop:brauerrep}.ii.

\smallskip
\noindent (c) Let us consider the particular case of $G =\mathfrak{S}_3$ and $\ell=3$. There are three irreducible representations over $K$: the trivial representation $K$, the sign $\varepsilon$ and the reflection representation, of dimension $2$. The latter has an integral version given by $M = \{(x_1,x_2,x_3) \in \calO^3 \, \mid \, x_1 + x_2 +x_3 = 0\}.$ The representations $K$ and $\varepsilon$ yield two non-isomorphic representations over $k$ by $\ell$-reduction, the trivial representation $k$ and the sign modulo $3$, which we still denote by $\varepsilon$. There is a short exact sequence 
 $$ 0 \longrightarrow \hskip -1.3mm \begin{array}[t]{l} k \longrightarrow kM \\ x \longmapsto (x,x,x) \end{array} \hskip -8mm \longrightarrow \varepsilon \longrightarrow 0$$
which shows that $[kM] = [k] + [\varepsilon]$. We deduce that the decomposition matrix in that case is 
$$ D = \left[ \begin{array}{cc} 1 & \cdot \\ \cdot & 1\\ 1 & 1 \\\end{array}\right]\cdot$$
Consequently, the two PIMs have characters $[K] + [KM]$ and $[\varepsilon] + [KM]$.
\end{example}

\subsection{Basic sets of characters}
We mentioned in a previous section that not every $kG$-modules can be lifted to characteristic zero. This is however true at the level of the Grothendieck groups: the class of any $kG$-module is the $\ell$-reduction of a virtual character.

\begin{theorem}The decomposition map $d : K_0(KG\sfmod) \longrightarrow K_0(kG\sfmod)$ is surjective.
\end{theorem}

By Proposition \ref{prop:brauerrep}, the map $e$ is the transpose of $d$, hence it is injective. Therefore if $P$ and $Q$ are two projective $kG$-modules then $P \simeq Q$ if and only if their lifts $\widetilde P$ and $\widetilde Q$ have the same character. In other words, a projective module is determined by its character over $K$.

\smallskip

Since $d$ is surjective, it is natural to search for a set $\mathcal{B}$ of ordinary irreducible characters such that $d(\mathcal{B})$ is a $\mathbb{Z}$-basis of $K_0(kG\sfmod)$. Such a set is called a \emph{basic set} (see \cite{GecHis91}). 
If it exists, then the decomposition matrix has the following shape:
$$D = \left[ \begin{array}{c} D_\mathcal{B} \\ * \\ \end{array}\right] \quad \text{with } \ D_\mathcal{B} \in \mathrm{GL}_n(\mathbb{Z}).$$
In that case, a projective $kG$-module is determined by the projection of its character on $\mathcal{B}$.

\smallskip

Now assume that $G = \bfG^F$ is a finite reductive group. Recall that the \emph{unipotent characters} are the irreducible constituents of the virtual characters
$$ R_w = \sum_{i \in \mathbb{Z}} (-1)^i [H_c^i(\bfX(w),K)] = [R\Gamma_c(\bfX(w),K)]$$
for $w \in W$. The \emph{unipotent blocks} are the $\ell$-blocks containing at least one unipotent character,
and the irreducible characters in the union of unipotent blocks are the constituents of 
$$ R_w(\theta) = \sum_{i \in \mathbb{Z}} (-1)^i [H_c^i(\bfY(\dot w),K)_\theta] = [R\mathrm{Hom}_G\big(\theta,R\Gamma_c(\bfY(w),K)\big)]$$
for $w \in W$ and $\theta \in \mathrm{Irr}_\ell \bfT^{\dot w F}$ an irreducible $\ell$-character of $\bfT^{\dot w F}$.

\begin{exercise}\label{exo:drwtheta}
Given $\theta \in \mathrm{Irr}_\ell \bfT^{\dot w F}$, show that $d(R_w(\theta)) = R_w(d(\theta)) = d(R_w)$ in
$K_0(kG\sfmod)$.\end{exercise}

Recall that $\ell$ is said to be very good for $\bfG$ if $\ell$ is good for every simple component of $\bfG$ and $\ell \nmid m+1$ for every component of $\bfG$ of type $A_m$. A sufficient condition for $\ell$ to be very good is $\ell > h$ where $h$ is the Coxeter number of $\bfG$.

\begin{theorem}[Geck--Hiss, Geck \cite{GecHis91,Gec93}]\label{thm:basicset}
Assume that $\ell$ is very good. Then the unipotent characters form a basic set for the union  $\mathcal{E}_\ell(G,1)$ of unipotent blocks. 
\end{theorem}

More generally, under the same assumption on $\ell$, given $s \in G^*$ a semisimple $\ell$-element, the Lusztig series $\mathcal{E}(G,s)$ is a basic set for the union of blocks $\mathcal{E}_\ell(G,s) := \bigcup \mathcal{E}(G,st)$ where $t$ runs over the set of semisimple $\ell$-elements of $C_{G^*}(s)$. 

\subsection{Decomposition numbers and Deligne--Lusztig characters}
This section is the core of this chapter and contains recent results on $\ell$-decomposition numbers for unipotent blocks when $\ell$ is not too small (see for example \cite{Dud13,DudMal15}). We start by listing the different tools and assumptions we are going to use to determine these numbers.

\smallskip

\noindent {\bf(HC)} The Harish-Chandra restriction/induction of a projective $\Lambda G$-module remains projective (this follows easily from the biadjointness and the exactness of the functors).

\smallskip

\noindent {\bf(Uni)} When $\ell$ is very good, the restriction of the decomposition matrix to the set of unipotent characters (a basic set by Theorem \ref{thm:basicset}), ordered by increasing $a$-function, has unitriangular shape. This is only conjectural (see \cite[Conj. 3.4]{GecHis96}). 

\smallskip

\noindent {\bf(Hecke)} The decomposition matrix of the Hecke algebra $\mathrm{End}_{\calO G} (\calO G/B)$ (corresponding to the unipotent principal series) embeds in the decomposition matrix of the finite group $G$. 

\smallskip

\noindent {\bf(Reg)} If $\theta \in \mathrm{Irr}_\ell \bfT^{\dot w F}$ is an ordinary irreducible $\ell$-character of $\bfT^{\dot w F}$ in general position (\emph{i.e.} $(-1)^{\ell(w)} R_w(\theta)$ is irreducible) then
 $$ \big\langle e([P]); (-1)^{\ell(w)} R_w \big\rangle_K \geq 0$$
for every projective $kG$-module $P$. This gives a non-trivial information since $(-1)^{\ell(w)}R_w$ is only a virtual unipotent character, even though $(-1)^{\ell(w)} R_w(\theta)$ is irreducible.

\begin{proof} Using the fact that $d(R_w(\theta)) = d(R_w)$ when $\theta \in \mathrm{Irr}_\ell \bfT^{\dot w F}$  (see Exercise \ref{exo:drwtheta}) and Brauer reciprocity we have
$$\begin{aligned}
\big\langle e([P]); R_w \big\rangle_K = &\,  \big\langle [P]; d(R_w) \big\rangle_k \\
 = &\, \big\langle [P]; d(R_w(\theta)) \big\rangle_k \\
 = &\, \big\langle e([P]); R_w(\theta) \big\rangle_K.
\end{aligned} $$
Since $(-1)^{\ell(w)} R_w(\theta)$ is assumed to be irreducible, the sign of this scalar product coincides with $(-1)^{\ell(w)}$. 
\end{proof}

Recall from Proposition \ref{prop:ywperfect} that the complex $R\Gamma_c(\bfY(\dot w),k)$ is perfect as a complex of $kG$-modules. Let $P_w = [R\Gamma_c(\bfY(\dot w),k)]$ denote its class in $K_0(kG\sfproj)$. Then  
\begin{equation}\label{eq:pw}
e(P_w) = [R\Gamma_c(\bfY(\dot w),K)] = \sum_{i \in \mathbb{Z}} (-1)^i [H_c^i(\bfY(\dot w),K)] = R_w +\,\text{non-unip. chars}.
\end{equation}
The following property is a character-theoretic consequence of Bonnaf\'e--Rouquier's theorem \ref{thm:mainBR}. 
It is particularly suited for determining decomposition numbers on cuspidal $kG$-modules, as we will see in the examples of the next section.
\smallskip

\noindent {\bf(DL)} Given a simple $kG$-module $S$, let $w \in W$ be minimal (for the Bruhat order) such that $\langle P_w,[S] \rangle \neq 0$ (\emph{i.e.} $[P_S]$ occurs in $P_w$). Then $\langle (-1)^{\ell(w)} P_w,[S] \rangle > 0$. 

\begin{proof} 
Let $S$ be a simple $kG$-module and $w\in W$ be such that $\langle (-1)^{\ell(w)} P_w,[S] \rangle < 0$. Take $C= 0 \longrightarrow P_0 \longrightarrow \cdots \longrightarrow P_{\ell(w)} \longrightarrow 0$ to be a reduced representative of $[R\Gamma_c(\bfY(\dot w),k)]$ with each $P_i$ projective. By assumption, there exists $i > 0$ such that $P_S$ is a direct summand of $P_i$. Taking $i$ to be maximal, we deduce that 
$$\begin{aligned}
0 \neq  &\,  \mathrm{Hom}_{\sfHo^b(kG\sfmod)} (C,S[-i-\ell(w)]) \\ 
= & \,  \mathrm{Hom}_{\sfD^b(kG\sfmod)} (R\Gamma_c(\bfY (\dot w),k),S[-i-\ell(w)]) \\
= & \,H^{i+\ell(w)} \big(R\mathrm{Hom}_{kG} (R\Gamma_c(\bfY (\dot w),k),S)\big). \end{aligned}$$
By Theorem \ref{thm:mainBR} if $v \leq w$ is minimal for the property that 
the complex \linebreak$R\mathrm{Hom}_{kG} (R\Gamma_c(\bfY (\dot v),k),S)$ is non-zero then $R\Gamma_c(\bfY (\dot v),k)$ has
a representative such that $P_S$ occurs only in middle degree. Therefore $\langle (-1)^{\ell(v)} P_v;[S] \rangle > 0$ and $w$ cannot be minimal for the property that $\langle P_w,[S] \rangle \neq 0$.
\end{proof}

\subsection{Examples in small rank}\label{sec:smallrank}
We discuss here three examples of small-rank finite reductive groups where the previous tools allow a complete determination of the (unipotent part of the) decomposition matrix. Here the assumption (Uni) was shown to hold by Geck \cite{Gec91} for finite linear and unitary groups  and  by White \cite{Whi90} for $\mathrm{Sp}_4(q)$.
\smallskip

\noindent (a) We start with $G = \mathrm{SL}_2(q)$. Its order is $|G| = q(q-1)(q+1)$. Assume that $\ell$ is an odd prime number with $\ell \nmid q$ and $\ell \mid q^2-1$. Then the principal block contains the two unipotent characters $1$ and $\mathrm{St}$, together with some non-unipotent characters. Using (Uni) and (Hecke) we have
$$ D = \left[\begin{array}{cc} 1 & \cdot \\ \alpha & 1 \\ * & * \\ \vdots & \vdots \end{array} \right]
\quad \text{with} \ \ \alpha = \left\{\hskip-1.3mm \begin{array}{l} 0 \text{ if } \ell \nmid q+1 \\
1 \text{ otherwise.} \end{array}\right.$$
Here the value of $\alpha$ can be obtained by (Hecke). Indeed, if $\ell \nmid q+1$ then $\mathrm{End}_{kG}(kG/B)$ is semisimple and $ kG/B = k \oplus M$ with $M$ a simple $kG$-module. Therefore $\alpha = 0$ in this case. Now if $\ell \mid q+1$ then $R_T^G(k) = kG/B$ is a PIM with character $1 + \mathrm{St}$, hence $\alpha = 1$. 

\smallskip

\noindent (b) Let $G = \mathrm{Sp}_4(q)$, whose order is $q^4(q^2-1)(q^4-1) = q^4 (q-1)^2(q+1)^2(q^2+1)$. 
We will denote by $s$ and $t$ the simple reflections in $W$ corresponding to the short and long simple root respectively. Assume in this example that $\ell$ is odd and $\ell \mid q+1$. Then the characters in the principal block are
$$\{\underbrace{1,\mathrm{St},\rho_1,\rho_2}\tikzmark{princ-series},\tikzmark{theta10}\theta_{10},\text{non-unipotent}\}$$
\begin{tikzpicture}[remember picture,overlay]
\draw[->,>=latex]
  ([shift={(-25pt,-12pt)}]pic cs:princ-series) |- 
  ++(-10pt,-8pt) 
  node[left,text width=3cm] 
    {\footnotesize in the principal series
    };
\draw[->,>=latex]
  ([shift={(5pt,-7pt)}]pic cs:theta10) |- 
  ++(10pt,-13pt) 
  node[right,text width=2cm] 
    {\footnotesize cuspidal
    };
\end{tikzpicture}
\vskip 3mm
There is another unipotent character, denoted by $\chi$, which under the assumptions on $\ell$ is of defect zero and this forms a block by itself. We have $R_1 = R_T^G(K) = 1 + \mathrm{St} + \rho_1 + \rho_2 + 2\chi$, giving all the unipotent characters lying in the principal series. The first approximation of the decomposition matrix is given by (Uni)
$$ D = \left[\begin{array}{ccccc} 1 & \cdot & \cdot &\cdot &\cdot  \\  
* & 1 & \cdot &\cdot &\cdot \vphantom{\rho_1}  \\
* &  * & 1 & \cdot &\cdot \vphantom{\rho_2}  \\
* & * &  * & 1 &\cdot \vphantom{\theta_{10}}  \\
* & * & * &  * & 1 \vphantom{\mathrm{St}}  \\\hline
* & * & * & * & * \\
\vdots & \vdots & \vdots & \vdots & \vdots \\
\end{array} \right] \begin{array}{l} 1  \\  
\rho_1 \\
\rho_2 \\
\theta_{10}  \\
\mathrm{St}  \\
\vphantom{*} \\
\vphantom{\vdots} \\
\end{array} $$
A (Hecke) argument shows that $R_T^G(k) = kG/B$ is indecomposable. Since it is projective, it gives the first column of the decomposition matrix. 

\smallskip

The second and third column can be obtained by (HC). First, let $L = \mathrm{GL}_2(q) \subset \mathrm{Sp}_4(q)$ and $P$ be the PIM of $\mathrm{GL}_2(q)$ such that $e([P]) = \mathrm{St}_{\mathrm{GL}_2(q)} +$ non-unipotent characters (see Example (a)). Let $b \in \calO G$ be the block idempotent corresponding to the principal block. Then $R_L^G(P)$ is projective and its character, cut by the block, is given by
$$ \begin{aligned}
e(bR_L^G([P])) = bR_L^G(e([P])) &\, = bR_L^G(\mathrm{St}_{\mathrm{GL}_2(q)}) + bR_L^G(\text{non-unipotent}) \\
&\, = \rho_1 + \mathrm{St} + \text{non-unipotent.} \end{aligned}
$$
Similarly with the Levi subgroup $L' = \mathrm{SL}_2(q) \times \mathbb{F}_q^\times \subset \mathrm{Sp}_4(q)$
we get
$$ e(bR_{L'}^G([P])) = bR_{L'}^G(e([P]))  = \rho_2 + \mathrm{St} + \text{non-unipotent}. $$
Consequently the unipotent part of the decomposition matrix is 
$$D_\text{unip} = \left[\begin{array}{ccccc} 1 & \cdot & \cdot &\cdot &\cdot  \\  
1 & 1 & \cdot &\cdot &\cdot  \\
1 &  \cdot & 1 & \cdot &\cdot  \\
\cdot & \cdot &  \cdot & 1 &\cdot \vphantom{\theta_{10}}  \\
 1 & \alpha_1 & \alpha_2 &  \beta & 1  \\
\end{array} \right]$$
with $\alpha_1,\alpha_2 \leq 1$. 

\smallskip

We use (Reg) to determine the exact value of $\alpha_1$ and $\alpha_2$. Since $|\bfT^{w_0 F}| = (q+1)^2$ there exists a non-trivial $\ell$-character $\theta \in \mathrm{Irr}_K\bfT^{w_0 F}$. Furthermore, if $(q+1)_\ell > 3$ then one can choose $\theta$ to be lying outside the reflection hyperplanes (in the reflection representation of $W$ on the group of characters of $\bfT$). In that case it is in general position, and (Reg) yields $\langle e([Q]);R_{w_0} \rangle \geq 0$ for every projective $kG$-module $Q$. Now $R_{w_0} = 1+\mathrm{St}-\rho_1-\rho_2 -2\theta_{10}$, so if we apply this to the projective indecomposable modules $Q_2$, $Q_3$ and $Q_4$ corresponding to the second, third and fourth columns of the decomposition matrix we get
$$ \begin{aligned}
& \langle e([Q_2]);R_{w_0} \rangle = \langle \rho_1 + \alpha_1 \mathrm{St}; 1+\mathrm{St}-\rho_1-\rho_2 -2\theta_{10}\rangle = -1+\alpha_1 \\
& \langle e([Q_3]);R_{w_0} \rangle = \langle \rho_2 + \alpha_2 \mathrm{St}; 1+\mathrm{St}-\rho_1-\rho_2 -2\theta_{10}\rangle = -1+\alpha_2 \\
& \langle e([Q_4]);R_{w_0} \rangle = \langle \theta_{10}+\beta \mathrm{St}; 1+\mathrm{St}-\rho_1-\rho_2 -2\theta_{10}\rangle = -2+\beta 
\end{aligned}$$
which gives $\alpha_1,\alpha_2 \geq 1$ (and hence $\alpha_1 = \alpha_2 =1$) and $\beta \geq 2$.

\smallskip

The final ingredient is (DL). To use it we decompose each virtual projective module $P_w = [R\Gamma_c(\bfY(\dot w),k)]$ on the basis of PIMs. To this end, recall from \eqref{eq:pw} that $e(P_w) = R_w + \text{non-unipotent characters}$. We have
$$\begin{aligned}
e(bP_1) = &\, 1+ \mathrm{St}+\rho_1 + \rho_2 + \text{non-unip.} = e([Q_1]) \\
e(bP_s) = &\, 1- \mathrm{St}+\rho_1 - \rho_2 + \text{non-unip.} \\
	= &\, 1+\mathrm{St}+\rho_1 + \rho_2 - 2(\rho_2 + \mathrm{St})+ \text{non-unip.} \\
	= &\, e([Q_1]-2[Q_3]) \\
e(bP_t) = &\, 1- \mathrm{St}-\rho_1 + \rho_2 + \text{non-unip.} \\
	= &\, 1+\mathrm{St}+\rho_1 + \rho_2 - 2(\rho_1 + \mathrm{St})+ \text{non-unip.} \\
	= &\, e([Q_1]-2[Q_2]) \\	
e(bP_{st}) = &\, 1+ \mathrm{St}+\theta_{10} + \text{non-unip.} \\
	= &\, (1+\mathrm{St}+\rho_1 + \rho_2) - (\rho_1 + \mathrm{St})- (\rho_2 + \mathrm{St})+ (\theta_{10}+2\mathrm{St})+ \text{non-unip.} \\
	= &\, e([Q_1]-[Q_2]-[Q_3]) + \theta_{10}+2\mathrm{St}+\text{non-unip.}  \\	
\end{aligned} $$
Since $\ell(st)=2$ we deduce from (DL) that $\theta_{10}+2\mathrm{St}+\text{non-unip.}$ must be a non-negative combination of $e([Q_4])$ and $e([Q_5])$, since $Q_4$ and $Q_5$ do not appear in the decomposition of $P_w$ for $w<st$. Writing 
$$\theta_{10}+2\mathrm{St} + \text{non-unip.} = e\big([Q_4]+(2-\beta)[Q_5]\big)$$
we deduce that $\beta \leq 2$, which forces $\beta = 2$. We conclude that the unipotent part of the $\ell$-decomposition matrix (when $\ell$ is odd and $(q+1)_\ell >3$) is given by
$$D_\text{unip} = \left[\begin{array}{ccccc} 1 & \cdot & \cdot &\cdot &\cdot  \\  
1 & 1 & \cdot &\cdot &\cdot  \\
1 &  \cdot & 1 & \cdot &\cdot  \\
\cdot & \cdot &  \cdot & 1 &\cdot \\
 1 & 1 & 1 &  2 & 1  \\
\end{array} \right].$$

\begin{exercise}
We follow the notation of \cite[\S13]{Car} for unipotent characters. 
Complete the determination of the $\ell$-decomposition matrix of $G = G_2(q)$ when $\ell \mid q+1$ and $\ell > 5$, which is given by
$$D_\text{unip} = \left[\begin{array}{cccccc} 1 & \cdot & \cdot &\cdot &\cdot &\cdot  \\  
1 & 1 & \cdot &\cdot &\cdot&\cdot \vphantom{\rho_1}  \\
1 &  \cdot & 1 & \cdot &\cdot &\cdot  \vphantom{\rho_2}\\
\cdot & \cdot &  \cdot & 1 &\cdot  &\cdot  \vphantom{G_2[1]} \\
\cdot & \cdot &  \cdot & \cdot &1  &\cdot  \vphantom{G_2[-1]} \\
 1 & 1 & 1 &  \alpha & \beta &1  \vphantom{\mathrm{St}} \\
\end{array} \right]
\begin{array}{l} 1 \\ \phi_{1,3}'' \\ \phi_{1,3}' \\ G_2[1] \\ G_2[-1]   \\ \mathrm{St}
\end{array} 
$$
with $\alpha,\beta \geq 2$ (see \cite{His89}). To this end, use (DL) with the following values of the Deligne--Lusztig characters, cut by the principal block $b$,
$$\begin{array}{l|l}
w & bR_w \\\hline
1 & 1+\phi_{1,3}'+\phi_{1,3}'' + \mathrm{St} \\
s,tst & 1-\phi_{1,3}'+\phi_{1,3}'' - \mathrm{St} \\
t,sts & 1+\phi_{1,3}'-\phi_{1,3}'' - \mathrm{St} \\
st,ts & 1 + G_2[-1] +\mathrm{St} \\
stst & 1 + G_2[1] +\mathrm{St} 
\end{array}$$
Here $s$ and $t$ denote the simple reflections in the Weyl group of type $G_2$ corresponding to the short and long simple root respectively.
\end{exercise}

\noindent (c) Let $G=\mathrm{SU}_5(q) = $``$\mathrm{SL}_5(-q)$''. Its order is given by 
\begin{tikzpicture}[remember picture,overlay]
\draw[->,>=latex]
  ([shift={(-20pt,-12pt)}]pic cs:defect) |- 
  ++(10pt,-10pt) 
  node[right,text width=3cm] 
    {\footnotesize largest defect
    };
\end{tikzpicture}$$ \begin{aligned}
|G| = &\, q^{10}((-q)^5-1)((-q)^4-1)((-q)^3-1)((-q^2)-1) \\
 = &\, q^{10}(q-1)^3 \underbrace{(q+1)^4}\tikzmark{defect}(q^2+1)(q^2-q+1)(q^4+1). \end{aligned}$$ 
\vskip1cm
We will work again in the case where $\ell \mid q+1$. In addition we will assume that $\ell > 5$ to ensure the existence of $\ell$-characters in regular position. As in the case of linear groups, the unipotent characters of $
\mathrm{SU}_n(q)$ are parametrized by partitions of $n$. Here, they are $1 = \rho_{(5)}, \rho_{(41)}, \rho_{(32)}, 
\rho_{(31^2)}, \rho_{(2^21)}, \rho_{(21^3)}, \rho_{(1^5)} = \mathrm{St}$ and are all contained in the principal $\ell$-block. By (Uni) the unipotent part of the decomposition matrix has the following shape
$$ D_{\text{unip}} = \left[ \begin{array}{ccccccc}
1 &\cdot &\cdot&\cdot&\cdot&\cdot&\cdot \\
* & 1 &\cdot &\cdot&\cdot&\cdot&\cdot \\
* & * & 1 &\cdot &\cdot&\cdot&\cdot \\
* &* &* &1 &\cdot &\cdot&\cdot \vphantom{1^2}\\
* &* &* &* &1 &\cdot &\cdot  \vphantom{2^2}\\
* &* &* &* &* &1 &\cdot  \vphantom{1^3}\\
* &* &* &* &* &* &1  \vphantom{1^5} \\ \end{array}\right]
\begin{array}{l}
5 \\ 41 \\ 32 \\ 31^2 \\ 2^21 \\ 21^3 \\ 1^5 \\ \end{array} $$
By (Hecke) the projective $kG$-module $R_T^G(k)$ decomposes as a direct sum of two PIMs. 
The corresponding decomposition of characters is $R_1 = (\rho_{(5)} + \rho_{(31^2)} + \rho_{(2^21)})+
(\rho_{(32)}+\rho_{(31^1)} + \rho_{(1^5)})$ which gives the first and third columns of the decomposition matrix. 

\smallskip

As in the previous examples, other columns can be determined by Harish-Chandra induction of projective modules of various Levi subgroups. The Levi subgroup  $L \subset \mathrm{GU}_3(q)\times \mathbb{F}_{q^2}^\times$ of type ${}^2A_2$ has two interesting PIMs $P'$ and $P''$ with respective characters
$$ \begin{aligned}
e([P']) = &\, \rho_{(21)} + 2 \rho_{(1^3)} + \text{non-unip.} \\
e([P'']) = &\, \rho_{(1^3)} + \text{non-unip.}
\end{aligned}$$ 
yielding by (HC) two projective $kG$-modules $R_L^G(P')$ and $R_L^G(P'')$ with characters
$$ \begin{aligned}
e(R_L^G([P'])) = &\, \rho_{(41)} + \rho_{(21^2)} + 2 (\rho_{(31^2)}+\rho_{(2^21)} + \rho_{(1^5)}) + \text{non-unip.} \\
e(R_L^G([P''])) = &\, \rho_{(31^2)}+\rho_{(2^21)} + \rho_{(1^5)} + \text{non-unip.}
\end{aligned}$$ 
Note that these projective modules might not be indecomposable.

\smallskip

If $\ell > 5$ there exists an $\ell$-character $\theta$ of $\bfT^{w_0 F}$ in general position and (Reg) applies. In other words, $\langle e([P]);R_{w_0}\rangle \geq 0$ for every projective $kG$-module $P$. With 
$$R_{w_0} = \rho_{(5)} +4\rho_{(41)} + 5\rho_{(32)}- 6\rho_{(31^2)} +5 \rho_{(2^21)} -4 \rho_{(21^2)} + \rho_{(1^5)}$$
we have $\langle e(R_L^G([P']));R_{w_0}\rangle = \langle e(R_L^G([P'']));R_{w_0}\rangle = 0$. Therefore 
$\langle e([P]);R_{w_0}\rangle = 0$ for every direct summand $P$ of $R_L^G(P')$ and $R_L^G(P'')$. From this we deduce that 
\begin{itemize}
\item[$\bullet$] $R_L^G(P')$ is indecomposable,
\item[$\bullet$] $R_L^G(P'') \simeq Q \oplus R_L^G(P')^{\oplus m}$ with $m =0,1,2$ and $Q$ indecomposable. But $m \neq 0$ is impossible (the PIM $R_L^G(P')$ cannot lie in the Harish-Chandra series of both $(L,P')$ and $(L,P'')$). 
\end{itemize}
This gives the second and fourth column of the decomposition matrix. As in the previous example, application of (Reg) to the fifth and sixth column of the decomposition matrix gives lower bounds for the decomposition numbers and $D_{\text{unip}}$ has the following shape
$$ D_{\text{unip}} = \left[ \begin{array}{ccccccc}
1 &\cdot &\cdot&\cdot&\cdot&\cdot&\cdot \\
\cdot & 1 &\cdot &\cdot&\cdot&\cdot&\cdot \\
\cdot & \cdot & 1 &\cdot &\cdot&\cdot&\cdot \\
1 &2 &1 &1 &\cdot &\cdot&\cdot \vphantom{1^2}\\
1 &2 &\cdot &1 &1 &\cdot &\cdot  \vphantom{2^2}\\
\cdot &1 &\cdot &\cdot &\alpha &1 &\cdot  \vphantom{1^3}\\
\cdot &2 &1 &1 &\beta &\gamma &1  \vphantom{1^5} \\ \end{array}\right] \quad \text{with} \ \gamma \geq 4 \text{ and } 5-4\alpha+\beta \geq 0.$$

The two missing columns correspond to projective covers of cuspidal simple $kG$-modules. They can be obtained using (DL) from the decomposition of each $P_w$ on the basis of PIMs. The minimal representatives of the $F$-conjugacy classes of $W \simeq \mathfrak{S}_5$ are ordered as follows under the Bruhat order: 
$1 \leq s_1,s_2 \leq s_1 s_2 ,s_2s_3s_3 \leq s_1 s_2 s_3 s_2 \leq w_0$. Let $Q_i$, $i=1,\ldots,7$, be the PIMs ordered as the columns of the decomposition matrix. Then
$$\begin{aligned}
e(P_{1}) = & \, e([Q_1]+[Q_3]) \ \text{ (already computed)} \\
e(P_{s_1}) =&\, \rho_{(5)} - \rho_{(32)} + \rho_{(2^21)} -\rho_{(1^5)} + \text{non-unip.} \\
= &\, e([Q_1]-[Q_3])  \\
e(P_{s_2}) =&\, \rho_{(5)} - \rho_{(41)} + \rho_{(32} - \rho_{(2^21)} -\rho_{(21^3)}- \rho_{(1^5)} + \text{non-unip.} \\
= &\, e([Q_1]-[Q_2]+[Q_3]).  \end{aligned}$$
Note that in the virtual module $P_{s_2}$, the module $Q_2$ appears with negative multiplicity. This is consistent with (DL) which ensures that since it does not occur in $P_1$, it must occur with a multiplicity whose sign is given by $(-1)^{\ell(s_2)}$.
$$\begin{aligned}
e(P_{s_1s_2}) = & \,  \rho_{(5)} - \rho_{(41)}  - \rho_{(2^21)} +\rho_{(21^3)}+ \rho_{(1^5)} + \text{non-unip.} \\
= &\, e\big([Q_1]-[Q_3] + [Q_5]+(2-\alpha)[Q_6] + (3-\beta-\gamma(2-\alpha))[Q_7]\big).  \\
\end{aligned}$$
Since $\ell(s_1s_2) = 2$, we must have $\alpha \leq 2$ and $3-\beta \geq \gamma(2-\alpha) \geq 4(2-\alpha)$. But recall that $5-4\alpha+\beta \geq 0$ which we can rewrite as $4(2-\alpha) \geq 3-\beta$. 
This forces $\beta = 4\alpha-5$, hence $\alpha =2$ (otherwise $\beta$ would be negative) and therefore $\beta = 3$. 
Consequently $ P_{s_1s_2} = [Q_1]-[Q_3] + [Q_5]$ and $Q_4$, $Q_5$, $Q_7$ have yet to occur. 
$$\begin{aligned}
e(P_{s_2s_3s_2}) = & \,  \rho_{(5)} +2\rho_{(41)} +\rho_{(32)} - \rho_{(2^21)} +2\rho_{(21^3)}- \rho_{(1^5)} + \text{non-unip.} \\
= &\, e([Q_1]+2[Q_2]+[Q_3] -\tikzmark{q4}6 [Q_4]) 
\begin{tikzpicture}[remember picture,overlay]
\draw[->,>=latex]
  ([shift={(-6pt,0pt)}]pic cs:q4) |- 
  ++(10pt,-13pt) 
  node[right,text width=3cm] 
    {\footnotesize$\ell(s_2s_3s_2) = 3$
    };
\end{tikzpicture} \\[15pt]
e(P_{s_1 s_2s_3s_2}) = & \,  \rho_{(5)} +\rho_{(41)} -\rho_{(32)} - \rho_{(2^21)} -\rho_{(21^3)}+ \rho_{(1^5)} + \text{non-unip.} \\
= &\, e([Q_1]+[Q_2]-[Q_3] -2 [Q_4]-2[Q_5]+2[Q_6]+(8-2\gamma)[Q_7]) 
\end{aligned}$$
which, since $\ell(s_1 s_2s_3s_2)$ is even forces $8-2\gamma \geq 0$ and therefore $\gamma =4$. We obtain finally
$$ D_{\text{unip}} = \left[ \begin{array}{ccccccc}
1 &\cdot &\cdot&\cdot&\cdot&\cdot&\cdot \\
\cdot & 1 &\cdot &\cdot&\cdot&\cdot&\cdot \\
\cdot & \cdot & 1 &\cdot &\cdot&\cdot&\cdot \\
1 &2 &1 &1 &\cdot &\cdot&\cdot \vphantom{1^2}\\
1 &2 &\cdot &1 &1 &\cdot &\cdot  \vphantom{2^2}\\
\cdot &1 &\cdot &\cdot &2 &1 &\cdot  \vphantom{1^3}\\
\cdot &2 &1 &1 &3 &4 &1  \vphantom{1^5} \\ \end{array}\right] .$$

\subsection{Example in $\mathrm{GU}_n(q)$}
The methods described in \S\ref{sec:smallrank} and used in the example of $\mathrm{SU}_5(q)$ in the previous section have shown powerful to determine completely the decomposition matrices for small-rank groups, up to the $42\times42$ decomposition matrix of $\mathrm{SU}_{10}(q)$ (see for example \cite{Dud13,DudMal15,DudMal16}). We give here a general example of a decomposition number that can be computed using Deligne--Lusztig characters. For the sake of simplicity we have chosen again the case of a group of  type ${}^2A$, although the proof can be  adapted to other classical groups.

\begin{theorem}[Dudas--Malle \cite{DudMal15}] Assume that $G=\mathrm{GU}_n(q)$, $\ell\mid q+1$ and $\ell > n$. Then $ d_{(1^n), (21^{n-2})} = n-1$.
\end{theorem}

\begin{proof}
Let $\theta$ be an irreducible $\ell$-character of $\bfT^{w_0 F}$ in general position (such a character exists since $\ell > n$). Then $(-1)^{\ell(w_0)} R_{w_0}(\theta)$ is an irreducible character, and hence
$$ (-1)^{\ell(w_0)} d(R_{w_0}(\theta)) = (-1)^{\ell(w_0)} d(R_{w_0}) = \sum_{S \in \mathrm{Irr}_k G} m_S[S]$$
with each $m_S \geq 0$. Among the simple $kG$-modules $S$ such that $m_S \neq 0$, choose $S_0$ with smallest possible $w \in W$ such that $\langle P_w; [S_0] \rangle \neq 0$. In other words, if $v < w$ then $m_S \langle P_v; [S] \rangle = 0$ for every simple $kG$-module $S$.

\smallskip

If $S$ is any simple $kG$-module such that $m_S \neq 0$, then by minimality of $w$ and (DL) we have
$(-1)^{\ell(w)} \langle P_w; [S] \rangle > 0$. In particular, $(-1)^{\ell(w)} m_S\langle P_w; [S] \rangle \geq 0$ for every simple $kG$-module $S$. Let us write
$$(-1)^{\ell(ww_0)} \langle e(P_w); R_{w_0}\rangle   = (-1)^{\ell(ww_0)} \langle P_w; d(R_{w_0}) \rangle 
 = \sum_{S \in \mathrm{Irr}_kG} (-1)^{\ell(w)}m _S\langle P_w; [S] \rangle.$$
By the orthogonality relation of Deligne--Lusztig characters, this sum is zero unless $w=w_0$. Since all the terms are non-negative, and one term is positive (for $S = S_0$) we deduce that $w=w_0$. In other words, if $[P_S]$ occurs in some $P_v$ for $v \neq w_0$ then $m_S = 0$. Now
$$\begin{aligned}
  \mathrm{rk}_\mathbb{Z} \langle P_v, v \neq w_0\rangle = &\, \#\{F\text{-conjugacy classes of $W$}\}-1 \\
  = &\, \#\{\text{unipotent characters of $G$}\}-1 \\
  = &\, \#\{\text{unipotent PIMs of $G$}\}-1  \quad \text{(by Theorem \ref{thm:basicset})}
\end{aligned}$$
Therefore the $\mathbb{Z}$-submodule of $K_0(kG\sfproj)$ spanned by the virtual projective modules $P_v$ for $v \neq w_0$ has codimension $1$. We deduce that there is at most one simple $kG$-module $S$ such that $m_{S} \neq 0$. It must be $S = S_0 = S_{(1^n)}$ since $e([P_{S_{(1^n)}}]) = \rho_{(1^n)}+\text{non-unipotent characters}$ and 
$$ \langle [P_{S_{(1^n)}}];d(R_{w_0})\rangle = \langle e([P_{S_{(1^n)}}]);R_{w_0}\rangle
 = \langle \mathrm{St};R_{w_0}\rangle = (-1)^{\ell(w_0)} \neq 0.$$
This proves that $d(R_{w_0}) = (-1)^{\ell(w_0)} [S_{(1^n)}]$. 

\smallskip

Now (Uni) shows that the character of the projective cover of $S_{(21^{n-2})}$ is of the form 
$ e([P_{S_{(21^{n-2})}}]) =  \rho_{(21^{n-2})} + \alpha \rho_{(1^n)}+\text{non-unipotent characters}$. 
Using the fact that it is orthogonal to $R_{w_0}$ we get
$$ 0 = \langle [P_{S_{(21^{n-2})}}] ; d(R_{w_0}) \rangle = \langle e([P_{S_{(21^{n-2})}}]);R_{w_0} \rangle
 =  \langle \rho_{(21^{n-2})}  ; R_{w_0}  \rangle + (-1)^{\ell(w_0)} \alpha$$
 which shows that 
$$d_{(1^n),(21^{n-2})} := \alpha = (-1)^{\ell(w_0)+1} \langle \rho_{(21^{n-2})}  ; R_{w_0}  \rangle = n-1.$$
Indeed, the scalar product $\langle\rho_\lambda  ; R_{w}\rangle$ is, up to a sign, equal to the value on $ww_0$ of the irreducible character of $\mathfrak{S}_n$ corresponding to the partition $\lambda$. So here, it equals the dimension of the representation of $\mathfrak{S}_n$ corresponding to $(21^{n-2})$ which is $n-1$ by the hook length formula.
\end{proof}

\subsection{Observations and conjectures}
The following conjectures were made by Geck in \cite{GecThesis} and Geck--Hiss in \cite{GecHis96}. 

\begin{conj}[Geck--Hiss] Assume that $\ell \neq p$ and that $\ell$ is large with respect to $|W|$.
\begin{itemize}
  \item[$\mathrm{(i)}$] The decomposition matrix has a unitriangular shape.
  \item[$\mathrm{(ii)}$] If $\rho$ is unipotent and cuspidal then $d(\rho)$ is irreducible \emph{i.e.} $d(\rho) = [S]$ for some simple $kG$-module $S$.
  \item[$\mathrm{(iii)}$] The unipotent part of the decomposition matrix is independent of $q$ (it depends only on the order of $q$ in $\mathbb{F}_\ell^\times$).
\end{itemize}
\end{conj}

The conjecture was first proven for groups of type $A$ and ${}^2 A$ by Geck \cite{Gec91} and for classical groups when $\ell$ is \emph{linear} by Gruber--Hiss \cite{GruHis97}. In addition, all the decomposition matrices computed so far for small rank group satisfy the conjecture. In a recent preprint \cite{DudMal17}, Malle and the author proved the part (ii) of the conjecture under the extra assumption that $p$ is good. 

\smallskip

The difficulty in proving this conjecture lies in producing nice projective modules. In the case where the conjecture is known to hold, generalized Gelfand--Graev modules were used. Malle and the author proposed a different strategy in \cite{DudMal14}, using again the cohomology of Deligne--Lusztig varieties. 

\begin{conj}[Dudas--Malle]\label{conj:qwproj}
For all $w \in W$ there is a sign $\varepsilon_w  = \pm1$ such that
$$ Q_w = \varepsilon_w D_G\tikzmark{alvis}([IH^\bullet\tikzmark{ih}(\overline{\bfX(w)},k(\bfT^{wF})_\ell)])$$
\begin{tikzpicture}[remember picture,overlay]
\draw[->,>=latex]
  ([shift={(-5pt,-3pt)}]pic cs:alvis) |- 
  ++(-10pt,-13pt) 
  node[left,text width=2.7cm] 
    {\footnotesize Alvis--Curtis duality
    };
\draw[->,>=latex]
  ([shift={(-10pt,-3pt)}]pic cs:ih) |- 
  ++(10pt,-13pt) 
  node[right,text width=4cm] 
    {\footnotesize intersection cohomology
    };
\end{tikzpicture}
\vskip 2pt
\noindent is the character of a (non-virtual) projective module.
\end{conj}

The $Q_w$'s are actually virtual projective modules whose character can be explicitly computed using Kazhdan--Lusztig polynomials. They satisfy a unitriangularity property (as we expect for the PIMs) and the multiplicities on unipotent characters do not depend on $q$. 

\begin{example}
For $w = w_0$ the variety $\bfX(w_0)$ is dense in $\bfG/\bfB$ and hence $\overline{\bfX(w_0)} = \bfG/\bfB$. It is a smooth projective variety paved by affine spaces $\bfB w \bfB/\bfB$, therefore
$$[IH^\bullet(\overline{\bfX(w)},K)] = [H^\bullet(\bfG/\bfB,K)] = |W| \cdot 1_G.$$
Consequently, $e(Q_{w_0}) = \pm |W|\cdot \mathrm{St}_G \pm\text{non-unipotent characters}$. There is indeed a PIM whose character is $\mathrm{St}_G+ \text{non-unip.}$, hence $Q_{w_0}$ is the character of a projective module as claimed in Conjecture \ref{conj:qwproj}. 
\end{example}

\section{Brauer trees of unipotent blocks}\label{chap:brauer}
 
This chapter is devoted to the study of unipotent blocks of finite reductive groups with cyclic defect groups. In that case the structure of the block is encoded in a planar embedded tree, the \emph{Brauer tree}. We explain how to use the cohomology complexes of Deligne--Lusztig varieties to get information on the characters of PIMs (which gives the tree as a graph) and extensions between simple modules (which gives the planar embedding). This is based on a recent work of Craven, Rouquier and the author \cite{CraDudRou17}. 
  
\subsection{Brauer trees}
Throughout this chapter, $G$ is any finite group and $b$ is an $\ell$-block of $\calO G$ with cyclic (and non-trivial) defect groups. We denote by $D$ a defect of $b$. The results on the structure of $b$ originate in a work of Brauer \cite{Bra41} subsequently completed by Dade \cite{Dad66} and Green \cite{Gre77}. For a self-contained treatment of blocks with cyclic defect groups we recommend \cite{Rou01} and \cite[\S V]{Alp}.
\smallskip

We denote by $\mathrm{Irr}_K b$ the set of irreducible ordinary characters of $G$ lying in $b$. There is a set $\mathrm{Exc}_K b \subset \mathrm{Irr}_K b$ called the set of \emph{exceptional characters} of $b$ such that if we define $\chi_\text{exc} := \sum_{\chi \in \mathrm{Exc}_K b} \chi$ then the character of any projective indecomposable $kG$-module $P$ in $kb$ is given by
$$ e([P]) = \chi + \chi'$$
with $\chi \neq \chi'$ and $\chi,\chi' \in \{\chi_\text{exc}\} \sqcup (\mathrm{Irr}_K b \smallsetminus \mathrm{Exc}_K b)$. In other words, a simple $kb$-module occurs in the $\ell$-reduction of either two distinct non-exceptional characters, or in one non-exceptional character and in every exceptional character. 

\smallskip

We define the \emph{Brauer graph} $\Gamma_b$ of $b$ as the graph with vertices labeled by 
$\{\chi_\text{exc}\} \sqcup (\mathrm{Irr}_K b \smallsetminus \mathrm{Exc}_K b)$ and edges $\chi \trait  \chi'$ for every PIM $P$ such that $e([P]) = \chi + \chi'$. The edges of the Brauer graph are therefore labeled by PIMs or equivalently by simple $kb$-modules (via their projective cover). 
The knowledge of the Brauer graph, together with the \emph{multiplicity} $m = \langle \chi_{\text{exc}};\chi_\text{exc} \rangle$ of the exceptional vertex, is equivalent to the knowledge of the decomposition matrix.

\begin{example}\label{ex:trees}
(a) Let $G = \mathbb{Z}/\ell \mathbb{Z}$ be the cyclic group of order $\ell$. There is only one PIM $P = kG$ which is the projective cover of the trivial module. One can write its character as
$$ e([P]) = 1_G + \sum \text{non-trivial characters} = 1_G + \chi_\text{exc}.$$ 
The convention here is to define the exceptional characters as the non-trivial characters of $G$. This gives the Brauer graph given in Figure \ref{fig:brauercyclic}. 
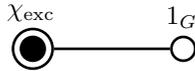
\begin{figure}[h]
\begin{center}
\begin{pspicture}(2,0.5)

  \psset{linewidth=1pt}

   \cnode[fillstyle=solid,fillcolor=black](0,0){5pt}{A2}
       \cnode(0,0){8pt}{A}
  \cnode(2,0){5pt}{B}
  \ncline[nodesep=0pt]{A}{B}\naput[npos=1.1]{$\vphantom{\big)}1_G$}\naput[npos=-0.2]{$\vphantom{\Big)}\chi_\text{exc}$}
\end{pspicture}
\end{center}
\caption{Brauer graph of $ \mathbb{Z}/\ell \mathbb{Z}$}
\label{fig:brauercyclic}
\end{figure}

\noindent
(b) Let us now consider the case of $G = \mathbb{Z}/\ell^r \mathbb{Z} \rtimes E$ where $E$ is an $\ell'$-subgroup of $\mathrm{Aut}_{\ell'}(\mathbb{Z}/\ell^r \mathbb{Z}) \simeq \mathbb{Z}/(\ell-1) \mathbb{Z}$. In particular the order $e$ of $|E|$ divides $\ell-1$. As in the previous example $\calO G$ is indecomposable and hence it forms a single block. 

\smallskip

Every simple $kG$-module $S$ has a trivial action of the $\ell$-group $\mathbb{Z}/\ell^r \mathbb{Z}$, and therefore it must be of the form $S = \mathrm{Inf}_E^G(\mathrm{Res}_E^G S)$. In addition, since $E$ is an $\ell'$-group, the restriction $\mathrm{Res}_E^G S$ is semisimple hence simple. We deduce that the simple $kG$-modules are in bijection with the irreducible representations of $E$ (over $k$ or $K$). More precisely, if we fix a generator $x$ of $E$ and $\zeta$ a primitive $e$-th root of unity in $\calO^\times$, one can consider the simple $kE$-modules $k_{\overline{\zeta}^i}$ to be the one-dimensional representation of $E$ on which $x$ acts by $\overline{\zeta}^i$. They lift to characteristic zero to one-dimensional $\calO G$-modules $\calO_{\zeta^i}$ and $KG$-modules $K_{\zeta^i}$ on which $x$ acts by $\zeta^i$. Let $T_i := \mathrm{Inf}_E^G \, k_{\zeta^i}$. Then:
\begin{itemize}
 \item[$\bullet$] The $kG$-modules $T_i = \mathrm{Inf}_E^G \, k_{\zeta^i}$ for $i =0,1,\ldots,e-1$ form a set of representatives of $\mathrm{Irr}_kG$.
 \item[$\bullet$] $T_i$ lifts to characteristic zero as $\widetilde{T}_i = \mathrm{Inf}_E^G \, \calO_{\zeta^i}$. We denote by $\theta_i$ the character of $K \widetilde{T}_i \simeq \mathrm{Inf}_E^G \, K_{\zeta^i}$.
 \item[$\bullet$] The projective cover of $T_i$ is $P_i :=  \mathrm{Ind}_E^G \, k_{\zeta^i}$, which lifts to characteristic zero as $\widetilde{P}_i = \mathrm{Ind}_E^G \, \calO_{\zeta^i}$. It has character
$$ \begin{aligned}
e([P_i]) = [\mathrm{Ind}_E^G \, K_{\zeta^i}] = [\mathrm{Inf}_E^G \, K_{\zeta^i}]  + \theta_{\text{exc}}
= \theta_i + \theta_{\text{exc}}
\end{aligned}$$
where $\theta_{\text{exc}}$ denotes the sum of the irreducible characters of $G$ which are non-trivial on $\mathbb{Z}/\ell^r\mathbb{Z}$. These are the exceptional characters of $G$.
\end{itemize}
We deduce that the Brauer graph is a star-shaped tree as shown in Figure \ref{fig:brauernormaldefect}.
\begin{figure}[h]
\begin{center}
\begin{pspicture}(4,4.5)

  \psset{linewidth=1pt}

   \cnode[fillstyle=solid,fillcolor=black](2,2){5pt}{A2}
       \cnode(2,2){8pt}{A}
  \cnode(4,2){5pt}{B}
  \cnode(3.73,3){5pt}{C}
  \cnode(3.73,1){5pt}{D}
    \cnode(3,3.73){5pt}{H}
    \cnode(3,0.27){5pt}{I}
    \cnode(2,4){5pt}{J}

  \cnode(0,2){5pt}{E}
  \cnode(0.27,3){5pt}{F}
  \cnode(0.27,1){5pt}{G}

  \ncline[nodesep=0pt]{A}{B}\ncput[npos=1.4]{$\theta_0$}
  \ncline[nodesep=0pt]{A}{C}\ncput[npos=1.45]{$\theta_1$}
  \ncline[nodesep=0pt]{A}{D}\ncput[npos=1.38]{$\hphantom{aaa}\theta_{e-1}$}
  \ncline[nodesep=0pt]{A}{E}
  \ncline[nodesep=0pt]{A}{F}
  \ncline[nodesep=0pt]{A}{G}
  \ncline[nodesep=0pt]{A}{H}\ncput[npos=1.42]{$\theta_2$}
  \ncline[nodesep=0pt]{A}{I}\ncput[npos=1.4]{$\hphantom{aaa}\theta_{e-2}$}
  \ncline[nodesep=0pt]{A}{J}\ncput[npos=1.4]{$\theta_3$}

\psellipticarc[linestyle=dotted,linewidth=1.5pt](2,2)(2,2){103}{135}
\psellipticarc[linestyle=dotted,linewidth=1.5pt](2,2)(2,2){225}{287}
%\psellipticarc[linewidth=1pt]{->}(2,2)(1,1){103}{135}
%\psellipticarc[linewidth=1pt]{->}(2,2)(1,1){225}{287}

\end{pspicture}
\end{center}
\caption{Brauer graph of $ \mathbb{Z}/\ell^r \mathbb{Z} \rtimes \mathbb{Z}/e\mathbb{Z}$ with $e\mid \ell-1$}
\label{fig:brauernormaldefect}
\end{figure}
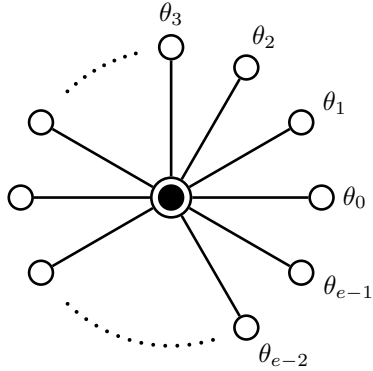

\smallskip

\noindent
(c) Let $G = \mathfrak{S}_\ell$ and $b$ be the principal $\ell$-block (with defect a Sylow subgroup $\mathbb{Z}/\ell\mathbb{Z}$ of $\mathfrak{S}_\ell$). The irreducible ordinary characters of $G$ are labeled by partitions of $\ell$. The characters in $b$ correspond to partitions which are $\ell$-hooks 
$$\mathrm{Irr}_K b = \{1_G = \chi_{(\ell)}, \chi_{(\ell-1,1)}, \chi_{(\ell-2,1^2)}, \ldots, \chi_{(1^\ell)} = \varepsilon\}.$$ 
Here $\varepsilon$ denotes the sign characters. The subgroup $\mathfrak{S}_{\ell-1}$ of $G$ is an $\ell'$-group, therefore every irreducible character is the character of a projective module. Consequently, the same holds for the induction of any representation from $\mathfrak{S}_{\ell-1}$ to $\mathfrak{S}_{\ell}$. Using the branching rules for induction we have
$$ \mathrm{Ind}_{\mathfrak{S}_{\ell-1}}^{\mathfrak{S}_\ell} \Yvcentermath1 \Yboxdim{7pt} \gyoung(;_2\hdts;,|2\vdts,;) \ = \ \tikzmark{projchar}\underbrace{\gyoung(;_2\hdts;,|2\vdts,;,;) \ + \ \gyoung(;_2\hdts;;,|2\vdts,;)} \ +\ \tikzmark{notinb}\left(\gyoung(;_2\hdts;,|2\vdts;/2,;)\right)$$
\begin{tikzpicture}[remember picture,overlay]
\draw[->,>=latex]
  ([shift={(43pt,-27pt)}]pic cs:projchar) |- 
  ++(-10pt,-10pt) 
  node[left,text width=2.8cm] 
    {\footnotesize projective character};
\draw[->,>=latex]
  ([shift={(23pt,-22pt)}]pic cs:notinb) |- 
  ++(10pt,-15pt) 
  node[right,text width=4cm] 
    {\footnotesize not in $b$};
\end{tikzpicture}
\vskip 5mm
\noindent This shows that the Brauer graph $\Gamma_b$ is a line as shown in Figure \ref{fig:brauersn}. Note that here any vertex can be chosen to be the exceptional vertex.
\begin{figure}[h]
\begin{center}
\begin{pspicture}(10,0.5)
  \psset{linewidth=1pt}
  \cnode(0,0){5pt}{A}
  \cnode(2,0){5pt}{B}
  \cnode(4,0){5pt}{C}
  \cnode(8,0){5pt}{D}
  \cnode(10,0){5pt}{E}
  \ncline[nodesep=0pt]{A}{B}\naput[npos=-0.1]{$\vphantom{\big)}(\ell)$}
  \ncline[nodesep=0pt]{B}{C}\naput[npos=-0.1]{$\vphantom{\big)}(\ell-1,1)$}
  \ncline[nodesep=0pt,linestyle=dashed]{C}{D}\naput[npos=0]{$\vphantom{\big)}(\ell-2,1^2)$}
  \ncline[nodesep=0pt]{D}{E}\naput[npos=-0.1]{$\vphantom{\big)}(2,1^{\ell-2})$}\naput[npos=1.1]{$\vphantom{\big)}(1^\ell)$}
\end{pspicture}
\end{center}
\caption{Brauer graph of the principal $\ell$-block of $\mathfrak{S}_\ell$}
\label{fig:brauersn}
\end{figure}
\end{example}

\begin{theorem}[{\cite[\S 23]{Alp}}] The Brauer graph is a tree, called the \emph{Brauer tree} of $b$.
\end{theorem}

A planar embedding of $\Gamma_b$ is defined by an ordering of the set of edges incident to any given vertex. Planar embedded trees will be drawn according to the anti-clockwise order around a vertex. 

\begin{theorem}[Structure of PIMs {\cite[\S 22]{Alp}}]\label{thm:pimsbrauer}
Let $\Gamma_b$ be the Brauer tree of $b$.
\begin{itemize}
 \item[$\mathrm{(i)}$] There exists a unique ordering around each vertex of $\Gamma_b$ such that if $S$ and $T$ are two simple $kb$-modules labeling edges incident to a given vertex then $T$ follows immediately $S$ if and only if $\mathrm{Ext}_{kG}^1(S,T) \neq 0$.
 \item[$\mathrm{(ii)}$] Given a simple $kb$-module $S$ labeling an edge between non-exceptional vertices as follows
\begin{center}
\begin{pspicture}(8,4.3)

  \psset{linewidth=1pt}
  \psset{unit = 0.8cm}
  \cnode(2,2.5){5pt}{A}
  \cnode(6,2.5){5pt}{B}
  \cnode(3,4.23){5pt}{H}\cnode[linestyle=none](3.5,5.1){0pt}{Hp}
  \cnode(3,0.77){5pt}{I}\cnode[linestyle=none](3.5,-0.1){0pt}{Ip}
  \cnode(1,4.23){5pt}{J}\cnode[linestyle=none](0.5,5.1){0pt}{Jp}
  \cnode(0,2.5){5pt}{K}\cnode[linestyle=none](-1,2.5){0pt}{Kp}
  \cnode(5,0.77){5pt}{H2}\cnode[linestyle=none](4.5,-0.1){0pt}{H2p}
  \cnode(5,4.23){5pt}{I2}\cnode[linestyle=none](4.5,5.1){0pt}{I2p}
  \cnode(7,0.77){5pt}{J2}\cnode[linestyle=none](7.5,-0.1){0pt}{J2p}
  \cnode(8,2.5){5pt}{K2}\cnode[linestyle=none](9,2.5){0pt}{K2p}

  \ncline[nodesep=0pt]{A}{B}\ncput*[npos=0.5]{$S$}
  \ncline[nodesep=0pt]{A}{H}\ncput*[npos=0.5]{$S_1$}
  \ncline[nodesep=0pt]{A}{I}\ncput*[npos=0.5]{$S_a$}
  \ncline[nodesep=0pt]{A}{J}\ncput*[npos=0.5]{$S_2$}
  \ncline[nodesep=0pt]{A}{K}\ncput*[npos=0.5]{$S_3$}
  \ncline[nodesep=0pt]{B}{H2}\ncput*[npos=0.5]{$S_1'$}
  \ncline[nodesep=0pt]{B}{I2}\ncput*[npos=0.5]{$S_b'$}
  \ncline[nodesep=0pt]{B}{J2}\ncput*[npos=0.5]{$S_2'$}
  \ncline[nodesep=0pt]{B}{K2}\ncput*[npos=0.5]{$S_3'$}
  \ncline[nodesep=3pt,linestyle=dashed]{H}{Hp}
  \ncline[nodesep=3pt,linestyle=dashed]{I}{Ip}
  \ncline[nodesep=3pt,linestyle=dashed]{J}{Jp}
  \ncline[nodesep=3pt,linestyle=dashed]{K}{Kp}
  \ncline[nodesep=3pt,linestyle=dashed]{H2}{H2p}
  \ncline[nodesep=3pt,linestyle=dashed]{I2}{I2p}
  \ncline[nodesep=3pt,linestyle=dashed]{J2}{J2p}
  \ncline[nodesep=3pt,linestyle=dashed]{K2}{K2p}   
   
  \psellipticarc[linestyle=dotted,linewidth=1.5pt](2,2.5)(2,2){193}{287}
  \psellipticarc[linewidth=1pt]{->}(2,2.5)(1,1){205}{275}
  \psellipticarc[linestyle=dotted,linewidth=1.5pt](6,2.5)(2,2){13}{107}
  \psellipticarc[linewidth=1pt]{->}(6,2.5)(1,1){25}{95}
\end{pspicture}
\end{center}
the Loewy structure of $P_S$ is given by
$$ P_S = \begin{array}{c} S \\ S_1 \ \ S_1' \\ S_2 \ \ S_2' \\ \vdots \ \ \ \ \vdots \\[3pt] S_a \ \ S_b' \\ S\end{array}$$
\end{itemize}
\end{theorem}

Together with this planar embedding, the Brauer tree $\Gamma_b$ will be referred to as the \emph{planar embedded Brauer tree} of $b$. This tree encodes the structure of the block $b$. Indeed, if $b'$ is a block of $G'$ with cyclic defect groups, then $kb$ and $kb'$ are Morita equivalent if and only if the planar embedded Brauer trees of $b$ and $b'$ coincide, and the multiplicity of the exceptional vertices are equal.

\begin{rmk}
Note that the structure of $P_S$ can also be described in the case where the edge of $S$ is connected to the exceptional node. In that case one needs to repeat the composition series $S_1/\cdots/S_a/S$ a number of times equal to the multiplicity $m = \langle \chi_\text{exc};\chi_\text{exc}\rangle$ of the exceptional vertex. In other words, one needs to turn around the exceptional vertex $m$ times and consider that all the other vertices have multiplicity one. 
\end{rmk}

The proof of the structure theorem consists in constructing a stable equivalence between the block $b$ of $\calO G$ and its Brauer correspondent $c$ in $N_G(D)$. This equivalence is built from the Green correspondence between $b$ and its Brauer correspondent in $N_G(\mathbb{Z}/\ell \mathbb{Z})$ and a Morita equivalence between the latter block and its Brauer correspondent in $N_G(D)$ (which is $c$). For a block with normal defect group (as in the case of $c$), the Brauer tree is a star, as shown in Example \ref{ex:trees}.b, and the structure theorem is easily proved. For more details see for example \cite{Rou01} or \cite[\S V]{Alp}.

\smallskip

The structure theorem for the PIMs has the following consequence, which we will use in the following section. Let $\chi$ be an irreducible ordinary character of $Kb$ which labels a leaf of $\Gamma_b$ (a vertex with only one incident edge). The $\ell$-reduction of $\chi$ is the simple $kG$-module $S$ which labels the unique edge incident to that leaf. This is the particular case of Theorem \ref{thm:pimsbrauer} where $a = 0$. Then $\Omega S$ is a uniserial module with composition factors $S_1'/\cdots/S_b'/S$. Moreover, it lifts to an ordinary character labeling the unique vertex adjacent to the leaf, say $\chi'$, so that $e([P_S]) = \chi+\chi'$. Similarly, $\Omega^2 S$ is a uniserial module which lifts to an ordinary character $\chi''$ such that $e([P_{S_1'}]) = \chi'+\chi''$. The edge labeled by $S_1'$ is the edge which comes directly after $S$ in the cyclic ordering around $\chi'$. If we iterate this process, we see that $\Omega^i S$ can be obtained by following the edges in a walk around the tree. More precisely, we obtain a sequence of ordinary characters $\chi_i$ and simple $kb$-modules $T_i$ such that:
\begin{itemize}
 \item[$\bullet$] $\Omega^iS$ is a uniserial module with head $T_i$ and socle $T_{i-1}$.
 \item[$\bullet$] $\Omega^iS$ lifts to an $\calO G$-lattice with character $\chi_i$ (which is either $\chi_\text{exc}$ or irreducible, and hence labels a vertex).
 \item[$\bullet$] $e([P_{T_i}]) = \chi_i + \chi_{i+1}$.
 \item[$\bullet$] $T_{i+1}$ labels the edge coming directly after $T_i$ in the cyclic ordering around the vertex labeled by $\chi_{i+1}$.  
\end{itemize}
If $e$ denotes the number of isomorphism classes of simple $kb$-modules, then $\Omega^{2e}S \simeq S$. Each simple $kG$-module appears exactly twice in the sequence $T_0,T_1,\ldots,T_{2e-1}$, called the \emph{Green walk around $\Gamma_b$} \cite{Gre77}. 

\begin{example} We consider the planar embedded tree as shown in Figure \ref{fig:walk}. Then for example $\Omega^3 S \simeq S_2$  lifts to a lattice with character $\rho_3$ whereas $\Omega^8 S$ is uniserial with composition series $S_4/S_1/S_2/S_3$, and it lifts to a lattice with character $\rho_2$. The sequence $T_0,T_1,\ldots,T_{9}$ is $S,S_1,S_2,S_2,S_3,S_3,S_4,S_4,S_1,S$, whereas the sequence $\chi_0,\ldots,\chi_9$ is $\chi,\rho_1,\rho_2,\rho_3,\rho_2,\chi_\text{exc},\rho_2,\rho_4,\rho_2,\rho_1$.
\begin{figure}[h]
\begin{center}
\begin{pspicture}(6,4.6)
  \psset{linewidth=1pt}
  \cnode(0,2.3){8pt}{A}\cnode[fillstyle=solid,fillcolor=black](0,2.3){5pt}{A}
  \cnode(2,2.3){5pt}{B}
  \cnode(4,2.3){5pt}{C}
  \cnode(6,2.3){5pt}{D}
  \cnode(2,4.3){5pt}{B1}
  \cnode(2,0.3){5pt}{B2}

  \ncline[nodesep=0pt]{A}{B}\ncput*[npos=0.5]{$S_3$}
  \ncline[nodesep=0pt]{B}{B1}\ncput*[npos=0.5]{$S_2$}\ncput[npos=1.33]{\vphantom{\big)}$\rho_3$}
  \ncline[nodesep=0pt]{B}{B2}\ncput*[npos=0.5]{$S_4$}\ncput[npos=1.3]{\vphantom{\big)}$\rho_4$}
  \ncline[nodesep=0pt]{B}{C}\ncput*[npos=0.5]{$S_1$}\naput[npos=0.1]{\vphantom{\big)}$\rho_2$}
  \ncline[nodesep=0pt]{C}{D}\ncput*[npos=0.5]{$S$}\naput[npos=-0.1]{\vphantom{\big)}$\rho_1$}\naput[npos=1.1]{\vphantom{\big)}$\chi$}
  
  \pscurve[linestyle=dotted](6.5,2.3)(6,2.9)(5,2.6)(4,2.9)(3,2.6)(2.3,3.3)(2.6,4.3)(2,4.8)(1.4,4.3)(1.7,3.3)(1,2.6)(0,2.9)(-0.6,2.3)(0,1.7)(1,2)(1.7,1.3)(1.4,0.3)(2,-0.3)(2.6,0.3)(2.3,1.3)(3,2)(4,1.7)(5,2)(6,1.7)(6.5,2.3)
  
  \rput(5,2.8){\footnotesize $0$}
  \rput(3,2.8){\footnotesize $1$}
  \rput(2.5,3.3){\footnotesize $2$}
  \rput(1.5,3.3){\footnotesize $3$}
  \rput(1,2.8){\footnotesize $4$}
  \rput(1,1.8){\footnotesize $5$}
  \rput(1.5,1.3){\footnotesize $6$}
  \rput(2.5,1.3){\footnotesize $7$}
  \rput(3,1.8){\footnotesize $8$}
  \rput(5,1.8){\footnotesize $9$}
  
\end{pspicture}
\end{center}
\caption{Walking around the Brauer tree}
\label{fig:walk}
\end{figure}
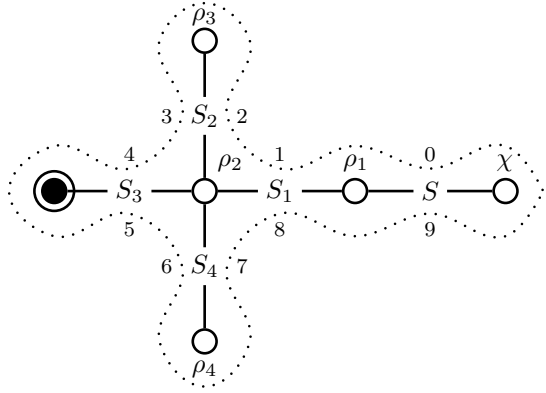
\end{example}

\subsection{The case of unipotent blocks}\label{sec:caseofunipotentblocks}
We now focus on the case of finite reductive groups. As before, $\bfG$ is a connected reductive group over $\overline{\mathbb{F}}_p$ together with a Frobenius endomorphism $F : \bfG \longrightarrow \bfG$ defining an $\mathbb{F}_q$-structure. Throughout this section and until the end of this chapter we will assume for simplicity that $(\bfG,F)$ is \emph{split} (\emph{i.e.} $F$ acts trivially on the Weyl group). 

\smallskip

We denote by $d$ the order of $q$ modulo $\ell$, or equivalently the order of the class of $q$ in $k$. The integer $d$ is minimal for the property that $\ell \mid \Phi_d(q)$, where $\Phi_d$ is the $d$-th cyclotomic polynomial. Recall that when $\ell$ is very good, the unipotent $\ell$-blocks are ``generic'' and parametrized by $d$-cuspidal pairs.
\begin{equation}\label{eq:paramblocks}
 \begin{array}{rcl} 
\left\{\hskip-1.3mm \begin{array}{c} \text{Unipotent $\ell$-blocks} \\ \text{with defect $D$} \end{array} \hskip-1.3mm\right\}
& \longleftrightarrow & 
\left\{\hskip-1.3mm\begin{array}{c} \text{$d$-cuspidal pairs $(\bfL,\rho)$} \\ \text{with $D \simeq (Z(\bfL)^\circ)_\ell^F$} \end{array} \hskip-1.3mm\right\} \hskip-1mm\Big/ G \\[10pt]
b(\bfL,\rho) & \testleftlong & (\bfL,\rho) \end{array}\end{equation}
When $D$ is cyclic, the non-exceptional characters in $b(\bfL,\rho)$ are the unipotent characters in $b(\bfL,\rho)$, which are the irreducible constituents of the virtual character
$$R_\bfL^\bfG(\rho) = \sum_{i \in \mathbb{Z}} (-1)^i [H_c^i(\bfY_\bfV,K)_\rho]$$
where $\bfV = R_u(\bfP)$ for some parabolic subgroup $\bfP$ with Levi complement $\bfL$. The variety $\bfY_\bfV$ is the parabolic Deligne--Lusztig variety attached to $\bfV$ (see \S\ref{sec:hcdl}). Recall that when $\bfL$ is a maximal torus of type $w$ and $\bfP$ is conjugate to the Borel subgroup $\bfB$ by $\dot w$ then $\bfY_\bfV \simeq \bfY(\dot w)$. In that case the non-exceptional characters in the block are the constituents of the Deligne--Lusztig character $R_w$. 

\smallskip

We first list the algebraic methods which can be used to determine the Brauer trees of  unipotent blocks of $G$ with cyclic defect groups. The first three arguments are not specific to finite reductive groups, whereas the last one relies on results by Geck \cite{Gec92}. 

\smallskip

\noindent {\bf(Parity)} If $\chi \trait  \chi'$ then $\chi(1) \equiv - \chi'(1)$ modulo $\ell$ (the dimension of a projective module is divisible by $\ell$). 

\smallskip

\noindent {\bf(Real stem)} If $b$ is real, the real characters form a single connected line containing the exceptional node called the \emph{real stem} of $\Gamma_b$.  The complex conjugation induces a symmetry of $\Gamma_b$ with respect to that line.

\smallskip

\noindent {\bf(Degree)} The dimension of a non-exceptional character $\chi$ equals the sum of the dimensions of the simple modules labeling the edges incident to $\chi$ in $\Gamma_b$.

\smallskip

\noindent {\bf(Hecke)} The Brauer trees of the blocks of the Hecke algebra associated to a given Harish-Chandra series in $b$ (as defined by Geck in \cite{Gec92}) are subtrees of $\Gamma_b$. Each of these subtrees is a line, with dimension (or rather $a$-function) increasing towards the exceptional vertex.
 
\begin{example}\label{ex:g2tree}
Let $G$ be a finite reductive group of type $G_2$. We denote by $s$ and $t$ the two simple reflections of its Weyl group $W$. The degrees of $W$ are $2$ and $6$, therefore the order of the finite group $G$ is
$$ \begin{aligned} 
|G_2(q)| = & \, q^6(q^6-1)(q^2-1) \\
= & \, q^6(q-1)^2(q+1)^2 (q^2+q+1)(q^2-q+1).
\end{aligned}$$
In this decomposition the exponents of the cyclotomic polynomials $\Phi_3(q) =q^2+q+1$ and $\Phi_6(q) =q^2-q+1$
are equal to $1$. Therefore when $\ell > 3$ (when $\ell$ is good) and $\ell$ divides one of these polynomials, the Sylow $\ell$-subgroups of $G$ are cyclic and the principal block has cyclic defect groups. 

\smallskip

Assume that $\ell > 3$ and $\ell \mid  \Phi_6(q) =q^2-q+1$, in which case $q$ has order $6$ modulo $\ell$ (the Coxeter number). A torus of type $w= st$ is a $\Phi_6$-Sylow subgroup since $|\bfT^{wF}| = q^2-q+1$. The trivial character of this torus is $6$-cuspidal, and the corresponding block via \eqref{eq:paramblocks} is the principal block, whose characters are
$$\begin{aligned}
\mathrm{Irr}_K b(\bfT_w,1) &\, = \left\{\hskip-1.3mm\begin{array}{c} \text{constituents of} \\ \text{$R_w = R_w(1)$} \end{array} \hskip-1.3mm\right\} \bigsqcup \left\{\hskip-1.3mm\begin{array}{c} \text{constituents of $R_w(\theta)$} \\ \text{for $\theta \in \mathrm{Irr}_\ell \bfT^{wF}$ and $\theta \neq 1$} \end{array} \hskip-1.3mm\right\}\\
& \, = \big\{\underbrace{1,\mathrm{St},\phi_{2,1}}_{\text{principal series}},\underbrace{G_2[-1],G_2[\theta],G_2[\theta^2]}_{\text{cuspidal characters}}\big\} \bigsqcup \mathrm{Exc}_K b
\end{aligned}$$
since $R_w = 1+\mathrm{St} -\phi_{2,1} + G_2[-1]+G_2[\theta]+G_2[\theta^2]$. As in \S\ref{sec:smallrank}, we use the notation of \cite[\S 13]{Car} for the unipotent characters. In particular $\theta$ is a primitive third root of $1$ in $\calO^\times$. The only non-real unipotent characters are $G_2[\theta]$ and $G_2[\theta^2]$, therefore the real stem is 
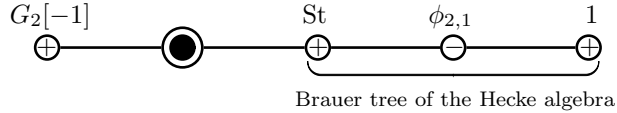
\begin{center}
\begin{pspicture}(8,2)
  \psset{linewidth=1pt}
  \psset{unit = 0.9cm}
 
  \cnode[fillstyle=solid,fillcolor=black](2,1.3){5pt}{A2}
  \cnode(2,1.3){8pt}{A}
  \cnode(0,1.3){5pt}{B}
  \cnode(4,1.3){5pt}{C}
  \cnode(6,1.3){5pt}{D}
  \cnode(8,1.3){5pt}{E}

  \ncline[nodesep=0pt]{A}{B}\nbput[npos=1.1]{$\vphantom{\big)} G_2[-1]$}\ncput[npos=1.135]{$+$}
  \ncline[nodesep=0pt]{A}{C}\naput[npos=1.1]{$\vphantom{\big)} \mathrm{St}$}\ncput[npos=1.135]{$+$}
  \ncline[nodesep=0pt]{C}{D}\naput[npos=1.1]{$\vphantom{\big)} \phi_{2,1}$}\ncput[npos=1.13]{$-$}
  \ncline[nodesep=0pt]{D}{E}\naput[npos=1.15]{$\vphantom{\big)} 1$}\ncput[npos=1.13]{$+$}
  \psbrace[nodesepA=-2.1cm,nodesepB=13pt,rot=90,braceWidthInner=5pt,braceWidthOuter=0.05pt,braceWidth=0.5pt](3.85,1.1)(8.15,1.1){\footnotesize Brauer tree of the Hecke algebra}
  
\end{pspicture}
\end{center}
By (Parity) the complex conjugate characters $G_2[\theta]$ and $G_2[\theta^2]$ must be connected to either the vertex labeled by $\phi_{2,1}$ or the exceptional vertex. But $\phi_{2,1}(1)-G_2[\theta](1)-G_2[\theta^2](1) < 0$ if $q>2$. Now $q=2$ would force $\ell = 3$, which is a prime number that we excluded. Therefore (Degree) forces  $\Gamma_b$ to be as in Figure~\ref{G2d6}.
\begin{figure}[h]
\begin{center}
\begin{pspicture}(8,4.5)

  \psset{linewidth=1pt}
  \psset{unit = 0.9cm}
 
  \cnode[fillstyle=solid,fillcolor=black](2,2.5){5pt}{A2}
  \cnode(2,2.5){8pt}{A}
  \cnode(0,2.5){5pt}{B}
  \cnode(4,2.5){5pt}{C}
  \cnode(6,2.5){5pt}{D}
  \cnode(8,2.5){5pt}{E}
  \cnode(2,0.5){5pt}{F}
  \cnode(2,4.5){5pt}{G}

  \ncline[nodesep=0pt]{A}{B}\nbput[npos=1.1]{$\vphantom{\big)} G_2[-1]$}
  \ncline[nodesep=0pt]{A}{C}\naput[npos=1.1]{$\vphantom{\big)} \mathrm{St}$}
  \ncline[nodesep=0pt]{C}{D}\naput[npos=1.1]{$\vphantom{\big)} \phi_{2,1}$}
  \ncline[nodesep=0pt]{D}{E}\naput[npos=1.15]{$\vphantom{\big)} 1$}
  \ncline[nodesep=0pt]{A}{F}\ncput[npos=1.5]{$\vphantom{\big)} G_2[\theta^2]$}
  \ncline[nodesep=0pt]{A}{G}\ncput[npos=1.5]{$\vphantom{\big)} G_2[\theta]$}
\end{pspicture}
\end{center}
\caption{Brauer tree of the principal $\Phi_{6}$-block of $G_2(q)$} 
\label{G2d6}
\end{figure}
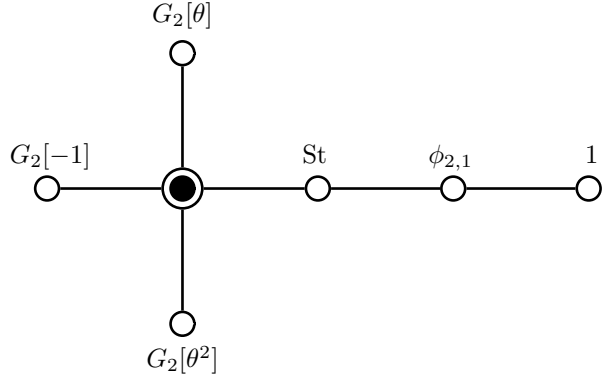
\end{example}

\begin{exercise}
If $\ell \mid q^2+q+1 = \Phi_3(q)$ then $|\bfT^{wF}| = q^2+q+1$, and $\bfT^{wF}$ is a $\Phi_3$-Sylow subgroup for $w = stst = (st)^2$. 
Then the non-exceptional characters of the principal block, with their parity, are 
$$ \mathrm{Irr}_K b \smallsetminus \mathrm{Exc}_K b = \begin{array}[t]{c@{\, }c@{\, }c@{\, }c@{\, }c@{\, }c} \{1,&\mathrm{St},&\phi_{2,2},&
G_2[1],&G_2[\theta],&G_2[\theta^2]\}. \\ + & + & - & + & - & - \end{array}$$
Here the only non-real characters are again the complex conjugate characters $G_2[\theta]$ and $G_2[\theta^2]$. Show that in that case the Brauer tree is given as in Figure \ref{G2d3}.
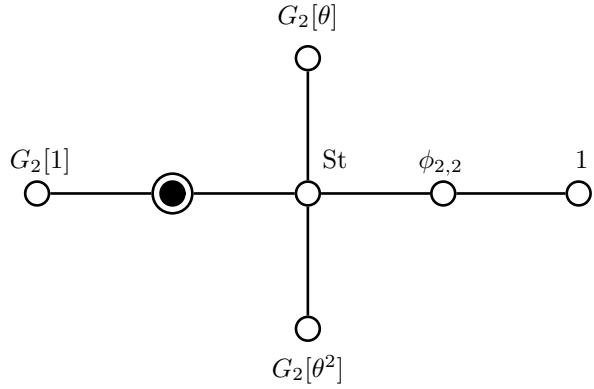
\begin{figure}[h]
\begin{center}
\begin{pspicture}(8,4.5)

  \psset{linewidth=1pt}
  \psset{unit = 0.9cm}
 
  \cnode[fillstyle=solid,fillcolor=black](2,2.5){5pt}{A2}
  \cnode(2,2.5){8pt}{A}
  \cnode(0,2.5){5pt}{B}
  \cnode(4,2.5){5pt}{C}
  \cnode(6,2.5){5pt}{D}
  \cnode(8,2.5){5pt}{E}
  \cnode(4,0.5){5pt}{F}
  \cnode(4,4.5){5pt}{G}

  \ncline[nodesep=0pt]{A}{B}\nbput[npos=1.1]{$\vphantom{\big)} G_2[1]$}
  \ncline[nodesep=0pt]{A}{C}\naput[npos=1.4]{$\vphantom{\big)} \mathrm{St}$}
  \ncline[nodesep=0pt]{C}{D}\naput[npos=1.1]{$\vphantom{\big)} \phi_{2,2}$}
  \ncline[nodesep=0pt]{D}{E}\naput[npos=1.15]{$\vphantom{\big)} 1$}
  \ncline[nodesep=0pt]{C}{F}\ncput[npos=1.5]{$\vphantom{\big)} G_2[\theta^2]$}
  \ncline[nodesep=0pt]{C}{G}\ncput[npos=1.5]{$\vphantom{\big)} G_2[\theta]$}
\end{pspicture}
\end{center}
\caption{Brauer tree of the principal $\Phi_{3}$-block of $G_2(q)$} 
\label{G2d3}
\end{figure}
\end{exercise}

\pagebreak
\begin{theorem} The Brauer trees of unipotent blocks are known for
\begin{itemize}
\item[$\mathrm{(i)}$] $\bfG$ of classical type $A$, $B$, $C$ and $D$ \emph{(Fong--Srinivasan \cite{FonSri84,FonSri90})}.
\item[$\mathrm{(ii)}$] $\bfG$ of exceptional type except $E_7$ and $E_8$.
\end{itemize}
\end{theorem}

Note that for groups of small rank, the determination of the trees follows from the determination of all the $\ell$-decomposition matrices for various $\ell$, which were more specifically solved by Burkhart \cite{Bur79} for ${}^2B_2$,
Shamash \cite{Sha89} for $G_2$, Geck \cite{Gec91bis} for ${}^3D_4$,
Hiss \cite{His90} for ${}^2G_2$ and ${}^2 F_4$, and Wings \cite{Win95} for $F_4$. The determination of the other trees were obtained by Hiss--L\"ubeck \cite{HisLub98} for $F_4$ and ${}^2E_6$, and Hiss--L\"ubeck--Malle \cite{HisLubMal95} for $E_6$.
In addition to the algebraic arguments used in the example of $G_2(q)$, it is often necessary to use partial information on the character table of the group to determine the tree (in order to decompose tensor products of characters).

\smallskip

The problem for larger exceptional groups such as $E_7$ and $E_8$ comes from the increasing number of cuspidal $kG$-modules. As in the case of decomposition matrices (see Chapter \ref{chap:dec}), these representations resist to algebraic methods which rely on Harish-Chandra induction and restriction, such as (Hecke). The idea developped in \cite{CraDudRou17} by Craven--Rouquier and the author is to use the cohomology complexes of Deligne--Lusztig varieties to produce perfect complexes satisfying the following proposition.

\begin{prop}\label{prop:twodegrees}
Let $C$ be a perfect complex of $kG$-modules. Assume that there exist integers $a < b$ such that 
$H^i(C) = 0$ for $i \neq a,b$. Then 
$$ H^a(C) \simeq \Omega^{b-a+1} H^b(C) \quad \text{in } kG\sfstab.$$
Consequently, the $kG$-modules $H^a(C)$ and $\Omega^{b-a+1} H^b(C)$ differ only by their projective summands.
\end{prop}

\begin{proof}
Using \eqref{eq:dttruncation} we get the following distinguished triangle in $\sfD^b(kG\sfmod)$
$$ H^a(C)[-a] \longrightarrow C \longrightarrow H^b(C)[-b] \rightsquigarrow$$
which we can also write as 
$$ C \longrightarrow H^b(C)[-b]  \longrightarrow H^a(C)[-a+1] \rightsquigarrow$$
Now by Theorem \ref{thm:derstab} the image in $kG\sfstab$ is also a distinguished triangle. Since $C$ is perfect, its image in the stable category is zero. Thus we obtain the following distinguished triangle in $kG\sfstab$
 $$ 0 \longrightarrow \Omega^bH^b(C)  \longrightarrow \Omega^{a-1} H^a(C) \rightsquigarrow$$
which yields $ \Omega^bH^b(C)  \simto \Omega^{a-1} H^a(C)$ in $kG\sfstab$.
\end{proof}

Given $\lambda \in \calO^\times$, we can consider the complex $C = bR\Gamma_c(\bfX(w),k)_{\overline{\lambda}}$ obtained from the cohomology complex of $\bfX(w)$ by cutting by the generalized $\overline{\lambda}$-eigenspace of $F$, and by the block $b$. In other words, the complex $C$ is isomorphic to a direct summand of $R\Gamma_c(\bfX(w),k)$ such that $H^i(C) = b H_c^i(\bfX(w),k)_{\overline{\lambda}}$. The requirements to use Proposition \ref{prop:twodegrees} are
\begin{itemize}
 \item[(1)] $C$ must be perfect. This follows from Proposition \ref{prop:ywperfect} if we assume $\ell \nmid |\bfT^{wF}|$.
 \item[(2)] The condition on the vanishing of $H^i(C)$ should already hold over $K$ by the universal coefficient formula (see \S\ref{sec:coxoverk}). This supposes a vanishing of many of the groups $H_c^i(\bfX(w),K)_{\mu}$ for every eigenvalue $\mu$ congruent to $\lambda$ modulo~$\ell$. 
 \item[(3)] The vanishing over $k$ should follow from the vanishing over $K$ and the property that each cohomology group $H_c^i(\bfX(w),\calO)$ is $\calO$-free.
\end{itemize}
In the few examples where $H_c^i(\bfX(w),\calO)$ has been explicitly computed, it is torsion-free whenever $\ell \nmid |\bfT^{wF}|$ (in other words (3) follows from (1)). The reader will find an example of this property in the following chapter, for varieties associated with Coxeter elements.

\begin{example}
Let $G$ be a finite reductive group of type $E_7$. We denote by $s_1,\ldots,s_7$ the simple reflections in $W$. If $q$ has order $14$ modulo $\ell$, in which case $\ell$ divises $\Phi_{14}(q)$ then the principal $\ell$-block has cyclic defect groups. It corresponds to the cuspidal pair $(\bfT_w,1)$ where $w$ is a element of $W$ of order $14$ (such an element can be taken to have length $9$). The fourteen non-exceptional characters in the block are given by the irreducible constituents of the Deligne--Lusztig character $R_w$. They consist of eight unipotent characters in the principal series (including $1$ and $\mathrm{St}$), four unipotent characters in the $D_4$-series and the two cuspidal unipotent characters of $E_7$, namely the complex conjugate characters $E_7[\mathrm{i}]$ and $E_7[-\mathrm{i}]$. The real stem of the Brauer tree is formed by the real characters (all except $E_7[\pm \mathrm{i}]$) with increasing dimension towards the exceptional vertex by (Hecke), see Figure \ref{E7d14real}. 
\begin{figure}[h]
\begin{center}
\begin{pspicture}(12,1.7)
  \psset{linewidth=1pt}
  \psset{unit = 1cm}
 
  \cnode[fillstyle=solid,fillcolor=black](4,1){5pt}{A2}
  \cnode(4,1){8pt}{A}
  \cnode(0,1){5pt}{B}
  \cnode(1,1){5pt}{C}
  \cnode(2,1){5pt}{D}
  \cnode(3,1){5pt}{E}
  \cnode(5,1){5pt}{F}
  \cnode(6,1){5pt}{G}
  \cnode(7,1){5pt}{H}
  \cnode(8,1){5pt}{I}
  \cnode(9,1){5pt}{J}
  \cnode(10,1){5pt}{K}
  \cnode(11,1){5pt}{L}
  \cnode(12,1){5pt}{M}  

  \ncline[nodesep=0pt]{B}{C}
  \ncline[nodesep=0pt]{C}{D}
  \ncline[nodesep=0pt]{D}{E}
  \ncline[nodesep=0pt]{E}{A}
  \ncline[nodesep=0pt]{A}{F}\naput[npos=1.35]{$\vphantom{\big)} \mathrm{St}$}
  \ncline[nodesep=0pt]{F}{G}
  \ncline[nodesep=0pt]{G}{H}  
  \ncline[nodesep=0pt]{H}{I}
  \ncline[nodesep=0pt]{I}{J}
  \ncline[nodesep=0pt]{J}{K}  
  \ncline[nodesep=0pt]{K}{L}
  \ncline[nodesep=0pt]{L}{M}\naput[npos=1.3]{$\vphantom{\big)} 1$}     

  \psbrace[nodesepA=-0.5cm,nodesepB=11pt,rot=90,braceWidthInner=5pt,braceWidthOuter=0.05pt,braceWidth=0.5pt](-0.15,0.8)(3.15,0.8){\footnotesize$D_4$-series}
  \psbrace[nodesepA=-0.9cm,nodesepB=11pt,rot=90,braceWidthInner=5pt,braceWidthOuter=0.05pt,braceWidth=0.5pt](4.85,0.8)(12.15,0.8){\footnotesize principal series}
  
\end{pspicture}
\end{center}
\caption{Real stem of Brauer tree of the principal $\Phi_{14}$-block of $E_7(q)$} 
\label{E7d14real}
\end{figure}
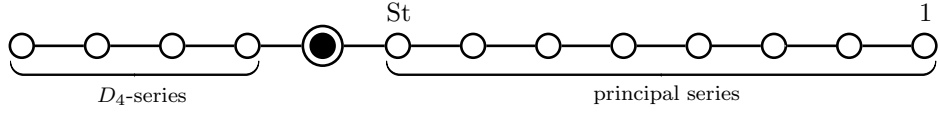

The missing characters in the tree are the complex conjugate characters $E_7[\pm\mathrm{i}]$. Unlike the case of $G_2(q)$ in Example \ref{ex:g2tree}, a combination of the (Degree) and (Parity) arguments is not enough to determine to which vertex they are attached. To remove the ambiguity, we consider the Deligne--Lusztig variety $\bfX(c)$ attached to a Coxeter element $c = s_1 s_2 \cdots s_7$ and the corresponding cohomology complex. Here we will consider the generalized eigenspaces of $F$ corresponding to the eigenvalues $1$ and $-1$. Over $K$, the cohomology of $\bfX(c)$, cut by the block, is
$$ \begin{array}
{r@{\, }c@{\, }c@{\, }l@{\, }l} bH_c^i(\bfX(c),K) = \big(& \mathrm{St} & \oplus \, E_7[\mathrm{i}]\,  \oplus & E_7[-\mathrm{i}] \big)[-7] & \oplus\, 1[-14] \\[5pt]
\text{with eigenvalues of $F$ in $K$} \hphantom{)} = \hphantom{\big(} & 1 & \ \mathrm{i}q^{7/2} & -\mathrm{i}q^{7/2} & \ \, \ q^7 \\
\text{and eigenvalues of $F$ in $k$} \hphantom{)} = \hphantom{\big(} & 1 &-1 &  \ \ \ 1 & \, -1 \end{array}$$
with the convention that $\mathrm{i} \equiv q^{7/2}$ modulo $\ell$ (here $\mathrm{i}^2 =-1$ in $K$). We obtain 
$$ \begin{aligned}
bH_c^\bullet(\bfX(c),K)_{1+\ell \calO} & \, \simeq \big(\mathrm{St} \oplus E_7[-\mathrm{i}]\big) [-7], \\
bH_c^\bullet(\bfX(c),K)_{-1+\ell \calO} & \, \simeq E_7[\mathrm{i}] [-7]  \oplus 1[-14]. 
\end{aligned}$$
If we assume that $H_c^\bullet(\bfX(c),\calO)$ is torsion-free, then the universal coefficient theorem shows that 
$bR\Gamma_c(\bfX(c),k)_{1}$ has only one non-zero cohomology group. By using the truncation functors of Proposition \ref{prop:truncation} it follows that it is quasi-isomorphic to a single projective $kG$-module in degree $7$, which lifts to a lattice of character $\mathrm{St} \oplus E_7[-\mathrm{i}]$. This shows that the vertex labeled by $E_7[-\mathrm{i}]$ is connected to the vertex labeled by $\mathrm{St}$ and gives the Brauer tree of $b$. Even better, using Proposition \ref{prop:twodegrees}  with the complex $bR\Gamma_c(\bfX(c),K)_{-1}$ which has only two non-zero cohomology groups yields
$$ H_c^7(\bfX(c),k)_{-1}  \simeq  \Omega^{14-7+1} H_c^{14}(\bfX(c),k)_{-1} \simeq  \Omega^{8} k$$
in the stable category $kG\sfstab$. This proves that $\Omega^8 \calO$ is an $\calO G$-lattice with character 
$H_c^7(\bfX(c),K)_{-1} = E_7[\mathrm{i}]$ and the planar embedded Brauer tree shown in Figure \ref{E7d14} is obtained from the Green walk.
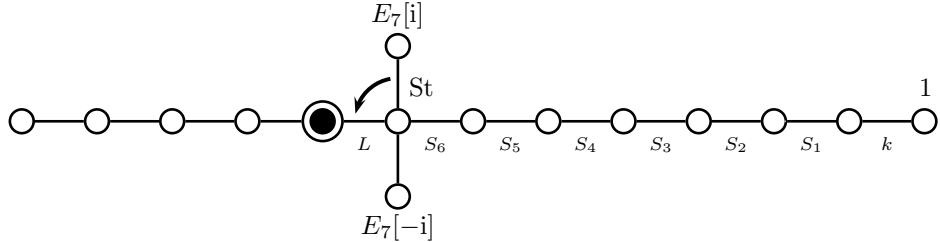
\begin{figure}[h]
\begin{center}
\begin{pspicture}(12,2.7)
  \psset{linewidth=1pt}
  \psset{unit = 1cm}
 
  \cnode[fillstyle=solid,fillcolor=black](4,1.3){5pt}{A2}
  \cnode(4,1.3){8pt}{A}
  \cnode(0,1.3){5pt}{B}
  \cnode(1,1.3){5pt}{C}
  \cnode(2,1.3){5pt}{D}
  \cnode(3,1.3){5pt}{E}
  \cnode(5,1.3){5pt}{F}
  \cnode(6,1.3){5pt}{G}
  \cnode(7,1.3){5pt}{H}
  \cnode(8,1.3){5pt}{I}
  \cnode(9,1.3){5pt}{J}
  \cnode(10,1.3){5pt}{K}
  \cnode(11,1.3){5pt}{L}
  \cnode(12,1.3){5pt}{M}  
  \cnode(5,0.3){5pt}{F1}  
  \cnode(5,2.3){5pt}{F2}  

  \ncline[nodesep=0pt]{B}{C}
  \ncline[nodesep=0pt]{C}{D}
  \ncline[nodesep=0pt]{D}{E}
  \ncline[nodesep=0pt]{E}{A}
  \ncline[nodesep=0pt]{A}{F}\naput[npos=1.9]{$\vphantom{\big)} \mathrm{St}$}\nbput[npos=0.5]{\scriptsize$L$}
  \ncline[nodesep=0pt]{F}{G}\nbput[npos=0.5]{\scriptsize$S_6$}
  \ncline[nodesep=0pt]{G}{H}\nbput[npos=0.5]{\scriptsize$S_5$}  
  \ncline[nodesep=0pt]{H}{I}\nbput[npos=0.5]{\scriptsize$S_4$}
  \ncline[nodesep=0pt]{I}{J}\nbput[npos=0.5]{\scriptsize$S_3$}
  \ncline[nodesep=0pt]{J}{K}  \nbput[npos=0.5]{\scriptsize$S_2$}
  \ncline[nodesep=0pt]{K}{L}\nbput[npos=0.5]{\scriptsize$S_1$}
  \ncline[nodesep=0pt]{L}{M}\naput[npos=1.3]{$\vphantom{\big)} 1$}\nbput[npos=0.5]{\scriptsize$k$}     
  \ncline[nodesep=0pt]{F}{F1}\ncput[npos=1.9]{$\vphantom{\big)} E_7[-\mathrm{i}]$}
  \ncline[nodesep=0pt]{F}{F2}\ncput[npos=1.9]{$\vphantom{\big)} E_7[\mathrm{i}]$}
 
  \psellipticarc[linewidth=1.5pt]{->}(5,1.3)(0.6,0.6){100}{170}
\end{pspicture}
\end{center}
\caption{Brauer tree of the principal $\Phi_{14}$-block of $E_7(q)$} 
\label{E7d14}
\end{figure}
\end{example}

\begin{exercise} Show that the perfect complex $R\Gamma_c(\bfX(c),k)_{-1}$ is homotopy equivalent to 
$$ 0 \longrightarrow \begin{array}{c} \textcolor{violet}{E_7[\mathrm{i}]} \\ \textcolor{violet}{L} \\ \textcolor{violet}{E_7[-\mathrm{i}]} \\ \textcolor{violet}{S_6} \\ \fbox{$E_7[\mathrm{i}]$} \end{array} \longrightarrow \begin{array}{c} \textcolor{purple}{S_6} \\ \textcolor{violet}{E_7[\mathrm{i}]} \hphantom{AAAA} \\ \textcolor{violet}{L} \qquad \textcolor{purple}{S_5} \\ \textcolor{violet}{E_7[-\mathrm{i}]} \hphantom{AAAA} \\ \textcolor{violet}{S_6} \end{array}  \hskip -4mm \longrightarrow \begin{array}{c} \textcolor{violet}{S_5} \\\textcolor{purple}{S_6} \quad \textcolor{violet}{S_4}  \\ \textcolor{purple}{S_5} \end{array} \longrightarrow \begin{array}{c} \textcolor{purple}{S_4} \\ \textcolor{violet}{S_5} \quad \textcolor{purple}{S_3}  \\ \textcolor{violet}{S_4} \end{array}  \longrightarrow \cdots \longrightarrow  \begin{array}{c} \fbox{$k$} \\ \textcolor{violet}{S_1} \\  \textcolor{violet}{k} \end{array} \longrightarrow 0.$$
\end{exercise}

This method has proven very powerful in the case of exceptional groups of type $E_7(q)$ and $E_8(q)$, removing also some ambiguity in the planar embedding of the Ree groups ${}^2F_4(q)$. Luckily, only small-dimensional Deligne--Lusztig varieties were needed to complete the determination of the trees, which made checking the vanishing of the cohomology possible.

\begin{theorem}[Craven--Dudas--Rouquier \cite{CraDudRou17}] All the planar embedded Brauer trees of unipotent blocks of finite reductive groups are explicitly known in terms of Lusztig's parametrization of unipotent characters.
\end{theorem}

\section{The Coxeter variety}

The last chapter of these notes is devoted to the study of Deligne--Lusztig varieties attached to a special class of Weyl group elements, called the Coxeter elements. These varieties were first studied by Lusztig \cite{Lus76bis}. Computing their $\ell$-adic cohomology turned out to be a key ingredient in Lusztig's subsequent work on the classification of ordinary irreducible characters of finite reductive groups. We explain in this chapter how to extend Lusztig's result to the modular setting, building on work of Bonnaf\'e--Rouquier \cite{BonRou08bis} and the author \cite{Dud12,Dud14}. 

\smallskip

Throughout this chapter we will assume for simplicity that $(\bfG,F)$ is \emph{split} (\emph{i.e.} $F$ acts trivially on the Weyl group). All the main results of this chapter can be generalized to the case where a power of $F$ acts trivially (which includes the case of the Ree and Suzuki groups).

 \subsection{Geometry of the Coxeter variety}
Let $\bfT$ be a split maximal torus of $\bfG$, contained in an $F$-stable Borel subgroup $\bfB$ of $\bfG$.
Recall that the simple reflections $s_1,\ldots,s_r$ of the Weyl group $W = N_\bfG(\bfT)/\bfT$ are exactly the elements $s \in W$ such that $\bfB s\bfB/\bfB$ has dimension $1$. 

\smallskip

A \emph{Coxeter element} $c$ of $W$ is a product $c = s_1 \cdots s_r$ of all the simple reflections, in any order. Coxeter elements are conjugate under $W$. The order of any Coxeter element is the \emph{Coxeter number} of $W$, denoted by $h$. A \emph{Coxeter variety} is a Deligne--Lusztig variety $\bfX( c)$ or  $\bfY(\dot c)$ attached to a Coxeter element. 
Since Coxeter elements are the elements of minimal length in their conjugacy class, they are actually conjugate by a sequence of cyclic shifts. Consequently, the cohomology of a Coxeter variety does not depend on the choice of a Coxeter element. For that reason we shall denote these varieties simply by $\bfX$ or $\bfY$. For more details on Coxeter elements see \cite[\S V.6]{Bki} and for Coxeter varieties see \cite{Lus76bis}. 

\begin{example}\label{ex:coxgln}
Let $\bfG = \mathrm{GL}_n(\overline{\mathbb{F}}_p)$. The permutation $c = (1,2)(2,3) \cdots (n-1,n) = (1,2,3,\ldots,n)$ is a Coxeter element of $\mathfrak{S}_n$. It has length $n-1$ and order $h = n$. For the general linear group the flag variety $\bfG/\bfB$ can be identified with the set of flags of vectors spaces $V_\bullet = (\{0\} = V_0 \subset V_1 \subset \cdots \subset V_n = 
\overline{\mathbb{F}}_p^n)$ such that $\dim_{\overline{\mathbb{F}}_p} V_i = i$.  With this description, the Deligne--Lusztig variety attached to an element $w \in \mathfrak{S}_n$ is
$$ \bfX(w) \simeq \{ V_\bullet \in \bfG/\bfB \, \mid \, V_\bullet \text{ and } F(V_\bullet) \text{ are in relative position } w\}.$$
Recall that two flags $V_\bullet$ and $V_\bullet'$ are said to be \emph{in relative position $w$} if there exist $e_1,\ldots,e_n \in\overline{\mathbb{F}}_p^n$ such that $e_1,\ldots,e_i$ is a basis of $V_i$ and $e_{w(1)},\ldots,e_{w(i)}$ is a basis of $V_i'$ for each $i = 0,\ldots,n$. With $w = c = (1,2,\ldots,n)$ we deduce that $V_\bullet \in \bfX$ if and only if $e_1,\ldots,e_i$ is a basis of $V_i$ and $e_2,\ldots,e_{i+1}$ is a basis of $F(V_i)$. This can be written 
$V_{i+1} = V_1 \oplus F(V_i)$, which yields in turn
$$V_{i+1} =  V_1 \oplus F(V_1) \oplus \cdots \oplus F^{i}(V_1).$$
This gives an explicit description of the Coxeter variety in the case of $\mathrm{GL}_n(\overline{\mathbb{F}}_p)$ as
$$\begin{aligned}
\bfX \simeq & \ \{ V_1 \in \mathbb{P}(\overline{\mathbb{F}}_p^n) \, \mid \,  \overline{\mathbb{F}}_p^n = V_1 \oplus F(V_1) \oplus \cdots \oplus F^{n-1}(V_1) \} \\ 
\simeq & \ \Big\{ [x_1:x_2:\cdots:x_n] \in \mathbb{P}_{n-1} \ \Big| \  \begin{array}[t]{|cccc|} 
x_1 & x_1^q & \cdots & x_1^{q^{n-1}} \\ 
x_2 & x_2^q & \cdots & x_2^{q^{n-1}} \\ 
\vdots & \vdots & & \hskip -4mm\vdots \\
x_n & x_n^q & \cdots & x_n^{q^{n-1}} \\ \end{array} \neq 0 \Big\}. \end{aligned}$$
\end{example}
 
If $v < c$ then $v$ is obtained from $c$ by removing some simple reflections. Therefore it is a product of simple reflections lying in a proper subset $I$ of $S$, and as such it is a Coxeter element of the parabolic subgroup $W_I$ of $W$. We shall write $v = c_I$. Consequently,
\begin{equation}\label{eq:compact} 
\overline{\bfX(c)} = \bigsqcup_{v \leq c} \bfX(v) =  \bigsqcup_{I \subset S} \bfX(c_I).
\end{equation}
Let $\bfP_I = \bfB W_I \bfB$ (resp. $\bfL_I$) be the standard parabolic subgroup of $\bfG$ (resp. standard Levi subgroup of $\bfG$) attached to $I$. Its unipotent radical will be denoted by $\bfU_I$. 
We will write $\bfX_I = \bfX_{\bfL_I}(c_I)$ for the Coxeter variety of $\bfL_I$. We have $\bfX(c_I) \simeq G/U_I \times_{L_I} \bfX_I$ and therefore 
\begin{equation}\label{eq:rlgcox}
R\Gamma_c( \bfX(c_I),\Lambda) \simeq R_{L_I}^G\big(R\Gamma_c( \bfX_I,\Lambda)\big).
\end{equation}

\smallskip There are two key properties of the Coxeter variety that are needed to compute its cohomology (and to show that it is torsion-free). The first one is given by \eqref{eq:compact}. The second is a result of Lusztig \cite{Lus76bis} giving the quotient of $\bfX$ by unipotent subgroups in terms of Coxeter varieties of Levi subgroups.

\begin{prop}\label{prop:quotientcox}
Let $I \subset S$. There is a (non-equivariant) isomorphism of varieties
$$ U_I \backslash \bfX \simeq \bfX_I \times (\mathbb{G}_m)^{r-|I|}$$
which induces an $L_I \times\langle F \rangle$-equivariant isomorphism of $\ell$-adic cohomology groups
$$ {}^*R_{L_I}^G \big(H_c^\bullet(\bfX,K)\big) \simeq  H_c^\bullet(\bfX_I,K) \otimes_K  H_c^{\bullet}(\mathbb{G}_m,K)^{\otimes r- |I|}. $$ 
\end{prop}

In particular, with $I$ being the empty set we deduce that $U \backslash \bfX \simeq (\mathbb{G}_m)^{r}$. One can actually refine Lusztig's result as follows  (see \cite[Lem. 4.2]{Dud14}): if $J = S \smallsetminus I$ is the complement of $I$ in $S$ then
\begin{equation}\label{eq:coxquotient}
(U_I \cap U_J) \backslash \bfX \simeq \bfX_I \times \bfX_J. 
\end{equation}
Again, this isomorphism is not equivariant for the action of $P_I$ or $P_J$ in general.

\subsection{Cohomology over $K$}
Using a combination of \eqref{eq:compact}, Proposition \ref{prop:quotientcox} and computations of $\#\bfX^{F^n}$ for $n = 1,\ldots,h$ (in order to use Lefschetz trace formula, see Theorem \ref{thm:lefschetz}) Lusztig gave in \cite{Lus76bis} a complete description of the cohomology of $\bfX$ over $K$, with the action of $G$ and $F$.

\begin{theorem}[Lusztig]\label{thm:coxchar0}\leavevmode
\begin{itemize}
\item[$\mathrm{(i)}$] A cuspidal character $\rho \in \mathrm{Irr}\, G$ occuring in the cohomology of $\bfX$ occurs in the middle degree $H_c^{r}(\bfX,K)$ only.
\item[$\mathrm{(ii)}$] The eigenspaces of $F$ on $H_c^\bullet(\bfX,K)$ give $h$ mutually non-isomorphic irreducible representations of $G$. 
\item[$\mathrm{(iii)}$] The eigenvalues of $F$ on $H_c^\bullet(\bfX,K)$, restricted to a given Harish-Chandra series,  are of the form
$$ \begin{array}{r|ccccccc} i & r & r+1 & \cdots & 2(r-m_\zeta)  \\\hline
H_c^i(\bfX,K)\vphantom{\big)^i} & \zeta q^{m_\zeta} & \zeta q^{m_\zeta+1} & \cdots & \zeta q^{r-m_\zeta} \\
\end{array}$$
for some root of unity $\zeta \in \calO^\times$ and some $m_\zeta \in \frac{1}{2} \mathbb{Z}_{\geq 0}$.  No constituent of $H_c^i(\bfX,K)$ for $i > 2(r-m_\zeta)$ lie in that Harish-Chandra series.
\end{itemize}
\end{theorem}

\begin{proof}[Sketch of proof]
If $v < c$ then $v = c_I$ for some proper subset $I$ of $S$ and from \eqref{eq:rlgcox} we get $H_c^\bullet( \bfX(v),K) \simeq R_{L_I}^G\big(H_c^\bullet( \bfX_I,K)\big)$. Therefore $H_c^\bullet( \bfX(v),K)$ contains no cuspidal character. By Theorem \ref{thm:mainBR}, it follows that $\rho$ cannot be a constituent of $H_c^i(\bfX,K)$ for $i \neq \ell(w) = r$, which proves (i). In addition, one can show using the purity of $\overline{\bfX}$ that
$F$ has eigenvalue $\zeta q^{r/2}$ with $|\zeta| = 1$ on the $\rho$-isotypic part of $H_c^r(\bfX,K)$. Furthermore, it is a general property of the cohomology of Deligne--Lusztig varieties that $\zeta$ is actually a root of unity. 

\smallskip
Let $\rho$ be a cuspidal character of $L_I$ for some $I \subset S$ and set $m = |I|$. The eigenvalue of $F$ on the $\rho$-isotypic part of $H_c^{m}(\bfX_I,K)$ is of the form $\zeta q^{m/2}$ for some root of unity $\zeta \in \calO^\times$. If $\chi$ is an irreducible character of $G$ lying in the Harish-Chandra series of $(L_I,\rho)$ then ${}^*R_{L_I}^G(\chi)$ is a non-zero multiple of $\rho$. In particular, the eigenvalues of $F$ on the Harish-Chandra series of $(L_I,\rho)$ correspond to the eigenvalues of $F$ on the $\rho$-isotypic part of 
${}^*R_{L_I}^G \big(H_c^\bullet(\bfX,K)\big)$, which by Proposition \ref{prop:quotientcox} and (i) are $\zeta q^{m/2}$ times the eigenvalues of $F$ on the cohomology of $(\mathbb{G}_m)^{r-m}$. Assertion (iii) follows.
 
\smallskip
The proof of (ii) requires to compute the endomorphism algebra of $H_c^\bullet(\bfX)$, which would go beyond the scope of these notes. 
\end{proof}

\begin{example}
(a) Let $\bfG = \mathrm{GL}_n(\overline{\mathbb{F}}_p)$. Recall that the unipotent characters of $G$ are parametrized by partitions of $n$. We will represent them by their Young diagram. Then the cohomology of the Coxeter variety (given in Example \ref{ex:coxgln}) with the eigenvalues of $F$ is
$$ \begin{array}{r|cccccc} i & n-1 & n & n+1 & \cdots & 2n-3 & 2n-2\\\hline
H_c^i(\bfX) \vphantom{\Big)}& \Yboxdim{7pt} \gyoung(;,|2\vdts,;) \ (1) & \Yboxdim{7pt} \gyoung(;;,|2\vdts,;) \ (q) &  \Yboxdim{7pt}\gyoung(;;;,|2\vdts,;)\ (q^2) & \cdots &  \Yboxdim{7pt} \gyoung(;_2\hdts;,;) \ (q^{n-2}) &  \Yboxdim{7pt} \gyoung(;_2\hdts;)  \ (q^{n-1}) 
\end{array}$$

\noindent 
(b) Let $\bfG$ a group of type $F_4$. Using the notation in \cite[\S 13]{Car} for the unipotent characters of $G$ (in particular $\theta$ is a primitive third root of $1$ and $\mathrm{i}$ a primitive fourth root of $1$), the cohomology of $\bfX$ is given by
$$ \begin{array}{r|ccccc}
i & 4 & 5 & 6 & 7 & 8 \\\hline
H_c^i(\bfX) \vphantom{\Big)}& \mathrm{St} \, (1) & \phi_{4,13} \, (q) & \phi_{6,6}'' \, (q^2) & \phi_{4,1} \, (q^3) & 1 \, (q^4)\\[5pt]
& B_{2,\varepsilon} \, (-q) & B_{2,r} \, (-q^2) & B_{2,1} \, (-q^3) \tikzmark{item-b2}\\[5pt]
& F_4[\pm \mathrm{i}] \,  (\pm \mathrm{i}q^2) \tikzmark{item-cusp}\\[5pt]
& F_4[\theta] \, (\theta q^2)\\[5pt]
& F_4[\theta^2] \,  (\theta^2 q^2) \\ \end{array}$$
\begin{tikzpicture}[remember picture,overlay]
\draw [decoration={brace,amplitude=0.5em},decorate,black]
  ([shift={(3pt,7pt)}]pic cs:item-cusp) |- 
  ++(5pt,-40pt);
\draw ([shift={(8pt,-15pt)}]pic cs:item-cusp) node[right,text width=6cm] 
    {\footnotesize cuspidal characters
    }
;
\draw [decoration={brace,amplitude=0.3em},decorate,black]
  ([shift={(3pt,9pt)}]pic cs:item-b2) |- 
  ++(5pt,-10pt);
\draw ([shift={(6pt,2pt)}]pic cs:item-b2) node[right,text width=6cm] 
    {\footnotesize series above $B_2$
    }
;
\end{tikzpicture}
\end{example}

\subsection{Cohomology over $k$}\label{sec:coxoverk}
Since $k$ is not flat over $\calO$, the mod-$\ell$ cohomology of a variety is not the $\ell$-reduction of the cohomology over $\calO$, but there is still an explicit relation, called the \emph{universal coefficient theorem}, given by the following exact sequence (see for example \cite[\S 2.7]{Ben})
$$ 0 \longrightarrow k \otimes_\calO H_c^i(\bfX,\calO) \longrightarrow H_c^i(\bfX,k) \longrightarrow 
\mathrm{Tor}_1^\calO(H_c^{i+1}(\bfX,\calO),k) \longrightarrow 0.$$
We will use it in the following particular case, for which we can give a direct proof.

\begin{prop}\label{prop:middle-torsionfree}
The middle cohomology group $H_c^r(\bfX,\calO)$ of the Coxeter variety $\bfX$ is torsion-free.
\end{prop}

\begin{proof}
We consider an $\calO$-free resolution of $R\Gamma_c(\bfX,\calO)$ which we truncate using $\widetilde \tau_{\geq r}$ (see \S\ref{sec:truncation} for the definition of the truncation). This yields quasi-isomorphisms
$$ \begin{aligned}  R\Gamma_c(\bfX,\calO) \simeq \,& (0 \longrightarrow C_{r-1} \mathop{\longrightarrow}\limits^d C_r \mathop{\longrightarrow}\limits^{d'} \cdots) \phantom{\longrightarrow\longrightarrow}\\ 
R\Gamma_c(\bfX,k) \simeq k \, {\mathop{\otimes}\limits^L}_\calO\, R\Gamma_c(\bfX,\calO) \simeq\,  & (0 \longrightarrow kC_{r-1} \mathop{\longrightarrow}\limits^{\overline d} kC_r \mathop{\longrightarrow}\limits^{\overline d'} \cdots) \end{aligned}$$
where each $C_i$ is $\calO$-free. Furthermore, since $H_c^{r-1}(\bfX,k)=0$ the map $\overline{d}$ is injective (note that $d$ is injective by definition). On the other hand, since $C_r$ and therefore $\mathrm{Ker}\, d'$ is $\calO$-free the exact sequence $0 \longrightarrow C_{r-1} \longrightarrow \mathrm{Ker}\, d' \longrightarrow H_c^r(\bfX,\calO) \longrightarrow 0$ tensored with $k$ yields an exact sequence 
$$ 0 \longrightarrow \mathrm{Tor}_1^\calO(H_c^r(\bfX,\calO),k) \longrightarrow kC_{r-1} \mathop{\longrightarrow}\limits^{\overline d} k\mathrm{Ker}\, d' \longrightarrow kH_c^r(\bfX,\calO) \longrightarrow 0$$ 
which forces $\mathrm{Tor}_1^\calO(H_c^r(\bfX,\calO),k) = 0$.
\end{proof}

\begin{theorem}\label{thm:coxlbig}
Assume that $\ell \nmid |G|$. Then $H_c^\bullet(\bfX,\calO)$ is torsion-free.
\end{theorem}

\begin{proof}
We proceed by induction on the semisimple rank of $\bfG$. If $\bfG$ is a torus, then $\bfX$ is a point and the result holds. Otherwise, if $S$ is non-empty, one can consider a proper subset $I$ of $S$. Using the isomorphism of varieties $U_I \backslash \bfX \simeq \bfX_I \times (\mathbb{G}_m)^{r-|I|}$ given in Proposition \ref{prop:quotientcox} we get an isomorphism of $\calO$-modules 
$$ {}^*R_{L_I}^G \big(H_c^\bullet(\bfX,\calO)\big) \simeq  H_c^\bullet(\bfX_I,\calO) \otimes_K  H_c^{\bullet}(\mathbb{G}_m,\calO)^{\otimes r- |I|}. $$
Note that we do not assume this isomorphism to be $L_I$-equivariant. By induction, the cohomology of $\bfX_I$ (a Coxeter variety for the Levi subgroup $\bfL_I$) is torsion-free. This shows that the torsion part of $H_c^\bullet(\bfX,\calO)$ is killed under Harish-Chandra restriction, and hence it is cuspidal. 

\smallskip

By the universal coefficient formula, a cuspidal $\calO G$-submodule of $H_c^i(\bfX,\calO)$ yields a subquotient of $H_c^{i}(\bfX,k)$. Let $m$ be the largest degree of $H_c^\bullet(\bfX,k)$ which has a cuspidal subquotient $M$ (or equivalently since $kG$ is semisimple, a direct summand). If $m > r$ then Theorem \ref{thm:mainBR} forces the $kG$-module $M$ to occur in the cohomology of a Deligne--Lusztig variety $\bfX(v)$ for $v<w$. But this is impossible by \eqref{eq:rlgcox}. This, together with Proposition \ref{prop:middle-torsionfree}, shows that the cohomology of $\bfX$ is free over $\calO$.
\end{proof}

\noindent {\bf Question.} Does the result hold for other Deligne--Lusztig varieties $\bfX(w)$? If so, can we replace the condition $\ell \nmid |G|$ by $\ell \nmid |\bfT^{\dot w F}|$?

\smallskip

From now on we assume that $\ell$ divides $\Phi_h(q)$, the $h$-th cyclotomic polynomial evaluated at $q$. We will also assume that $\ell \nmid h$ so that $h$ is actually the order of $q$ modulo $\ell$. We first observe from the explicit values of the eigenvalues of $F$ given in \cite[Table 7.3]{Lus76bis} that:
\begin{itemize}
\item[$\mathrm{(i)}$] The classes in $k$ of the $h$ eigenvalues of $F$ on $H_c^\bullet(\bfX,K)$ are exactly the $h$-th roots of unity in $k$.
\end{itemize}

\noindent
Under the assumption on $\ell$, all the proper standard Levi subgroups of $G$ are $\ell'$-groups. Consequently, the proof of Theorem \ref{thm:coxlbig} shows that
\begin{itemize}
\item[$\mathrm{(ii)}$] The torsion-part of $H_c^i(\bfX,\calO)$ is a cuspidal $\calO G$-module.
\end{itemize}

\noindent 
Following Theorem \ref{thm:coxchar0}.iii, let $\lambda_\zeta = \zeta q^{m_\zeta}$ (resp. $\mu_\zeta = \zeta^{-1} q^{r-m_\zeta}$) be the smallest 
(resp. largest) eigenvalue of $F$ on $H_c^\bullet(\bfX,K)$ within the Harish-Chandra series corresponding to $\zeta$ (resp. to $\zeta^{-1}$). By Theorems \ref{thm:coxchar0} and \ref{thm:coxlbig} for $\bfX_I$, together with (i) we deduce the following property

\begin{itemize}
\item[$\mathrm{(iii)}$] Let $I$ be a proper subset of $S$. Then the generalized eigenspaces of $F$ on the cohomology of $\bfX_I$ for the eigenvalues $\overline{\mu}_\zeta$ and $\overline{\lambda}_\zeta$ satisfy $H_c^\bullet(\bfX_I,k)_{\overline{\mu}_\zeta} = 0$ and $H_c^i(\bfX_I,k)_{\overline{\lambda}_\zeta} = 0$ for $i \neq |I|$.
\end{itemize}

\noindent Using this observation, we get, for $i>r$
$$ H_c^i(\bfX,k)_{\overline{\lambda}_\zeta} \simeq \tikzmark{iii} H_c^i(\overline{\bfX},k)_{\overline{\lambda}_\zeta}
\simeq\tikzmark{verdier} \big(H_c^{2r-i}(\overline{\bfX},k)_{\overline{\mu}_\zeta}\big)^* \simeq \tikzmark{iiibis} \big(H_c^{2r-i}(\bfX,k)_{\overline{\mu}_\zeta}\big)^*$$
\begin{tikzpicture}[remember picture,overlay]
\draw[->,>=latex]
  ([shift={(-6pt,-7pt)}]pic cs:iii) |- 
  ++(-10pt,-13pt) 
  node[left,text width=1cm] 
    {\footnotesize by (iii)};
\draw[->,>=latex]
  ([shift={(-6pt,-7pt)}]pic cs:verdier) |- 
  ++(10pt,-13pt) 
  node[right,text width=4cm] 
    {\footnotesize by Poincar\'e duality};
\draw[->,>=latex]
  ([shift={(-7pt,7pt)}]pic cs:iiibis) |- 
  ++(10pt,+13pt) 
  node[right,text width=2cm] 
    {\footnotesize by (iii)};
\end{tikzpicture}
\vskip2mm
\noindent 
which is zero since $2r-i < r$. This proves that
\begin{equation}\label{eq:xc}
R\Gamma_c(\bfX,k)_{\overline{\lambda}_\zeta} \simeq H_c^r(\bfX,k)_{\overline{\lambda}_\zeta}[-r] \quad \text{in } \sfD^b(kG\sfmod).
\end{equation}
In addition, the universal coefficient formula shows that $H_c^r(\bfX,k)_{\overline{\lambda}_\zeta}$ is the mod-$\ell$ reduction of the $KG$-module $H_c^r(\bfX,K)_{\lambda_\zeta}$. This information can be used in combination with the following result, which holds for a more general class of Deligne--Lusztig varieties.

\begin{theorem}[Dudas--Rouquier \cite{DudRou14}]\label{thm:xwinstab}
Let $m \in \mathbb{Z}$ and $\overline{q}$ be the class of $q$ in $k$. Then 
$$R\Gamma_c(\bfX,k)_{\overline{q}^m} \simeq \Omega^{2m} k$$
in $kG\sfstab$. 
\end{theorem} 

\begin{proof}[Idea of proof] One can compute explicitly the closed subvariety $\bfX_\ell$ of $\bfX$ consisting of the points $x \in \bfX$ such that $\ell$ divides the order of $\mathrm{Stab}_G(x)$. Then the cohomology complexes of $\bfX_\ell$ and $\bfX$ are isomorphic in $kG\sfstab$.
\end{proof}

Choose $m \in \mathbb{Z}$ such that $\lambda_\zeta \equiv q^m$ modulo $\ell$. Then Theorem \ref{thm:xwinstab} and \eqref{eq:xc} show that $H_c^r(\bfX,k)_{\overline{\lambda}_\zeta} \simeq \Omega^{2m-r}k$ up to projective summands. This proves that $\Omega^{2m-r}k$ lifts to an $\calO G$-lattice with character $H_c^r(\bfX,K)_{\lambda_\zeta}$, yielding information on the Green walk around the Brauer tree of the principal $\ell$-block when $\ell \mid \Phi_h(q)$. 

\begin{example}
Let $\bfG$ be a group of type $F_4$, so that $h=12$. Assume that $q$ has order $12$ modulo $\ell$. We choose $\theta$ (resp. $\mathrm{i}$) to be congruent to $q^4$ (resp. $q^3$) modulo $\ell$. The various data attached to the representations occurring in the cohomology group of $\bfX$ in middle degree are listed in the following table.
$$\begin{array}{c|cccccc}
\zeta & 1 & -1 & \mathrm{i} & -\mathrm{i} & \theta & \theta^2\\
\lambda_\zeta & 1 & -q & \mathrm{i}q^2 & -\mathrm{i} q^2 & \theta q^2 & \theta^2 q^2 \\
q^m & q^0 & q^7 & q^5 & q^{11} & q^{6} & q^{10} \\
2m-r & -4 & 10 & 6 & 18 & 8 & 16 \\
H_c^r(\bfX)_\lambda & \mathrm{St} & B_{2,\varepsilon} & F_4[\mathrm{i}] & F_4[-\mathrm{i}] & 
F_4[\theta] & F_4[\theta^2] \\
\end{array}$$
We get therefore $[\Omega^{-4} \calO] = \mathrm{St}$,  $[\Omega^{10} \calO] = B_{2,\varepsilon}$,
 $[\Omega^{6} \calO] = F_4[\mathrm{i}]$, $[\Omega^{8} \calO] = F_4[\theta]$ and the planar embedded Brauer tree is 
given in Figure~\ref{F4d12}.
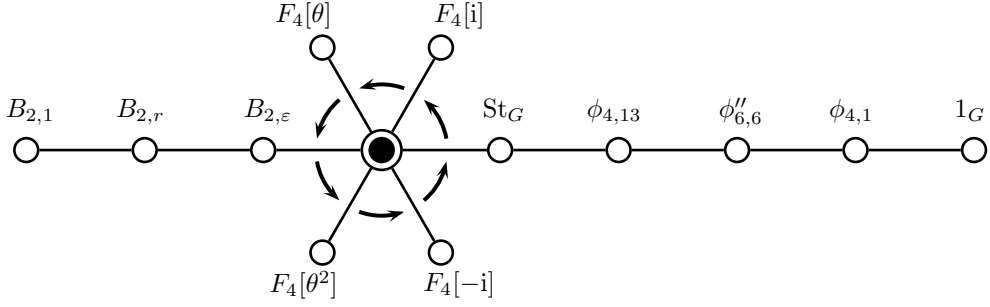
\begin{figure}[h] 
\begin{center}
\begin{pspicture}(14.1,4)

  \psset{linewidth=1pt}
  \psset{unit = 0.9cm}
 
  \cnode[fillstyle=solid,fillcolor=black](5.25,2){5pt}{A2}
  \cnode(5.25,2){8pt}{A}
  \cnode(7,2){5pt}{B}
  \cnode(8.75,2){5pt}{C}
  \cnode(10.5,2){5pt}{D}
  \cnode(12.25,2){5pt}{E}
  \cnode(14,2){5pt}{F}
  \cnode(3.5,2){5pt}{G}
  \cnode(1.75,2){5pt}{H}
  \cnode(0,2){5pt}{I}
  \cnode(6.125,3.515){5pt}{M}
  \cnode(6.125,0.485){5pt}{N}
  \cnode(4.375,3.515){5pt}{O}
  \cnode(4.375,0.485){5pt}{P}

  \ncline[nodesep=0pt]{A}{B}
  \ncline[nodesep=0pt]{B}{C}\naput[npos=-0.1]{$\vphantom{\Big(} \mathrm{St}_G$}\naput[npos=1.1]{$\vphantom{\Big(} \phi_{4,13}$}
  \ncline[nodesep=0pt]{C}{D}
  \ncline[nodesep=0pt]{D}{E}\naput[npos=-0.1]{$\vphantom{\Big(} \phi_{6,6}''$}\naput[npos=1.1]{$\vphantom{\Big(} \phi_{4,1}$}
  \ncline[nodesep=0pt]{E}{F}\naput[npos=1.1]{$\vphantom{\Big(} 1_G$}  
  \ncline[nodesep=0pt]{I}{H}\naput[npos=-0.1]{$\vphantom{\Big(} B_{2,1}$}\naput[npos=1.1]{$\vphantom{\Big(} B_{2,r}$}
  \ncline[nodesep=0pt]{H}{G}
  \ncline[nodesep=0pt]{G}{A}\naput[npos=-0.1]{$\vphantom{\Big(} B_{2,\varepsilon}$}
  \ncline[nodesep=0pt]{M}{A}\ncput[npos=-0.6]{$F_4[\mathrm{i}]$}
  \ncline[nodesep=0pt]{N}{A}\ncput[npos=-0.6]{$F_4[-\mathrm{i}]$}
  \ncline[nodesep=0pt]{O}{A}\ncput[npos=-0.6]{$F_4[\theta]$}
  \ncline[nodesep=0pt]{P}{A}\ncput[npos=-0.6]{$F_4[\theta^2]$}

  \psellipticarc[linewidth=1.5pt]{->}(5.25,2)(1,1){10}{50}
  \psellipticarc[linewidth=1.5pt]{->}(5.25,2)(1,1){70}{110}
  \psellipticarc[linewidth=1.5pt]{->}(5.25,2)(1,1){130}{170}
  \psellipticarc[linewidth=1.5pt]{->}(5.25,2)(1,1){190}{230}
  \psellipticarc[linewidth=1.5pt]{->}(5.25,2)(1,1){250}{290}
  \psellipticarc[linewidth=1.5pt]{->}(5.25,2)(1,1){310}{350}

\end{pspicture}
\end{center}
\caption{Brauer tree of the principal $\Phi_{12}$-block of $F_4(q)$} 
\label{F4d12}
\end{figure}
\end{example}

Knowing only the characters of the PIMs (in other words, the Brauer tree without the planar embedding) we can finally show that the cohomology of $\bfX$ over $\calO$ is torsion-free when $\ell \mid \Phi_h(q)$. 

\begin{theorem}\label{thm:coxcharl}
Assume $\ell \mid \Phi_h(q)$ and $\ell \nmid h$, so that $q$ has order $h$ modulo $\ell$. Then 
$H_c^\bullet(\bfX,\calO)$ is torsion-free.
\end{theorem}

\begin{proof}[Idea of proof] 
Since the torsion part of the cohomology of $\bfX$ over $\calO$ is cuspidal (see (ii) above), it is enough to show that for every simple cuspidal $kG$-module $M$, the complex $R\mathrm{Hom}_{kG}(P_M,R\Gamma_c(\bfX,k))$ has zero cohomology outside the middle degree $r = \ell(c)$. Indeed, by the universal coefficient formula this shows that the torsion-part of $R\Gamma_c(\bfX,\calO)$ is in degree $r$ only, and we can invoke Proposition \ref{prop:middle-torsionfree} to conclude.

\smallskip
The computation of $R\mathrm{Hom}_{kG}(P_M,R\Gamma_c(\bfX,k))$ is achieved by using the explicit character of $P_M$ as follows. The shape of the tree, as conjectured in \cite{HisLubMal95} and proved in \cite{DudRou14} ensures that $M$ labels an edge incident to the exceptional vertex. In other words, we have $e([P_M]) = \chi_\text{exc} + \chi$. 

\smallskip
We distinguish two cases : assume first that $\chi$ is cuspidal, then $P_M$ does not occur in any of the cohomology complexes $R\Gamma_c(\bfX(v),k)$ for $v < c$ since otherwise $\chi$ would occur in $H_c^\bullet(\bfX(v),K)$ (recall that $H_c^\bullet(\bfX(v),\calO)$ is torsion-free).  Consequently, the map
$R\Gamma_c(\bfX,k) \longrightarrow R\Gamma(\bfX,k)$ induces an isomorphism
$$ R\mathrm{Hom}_{kG}(P_M,R\Gamma_c(\bfX,k)) \simto R\mathrm{Hom}_{kG}(P_M,R\Gamma(\bfX,k))$$
 which proves that the cohomology of  this complex vanishes in degrees higher than $\dim \bfX = r$ and lower than 
 $\dim \bfX$.
 
\smallskip Assume now that $\chi$ lies in a Harish-Chandra series above a proper Levi subgroup $L_I$ of $G$. Writing $S = I \sqcup J$, one shows that $P_M$ is a direct summand of $R = \mathrm{Ind}_U^G \mathrm{Inf}_{U/U_I\cap U_J}^U (Q)$ for some (projective) $kU/(U_I\cap U_J)$-module $Q$. Now, with $\widetilde{Q}$ (resp. $\widetilde R$) being a lattice lifting $Q$, we obtain the following isomorphism using adjunction and \eqref{eq:coxquotient}
 $$\begin{aligned}
  R\mathrm{Hom}_{\calO G}(\widetilde{R},R\Gamma_c(\bfX,\calO))
 & \, \simeq  R\mathrm{Hom}_{\calO U/U_I\cap U_J}(\widetilde{Q},R\Gamma_c(\bfX,\calO)^{U_I\cap U_J})\\
 & \, \simeq  R\mathrm{Hom}_{\calO U/U_I\cap U_J}(\widetilde{Q},R\Gamma_c(\bfX_I,\calO)\otimes R\Gamma_c(\bfX_J,\calO)). \end{aligned}$$
The cohomology of this complex is torsion-free by Theorem \ref{thm:coxlbig}. Therefore the same holds for $R\mathrm{Hom}_{kG}(\widetilde{P}_M,R\Gamma_c(\bfX,\calO))$. Now its cohomology over $K$ vanishes outside of the degree $r$ since $\chi$ occurs in the cohomology of $\bfX$ in middle degree only, and by the universal coefficient theorem the same holds over $k$.  
\end{proof}

\subsection{Applications} 
Brou\'e's abelian defect group conjecture \cite{Bro90} predicts the existence of a derived equivalence between a block of a finite group with abelian defect and its Brauer correspondent. In the case of finite reductive groups, defect groups of unipotent blocks are generic. When $\ell$ is large enough and $d$ is the order of $q$ modulo $\ell$, they correspond to the $\ell$-part of $\Phi_d$-tori in $\bfG$, and their centralizers are $d$-Levi subgroups (see \S\ref{sec:caseofunipotentblocks}). Brou\'e suggested in \cite{Bro90} that in this case the derived equivalence should be induced by the cohomology complex of a Deligne--Lusztig variety associated with such a $d$-Levi subgroup. This was proven by Bonnaf\'e--Rouquier and the author in the case when $d=h$ is the Coxeter number.

\begin{theorem}[Bonnaf\'e--Rouquier \cite{BonRou08bis}, Dudas \cite{Dud12}]\label{thm:derivedeq}
Assume that $q$ has order $h$ modulo $\ell$. The action of $\bfT^{\dot cF}$ on $R\Gamma_c(\bfY,\calO)$ can be extended to an action of $N_{\bfG^{\dot c F}}(\bfT)$ such that the functor
$$R\Gamma_c(\bfY,\calO) \, {\mathop{\otimes}\limits^L}_{\calO N_{\bfG^{\dot c F}}(\bfT)} - : \sfD^b(\calO N_{\bfG^{\dot c F}}(\bfT)\sfmod) \longrightarrow \sfD^b(\calO G\sfmod)$$
induces a derived equivalence between the principal blocks of $N_{\bfG^{\dot c F}}(\bfT)$ and $G$.
\end{theorem}

The extension of the action of $\bfT^{\dot cF}$ to $N_{\bfG^{\dot c F}}(\bfT)$ is essentially given by twisting the action of the Frobenius endomorphism, since here  $N_{\bfG^{\dot c F}}(\bfT)/\bfT^{\dot cF}$ is a cyclic group generated by the image of $F$. For more general blocks, it is expected that the braid operators constructed in \cite{BroMic,DigMic14} should provide the extension of the action of the finite torus (see \cite{BroMal}). 

\smallskip

Once we extended the action, the key point is to prove that $R\Gamma_c(\bfY,\calO)$ is a tilting complex, that is that 
$$ R\mathrm{End}_{\calO G}(R\Gamma_c(\bfY,\calO)) \simeq \calO N_{\bfG^{\dot c F}}(\bfT) [0].$$
Rouquier proved in \cite{Rou02} that it is enough to show the vanishing of the cohomology groups of that complex over $k$, that is to show that 
\begin{equation}\label{eq:disjoint}
R\mathrm{Hom}_{k G}(R\Gamma_c(\bfY,k),R\Gamma_c(\bfY,k)[n]) \simeq 0 \quad \text{if } n \neq 0.\end{equation}
The solution to this problem given in \cite{BonRou08} and \cite{Dud12,Dud14} relies on showing first that the cohomology groups $H_c^\bullet(\bfY,\calO)$ are torsion-free and then to use the cohomology over $K$ (computed by Lusztig) to determine $H_c^\bullet(\bfY,k)$ and find an explicit representative for $R\Gamma_c(\bfY,k)$. 

\smallskip

\begin{prop}  Let $\chi_\lambda$ be the unipotent character corresponding to the generalized $\lambda$-eigenspace of $F$ on $H_c^{n_\lambda}(\bfX,K)$ for some $n_\lambda \geq r$. With the following notation for the subtree of the Brauer tree of the principal $\ell$-block of $kG$ corresponding to the Harish-Chandra series of $\chi_\lambda$
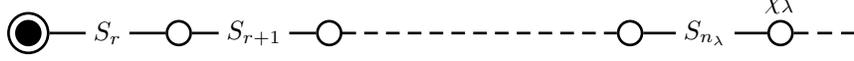
\begin{center}
\begin{pspicture}(11,1)
  \psset{linewidth=1pt}
  \cnode(0,0.5){8pt}{A}  \cnode[fillcolor=black,fillstyle=solid](0,0.5){5pt}{A2}
  \cnode(2,0.5){5pt}{B}
  \cnode(4,0.5){5pt}{C}
  \cnode(8,0.5){5pt}{D}
  \cnode(10,0.5){5pt}{E}
  \cnode[linestyle=none](11,0.5){0pt}{F}
  
  \ncline[nodesep=0pt]{A}{B}\ncput*[npos=0.5]{$S_r$}
  \ncline[nodesep=0pt]{B}{C}\ncput*[npos=0.5]{$S_{r+1}$}
  \ncline[nodesep=0pt,linestyle=dashed]{C}{D}
  \ncline[nodesep=0pt]{D}{E}\ncput*[npos=0.5]{$S_{n_\lambda}$}\naput[npos=1.1]{$\vphantom{\big)}\chi_{\lambda}$}
  \ncline[nodesep=0pt, linestyle=dashed]{E}{F}
\end{pspicture}
\end{center}
the complex $R\Gamma_c(\bfY,k)_{\overline{\lambda}}$ is isomorphic to
$$ 0 \longrightarrow P_{S_r} \longrightarrow P_{S_{r-1}} \longrightarrow \cdots \longrightarrow P_{S_{n_{\lambda}}} \longrightarrow 0.$$
\end{prop}

\begin{proof}[Sketch of proof]
Lusztig's result on the quotient of $\bfX$ (Proposition \ref{prop:quotientcox}) can be generalized to $\bfY$ as follows. Let us decompose the torus $\bfT^{\dot cF}$ as $\bfT^{\dot cF} = (\bfT^{\dot cF})_\ell \times (\bfT^{\dot cF})_{\ell'}$ as a product of an $\ell$-group and an $\ell'$-group. We define $\bfY_\ell \simeq \bfY/(\bfT^{\dot c F})_{\ell'}$. It is an intermediate quotient between $\bfY$ and $\bfX \simeq \bfY/\bfT^{\dot cF}$ whose cohomology contains only the principal $\ell$-block of $\bfT^{\dot cF}$. Then one shows that for every proper subset $I$ of $S$ there is an isomorphism of $\calO$-modules 
$$ {}^*R_{L_I}^G \big(H_c^\bullet(\bfY_{\ell},\calO)\big) \simeq  H_c^\bullet(\bfX_I,\calO) \otimes_\calO  H_c^{\bullet}(\mathbb{G}_m,\calO)^{\otimes r- |I|}. $$ 
In particular, the torsion part of the cohomology of $\bfY_{\ell}$ is cuspidal. As in the case of $\bfX$, it is enough to show that cuspidal modules occur in the middle degree only. This was proven for $\bfX$ along the way to Theorem \ref{thm:coxcharl}. The same property holds for $\bfY$ since $R\Gamma_c(\bfY_{\ell},k) \simeq R\Gamma_c(\bfX,k(\bfT^{\dot cF})_\ell )$ is built from successive extensions of 
$R\Gamma_c(\bfX,k)$ in the same way that $k(\bfT^{\dot cF})_\ell$ is built from extensions of the trivial representation.

\smallskip

By definition, the cohomology groups $H_c^i(\bfY_\ell,K)$ are the sum of the cohomology groups $H_c^i(\bfY,K)_\theta$ where $\theta$ runs over the irreducible $\ell$-characters of $\bfT^{\dot c F}$. When $\theta =1$,  $H_c^i(\bfY,K)_\theta = H_c^i(\bfX,K)$ which we know explicitly. When $\theta$ is non-trivial, the assumption on $\ell$ forces $\theta$ to be in general position and  $H_c^i(\bfY,K)_\theta = 0$ except when $i = r = \ell(c)$ in which case it equals the exceptional character $\pm R_w(\theta)$. Consequently, a given eigenspace of $F$ on $H_c^i(\bfY_{\ell},K)$ in non-zero in at most two degrees, one corresponding to the eigenspace on  
$H_c^i(\bfX,K)$, and the other being the middle degree. Since $R\Gamma_c(\bfY_\ell,\calO)$ is a direct summand of the perfect complex $R\Gamma_c(\bfY,\calO)$ (recall that it corresponds to the principal block of $\bfT^{\dot cF}$) then $R\Gamma_c(\bfY_\ell,\calO)$ is also perfect and therefore $[R\Gamma_c(\bfY_\ell,K)]$ is the character of a virtual projective module. This forces each generalized $\lambda$-eigenspace of $F$ to be of the following form
\begin{equation}\label{eq:cohY}
H_c^\bullet(\bfY_\ell,K)_\lambda = \chi_\mathrm{exc}[-r] \oplus \chi_\lambda[-n_\lambda]
\end{equation}
where $n_\lambda$ is the unique degree of the cohomology of $\bfX$ on which $F$ acts by $\lambda$ and $\chi_\lambda$ is the corresponding unipotent character (we assume here that $\lambda$ is one of the eigenvalues listed in Theorem \ref{thm:coxchar0}.iii). 

\smallskip

From  the shape of the Brauer tree we observe that $n_\lambda -r+1$ is exactly the distance between the node labeling $\chi_{\mathrm{exc}}$ and the node labeling $\chi_\lambda$. On the other hand, since the cohomology of $\bfY_\ell$ is torsion-free we deduce from \eqref{eq:cohY} and Proposition \ref{prop:twodegrees} that 
\begin{equation}
\label{eq:isostabY} H_c^r(\bfY_\ell,k)_{\overline{\lambda}} \simeq \Omega^{n_\lambda -r+1} H_c^{n_\lambda}(\bfY_\ell,k)_{\overline{\lambda}}
\end{equation}
in the stable category. Both cohomology groups lift to $\calO G$-lattices with characters $\chi_\text{exc}$ and $\chi_{\lambda}$, therefore they have no projective summands and the previous isomorphism holds in $kG\sfmod$. 

\smallskip

Recall that we have a distinguished triangle in $\sfD^b(kG\sfmod)$   
$$   H_c^r(\bfY_\ell,k)_{\overline{\lambda}} [-r] \longrightarrow R\Gamma_c(\bfY_\ell,k)_{\overline{\lambda}} \longrightarrow H_c^{n_\lambda}(\bfY_\ell,k)_{\overline{\lambda}} [-n_\lambda] \rightsquigarrow $$
which is determined by a map $H_c^{n_\lambda}(\bfY_\ell,k)_{\overline{\lambda}} [-n_\lambda] \longrightarrow H_c^r(\bfY_\ell,k)_{\overline{\lambda}} [-r+1]$, which is in turn determined by an element of
$\mathrm{Ext}_{kG}^{n_\lambda-r+1} (H_c^{n_\lambda}(\bfY_\ell,k)_{\overline{\lambda}}, H_c^r(\bfY_\ell,k)_{\overline{\lambda}})$. By Proposition \ref{prop:stab} and the isomorphism \eqref{eq:isostabY}, this group of extensions is isomorphic to $\mathrm{End}_{kG}(H_c^{n_\lambda}(\bfY_\ell,k)_{\overline{\lambda}})$ and hence it is one-dimensional. Therefore up to isomorphism there is a unique non-zero map  $H_c^{n_\lambda}(\bfY_\ell,k)_{\overline{\lambda}} [-n_\lambda] \longrightarrow H_c^r(\bfY_\ell,k)_{\overline{\lambda}} [-r+1]$, and the mapping cone of this map can be obtained from a truncated projective resolution of $H_c^{n_\lambda}(\bfY_\ell,k)_{\overline{\lambda}}$, which is exactly the complex given in the theorem. 
\end{proof}

Now this representative is exactly the one given by Rickard in \cite{Ric89} to construct a tilting complex for Brauer trees algebras. In particular, it satisfies \eqref{eq:disjoint} and Theorem \ref{thm:derivedeq} follows.

\end{document}